\documentclass[openany]{memo-l}

\usepackage{amsmath, amssymb, amsthm, graphicx, enumerate, todonotes, xspace, hyperref, ifthen, soul}

\newtheorem{thm}{Theorem}
\newtheorem{lem}[thm]{Lemma}
\newtheorem{prop}[thm]{Proposition}
\newtheorem{corollary}[thm]{Corollary}

\numberwithin{equation}{chapter}
\numberwithin{thm}{chapter}
\numberwithin{section}{chapter}

\providecommand\noopsort[1]{}

\newcommand{\secref}[1]{Section~\ref{#1}}

\newcommand{\firstmention}[1]{\textbf{#1}}
\newcommand{\R}{\mathbb{R}}
\newcommand{\N}{\mathbb{N}}
\newcommand{\C}{\mathbb{C}}

\newcommand{\im}{\operatorname{Im}}
\newcommand{\re}{\operatorname{Re}}
\newcommand{\mellin}[1]{\mathcal M\left[#1\right]}
\newcommand{\loglog}{\log\log}
\newcommand{\gausshyper}{{}_2F_1}
\newcommand{\onefone}{{}_1F_1}

\makeatletter
\newlength{\negph@wd}
\DeclareRobustCommand{\negphantom}[1]{%
  \ifmmode
    \mathpalette\negph@math{#1}%
  \else
    \negph@do{#1}%
  \fi
}
\newcommand{\negph@math}[2]{\negph@do{$\m@th#1#2$}}
\newcommand{\negph@do}[1]{%
  \settowidth{\negph@wd}{#1}%
  \hspace*{-\negph@wd}%
}
\makeatother

\begin{document}

\frontmatter

\title{Orthogonal polynomial expansions for the Riemann xi function}

\author{Dan Romik 
}

\address{Department of Mathematics, University of California, Davis, One Shields Ave, Davis CA 95616, USA}
\curraddr{}

\email{romik@math.ucdavis.edu}

\thanks{This material is based upon work supported by the National Science Foundation under Grant No. DMS-1800725.}

\date{}

%
%


\begin{abstract}

We study infinite series expansions for the Riemann xi function $\Xi(t)$ in three specific families of orthogonal polynomials: (1) the Hermite polynomials; (2) the symmetric Meixner-Pollaczek polynomials $P_n^{(3/4)}(x;\pi/2)$; and (3) the continuous Hahn polynomials $p_n\left(x; \frac34,\frac34,\frac34,\frac34\right)$. The first expansion was discussed in earlier work by Tur\'an, and the other two expansions are new. For each of the three expansions, we derive formulas for the coefficients, show that they appear with alternating signs, derive formulas for their asymptotic behavior, and derive additional interesting properties and relationships. We also apply some of the same techniques to prove a new asymptotic formula for the Taylor coefficients of the Riemann xi function.

Our results continue and expand the program of research initiated in the 1950s by Tur\'an, who proposed using the Hermite expansion of the Riemann xi function as a tool to gain insight into the location of the Riemann zeta zeros. We also uncover a connection between Tur\'an's ideas and the separate program of research involving the so-called De Bruijn--Newman constant. Most significantly, the phenomena associated with the new expansions in the Meixner-Pollaczek and continuous Hahn polynomial families suggest that those expansions may be even more natural tools than the Hermite expansion for approaching the Riemann hypothesis and related questions.
\end{abstract}

\maketitle

\tableofcontents

\mainmatter

\chapter{Introduction}

\label{ch:introduction}

\section{Background}

\label{sec:intro-background}

This paper concerns the study of certain infinite series expansions for the Riemann xi function $\xi(s)$. Recall that $\xi(s)$ is defined in terms of Riemann's zeta function $\zeta(s)$ by
\begin{equation}
\xi(s) = \frac12 s (s-1) \pi^{-s/2} \Gamma\left(\frac{s}{2}\right) \zeta(s) \qquad (s\in\C).
\end{equation}
$\xi(s)$ is an entire function and satisfies the functional equation
\begin{equation}
\label{eq:xi-functional-eq}
\xi(1-s) = \xi(s).
\end{equation}
It is convenient and customary to perform a change of variables, denoting
\begin{equation}
\Xi(t) = \xi\left(\frac12+i t\right) \qquad (t\in\C),
\end{equation}
a function that (in keeping with convention) will also be referred to as the Riemann xi function. The functional equation \eqref{eq:xi-functional-eq} then becomes the statement that $\Xi(t)$ is an even function.
The xi function has been a key tool in the study of the complex-analytic properties of $\zeta(s)$ and, crucially, the Riemann Hypothesis (RH). Two additional standard properties of $\Xi(t)$ are that it takes real values on the real line, and that RH can be stated as the claim that all the zeros of $\Xi(t)$ are real. \cite{titschmarsh}

\subsection{Some well-known representations of the Riemann xi function}

Much research on the zeta function has been based on studying various series and integral representations of $\zeta(s)$, $\xi(s)$ and $\Xi(t)$, in the hope that this might provide information about the location of their zeros. For example, it is natural to investigate the sequence of coefficients in the Taylor expansion
\begin{equation}
\label{eq:riemannxi-taylor}
\xi(s) = \sum_{n=0}^\infty a_{2n} \left(s-\frac12\right)^{2n}.
\end{equation}
Riemann himself derived in his seminal 1859 paper a formula for the coefficients $a_{2n}$ \cite[p.~17]{edwards}, which in our notation reads as
\begin{equation}
\label{eq:riemannxi-taylorcoeff-int}
a_{2n} = \frac{1}{2^{2n-1}(2n)!} \int_1^\infty \omega(x)x^{-3/4}(\log x)^{2n}\,dx,
\end{equation}
(where $\omega(x)$ is defined below in \eqref{eq:omegax-def}),
and which plays a small role in the theory.
The study of the numbers $a_{2n}$ remains an active area of research \cite{coffey, csordas-norfolk-varga, griffin-etal, pustylnikov1, pustylnikov2, pustylnikov3}---we will also prove a result of our own about them in \secref{sec:xi-taylor-asym}---but, disappointingly, the Taylor expansion \eqref{eq:riemannxi-taylor} has not provided much insight into the location of the zeros of $\zeta(s)$.

Another important way to represent $\xi(s)$, also considered by Riemann, is as a Mellin transform, or---which is equivalent through a standard change of variables---as a Fourier transform. Specifically, define functions $\theta(x), \omega(x), \Phi(x)$ by
\begin{align}
\label{eq:jactheta-def}
\theta(x) &
= \sum_{n=-\infty}^\infty e^{-\pi n^2 x} = 1+2\sum_{n=1}^\infty e^{-\pi n^2 x} & (x>0),
\\
\label{eq:omegax-def}
\omega(x) &= \frac12\left(2x^2 \theta''(x) + 3 x \theta'(x)\right)  = \sum_{n=1}^\infty (2\pi^2 n^4 x^2 - 3\pi n^2 x)e^{-\pi n^2 x} & (x>0),
\\
\label{eq:phix-def}
\Phi(x) & = 2 e^{x/2} \omega(e^{2x}) 
= 2\sum_{n=1}^\infty \left(2\pi^2 n^4 e^{9x/2}-3\pi n^2 e^{5x/2}\right) \exp\left(-\pi n^2 e^{2x}\right) & (x\in\R).
\end{align}
Then it is well-known that $\theta(x), \omega(x), \Phi(x)$ are positive functions, satisfy the functional equations (all equivalent to each other, as well as to \eqref{eq:xi-functional-eq})
\begin{equation}
\label{eq:theta-functional-equation}
\ \ 
\theta\left(\frac{1}{x}\right) = \sqrt{x}\, \theta(x),
\qquad 
\omega\left(\frac{1}{x}\right) = \sqrt{x}\, \omega(x),
\qquad
\Phi\left(-x\right) = \Phi(x),
\end{equation}
and that $\xi(s)$ has the Mellin transform representation
\begin{equation}
\label{eq:riemannxi-mellintrans}
\xi(s) = \int_0^\infty \omega(x)x^{s/2-1}\,dx,
\end{equation}
and the Fourier transform representation
\begin{equation}
\label{eq:riemannxi-fouriertrans}
\Xi(t) = \int_{-\infty}^\infty \Phi(x) e^{i t x}\,dx.
\end{equation}
The right-hand side of \eqref{eq:riemannxi-fouriertrans} is also frequently written in equivalent form as a cosine transform, that is, replacing the $e^{itx}$ term with $\cos(tx)$, which is valid since $\Phi(x)$ is an even function. For additional background, see \cite{edwards, titschmarsh}.

\subsection{P\'olya's attack on RH and its offshoots by De Bruijn, Newman and others}

P\'olya in the 1920s began an ambitious line of attack on RH in a series of papers \cite{polya1923, polya1926a, polya1926b, polya1927} (see also \cite[Ch.~X]{titschmarsh}) in which he investigated sufficient conditions for an entire function represented as the Fourier transform of a positive even function to have all its zeros lie on the real line. P\'olya's ideas have been quite influential and found important applications in areas such as statistical physics (see \cite{lee-yang, newman-wu}, \cite[pp.~424--426]{polya-collected}). One particular result that proved consequential is P\'olya's discovery that the factor $e^{\lambda x^2}$, where $\lambda>0$ is constant, is (to use a term apparently coined by De Bruijn \cite{debruijn}) a so-called \firstmention{universal factor}. That is to say, P\'olya's theorem states that if an entire function $G(z)$ is expressed as the Fourier transform of a function $F(x)$ of a real variable, and all the zeros of $G(z)$ are real, then, under certain assumptions of rapid decay on $F(x)$ (see \cite{debruijn} for details), the zeros of the Fourier transform of $F(x)e^{\lambda x^2}$ are also all real. This discovery spurred much follow-up work by De Bruijn \cite{debruijn}, Newman \cite{newman} and others 
\cite{csordas-norfolk-varga1988, csordas-odlyzko-etal, csordas-ruttan-varga, csordas-smith-varga, ki-kim-lee, norfolk-ruttan-varga, odlyzko, polymath15, rodgers-tao, saouter-etal, teriele}
on the subject of what came to be referred to as the \firstmention{De Bruijn-Newman constant}; the rough idea is to launch an attack on RH by generalizing the Fourier transform \eqref{eq:riemannxi-fouriertrans} through the addition of the ``universal factor'' $e^{\lambda x^2}$ inside the integral, and to study the set of real $\lambda$'s for which the resulting entire function has only real zeros. See \secref{sec:hermite-poissonpolya} where some additional details are discussed, and see \cite[Ch.~5]{broughan}, \cite{newman-wu} for accessible overviews of the subject.

\subsection{Tur\'an's approach} 

Next, we survey another attack on RH that is the closest one conceptually to our current work, proposed by P\'al Tur\'an. In a 1950 address to the Hungarian Academy of Sciences \cite{turan1952} and follow-up papers \cite{turan1954, turan1959}, Tur\'an took a novel look at the problem, starting by re-examining the idea of looking at the Taylor expansion \eqref{eq:riemannxi-taylor} and then analyzing why it fails to lead to useful insights and how one might try to improve on it. He argued that the coefficients in the Taylor expansion of an entire function provide the wrong sort of information about the zeros of the function, being in general well-suited for estimating the distance of the zeros from the origin, but poorly adapted for the purpose of telling whether the zeros lie on the real line. As a heuristic explanation, he pointed out that the level curves of the power functions $z\mapsto z^n$ are concentric circles, and argued that one must therefore look instead for series expansions of the Riemann xi function in functions whose level curves approximate straight lines running parallel to the real axis. He then argued that the Hermite polynomials
\begin{equation*}
H_n(x) = (-1)^n e^{x^2} \frac{d^n}{dx^n}\left( e^{-x^2}\right)
\end{equation*}
are such a family of functions, and proceeded to prove several results demonstrating his main thesis that the coefficients in the Fourier series expansion of a function in Hermite polynomials can in many cases provide useful information about the distance of the zeros of the function from the real line.

Tur\'an also made the important observation that the expansion
of $\Xi(t)$ in Hermite polynomials has a rather nice structure, being expressible in the form
\begin{equation}
\label{eq:hermite-expansion-intro}
\Xi(t) = \sum_{n=0}^\infty (-1)^n b_{2n} H_{2n}(t)
\end{equation}
in which, he pointed out, the coefficients $b_{2n}$ are given by the formula\footnote{Actually Tur\'an's formula in \cite{turan1959} appears to contain a small numerical error, differing from \eqref{eq:turan-coeff-formula} by a factor of $\frac{\pi}{2}$.}
\begin{equation}
\label{eq:turan-coeff-formula}
b_{2n} = \frac{1}{2^{2n} (2n)!} \int_{-\infty}^\infty x^{2n} e^{-\frac{x^2}{4}} \Phi(x)\,dx,
\end{equation}
and in particular are positive numbers.

Note that the Hermite polynomials have the symmetry $H_n(-x) = (-1)^n H_n(x)$, so, as with the case of the Taylor expansion \eqref{eq:riemannxi-taylor}, the presence of only even-indexed coefficients in \eqref{eq:hermite-expansion-intro} is a manifestation of the functional equation \eqref{eq:theta-functional-equation}, and hence serves as another indication that the expansion \eqref{eq:hermite-expansion-intro} is a somewhat natural one to consider. (Of course, the same would be true for any other family of even functions; this is obviously a weak criterion for naturalness.)

Tur\'an focused most of his attention on Hermite expansions of polynomials rather than of entire functions like $\Xi(t)$. His ideas on locating polynomial zeros using knowledge of the coefficients in their Hermite expansions appear to have been quite influential, and have inspired many subsequent fruitful investigations into the relationship between the expansion of a polynomial in Hermite polynomials and other orthogonal polynomial families, and the location of the zeros of the polynomial.
See the papers \cite{bates, bleecker-csordas, iserles-norsett1987, iserles-norsett1988, iserles-norsett1990, iserles-saff, piotrowski, schmeisser}.

By contrast, Tur\'an's specific observation about the expansion \eqref{eq:hermite-expansion-intro} of $\Xi(t)$ does not seem to have led to any meaningful follow-up work. We are not aware of any studies of the behavior of the coefficients $b_{2n}$, nor of any attempts to determine whether the Hermite polynomials are the only---or even the most natural---family of polynomials in which it is worthwhile to expand the Riemann xi function (but see \secref{sec:intro-relatedwork} for discussion of some related literature).

\section[Our new results]{Our new results: Tur\'an's program revisited and extended; expansion of $\Xi(t)$ in new orthogonal polynomial bases}

This paper can be thought of as a natural continuation of the program of research initiated by Tur\'an in his 1950 address.
One full chapter---Chapter~\ref{ch:hermite}---is dedicated to the study of the Hermite expansion \eqref{eq:hermite-expansion-intro}, answering several questions that arise quite naturally from Tur\'an's work and that have not yet been addressed in the literature. For example, in Theorem~\ref{thm:hermite-coeff-asym} we derive an asymptotic formula for the coefficients $b_{2n}$.

It is however in later chapters that it will be revealed that Tur\'an's vision of understanding the Riemann xi function by studying its expansion in Hermite polynomials was too narrow in its scope, since it turns out that there is a wealth of new and interesting results related to the notion of expanding $\Xi(t)$ in \emph{different} families of orthogonal polynomials. Two very specific orthogonal polynomial families appear to suggest themselves as being especially natural and possessing of excellent properties, and it is those that are conceptually the main focus of this paper, being the subject of Chapters~\ref{ch:fn-expansion}--\ref{ch:gn-expansion}. These families are the \firstmention{Meixner-Pollaczek polynomials} $P_n^{(\lambda)}(x;\phi)$ with the specific parameter values $\phi=\pi/2$, $\lambda=\frac34$; and the \firstmention{continuous Hahn polynomials} $p_n(x;a,b,c,d)$ with the specific parameter values $a=b=c=d=\frac34$. We denote these families of polynomials by $(f_n)_{n=0}^\infty$ and $(g_n)_{n=0}^\infty$, respectively; they are given explicitly by the hypergeometric formulas
\begin{align}
\label{eq:fn-def-intro}
f_n(x) &= \frac{(3/2)_n}{n!} 
i^n \gausshyper\left(-n, \frac34+ix; \frac32; 2 \right),
\\
\label{eq:gn-def-intro}
g_n(x) &=
i^n (n+1) \  {}_3F_2\left(-n,n+2,\frac34+ix;\frac32,\frac32; 1\right)
\end{align}
(where $(3/2)_n$ is a Pochhammer symbol),
and form systems of polynomials that are orthogonal with respect to the weight functions
$\left|\Gamma\left(\frac34+ix\right)\right|^2$
and
$\left|\Gamma\left(\frac34+ix\right)\right|^4$ on $\R$, respectively.

As our analysis will show, the expansions of $\Xi(t)$ in the polynomial families $(f_n)_{n=0}^\infty$ and $(g_n)_{n=0}^\infty$ have forms that are pleasingly similar to the Hermite expansion \eqref{eq:hermite-expansion-intro}, namely
\begin{align}
\label{eq:fn-expansion-intro}
\Xi(t) &= \sum_{n=0}^\infty (-1)^n c_{2n} f_{2n}\left(\frac{t}{2}\right),
\\
\label{eq:gn-expansion-intro}
\Xi(t) &= \sum_{n=0}^\infty (-1)^n d_{2n} g_{2n}\left(\frac{t}{2}\right),
\end{align}
where, importantly, the coefficients $c_{2n}$ and $d_{2n}$ again turn out to be positive numbers. Much more than this can be said, and in Chapters~\ref{ch:fn-expansion}--\ref{ch:gn-expansion} we undertake a comprehensive analysis of the meaning of the expansions \eqref{eq:fn-expansion-intro}--\eqref{eq:gn-expansion-intro}, the relationship between them, and the behavior of the coefficients $c_{2n}$ and $d_{2n}$. Among other results, we will prove that the coefficients satisfy the two asymptotic formulas
\begin{align}
\label{eq:intro-fncoeff-asym}
c_{2n} &\sim 16 \sqrt{2} \pi^{3/2}\, \sqrt{n} \,\exp\left(-4\sqrt{\pi n}\right),
\\
\label{eq:intro-gncoeff-asym}
 d_{2n} & \sim
\left(\frac{128\times 2^{1/3} \pi^{2/3}e^{-2\pi/3}}{\sqrt{3}}\right)
n^{4/3}  \exp\left(-3(4\pi)^{1/3} n^{2/3}\right)
\end{align}
as $n\to\infty$.
See Theorems~\ref{THM:FN-COEFF-ASYM} and~\ref{thm:gn-coeff-asym} for precise statements, including explicit rate of convergence estimates.

There are many other results. What follows is a brief summary of the main results proved in each chapter.

\begin{itemize}

\item \textbf{Chapter~\ref{ch:hermite}:}

\begin{itemize}

\item We prove a theorem (Theorem~\ref{THM:HERMITE-EXPANSION} in \secref{sec:hermite-mainresults}) on the existence of the Hermite expansion, including the fact that the expansion converges throughout the complex plane and an effective rate of convergence estimate.

\item We prove an asymptotic formula for the coefficients $b_{2n}$ (Theorem~\ref{thm:hermite-coeff-asym} in \secref{sec:hermite-asym}).

\item We prove a theorem (Theorem~\ref{thm:hermite-poisson-polya} in \secref{sec:hermite-poissonpolya}) that reveals a connection between Tur\'an's ideas on the Hermite expansion and the separate thread of research on the topic of the De Bruijn-Newman constant described in the previous section. The idea is that the so-called P\'olya-De Bruijn flow---the one-parameter family of approximations to the Riemann xi function obtained by introducing the factor $e^{\lambda x^2}$ to the Fourier transform in \eqref{eq:riemannxi-fouriertrans}---shows up in a natural way also when taking the Hermite expansion \eqref{eq:hermite-expansion-intro} and using it to separately construct a family of approximations inspired by the standard construction of Poisson kernels in the theory of orthogonal polynomials.

\end{itemize}

\bigskip
\item \textbf{Chapter~\ref{ch:fn-expansion}:}
\begin{itemize}
\item We develop the basic theory of the expansion \eqref{eq:fn-expansion-intro} of $\Xi(t)$ in the polynomials $f_n$, deriving formulas for the coefficients, showing that they alternate in sign, and proving that the expansion converges throughout the complex plane, including an effective rate of convergence estimate (Theorem~\ref{THM:FN-EXPANSION} in \secref{sec:fnexp-mainresults}).

\item We prove the asymptotic formula given in \eqref{eq:intro-fncoeff-asym} above for the coefficients $c_{2n}$ (Theorem~\ref{THM:FN-COEFF-ASYM} in \secref{sec:fnexp-mainresults}).

\item We study the Poisson flow associated with the $f_n$-expansion, by analogy with the results of Chapter~\ref{ch:hermite}, and show that this flow is the Fourier transform of a family of functions with compact support; that it evolves according to an interesting dynamical law---a differential difference equation; and that, in contrast to the Poisson flow associated with the Hermite expansion, this flow does not preserve the reality of zeros of a polynomial in either  direction of the time parameter.

\end{itemize}

\bigskip
\item \textbf{Chapter~\ref{ch:radial}:}
\begin{itemize}

\item We develop an alternative point of view that reinterprets the $f_n$-expansion \eqref{eq:fn-expansion-intro} developed in Chapter~\ref{ch:fn-expansion} as arising (through the action of the Mellin transform) from an expansion of the elementary function
\begin{equation*}
\qquad \frac{d^2}{dr^2}\left( \frac{r}{4}\coth(\pi r)\right) =
-\frac{\pi}{2}\frac{1}{\sinh^2(\pi r)} + \frac{\pi^2 r}{2} \frac{\cosh(\pi r)}{\sinh^3(\pi r)},
\end{equation*}
in an orthogonal basis of eigenfunctions of the radial Fourier transform in $\R^3$, a family of functions which can be defined in terms of the Laguerre polynomials $L_n^{1/2}(x)$.

\item We introduce and study the properties of several more special functions, including a function $\tilde{\nu}(t)$, defined as a certain integral transform of the function $\omega(x)$, that is shown to be a generating function for the coefficient sequence $c_n$, and will later play a key role in Chapter~\ref{ch:gn-expansion}.

\end{itemize}

\bigskip
\item \textbf{Chapter~\ref{ch:gn-expansion}:}
\begin{itemize}
\item We develop the basic theory of the expansion of $\Xi(t)$ in the polynomials $g_n$, deriving formulas for the coefficients, showing that they alternate in sign, and proving that the expansion converges throughout the complex plane, including an effective rate of convergence estimate. (Theorem~\ref{THM:GN-EXPANSION} in \secref{sec:gnexp-mainresults}).

\item We prove the asymptotic formula given in \eqref{eq:intro-gncoeff-asym} above for the coefficients $d_{2n}$ (Theorem~\ref{thm:gn-coeff-asym} in \secref{sec:gnexp-mainresults}).

\item We show in Sections~\ref{sec:gnexp-chebyshev}--\ref{sec:gnexp-mellin} that, analogously to the results of Chapter~\ref{ch:radial}, the $g_n$-expansion also affords a reinterpretation as arising, through the Mellin transform, from the expansion of the function $\tilde{\nu}(t)$ introduced in Chapter~\ref{ch:radial} in yet another family of orthogonal polynomials, the Chebyshev polynomials of the second kind.

\end{itemize}

\bigskip
\item \textbf{Chapter~\ref{ch:misc}:}

This chapter contains a few additional results that enhance and supplement the developments in the earlier chapters.

\medskip
\begin{itemize}

\item We apply the asymptotic analysis techniques we developed in Chapter~\ref{ch:hermite} to prove an asymptotic formula for the Taylor coefficients $a_{2n}$ of the Riemann xi function (Theorem~\ref{thm:riemannxi-taylorcoeff-asym} in \secref{sec:xi-taylor-asym}).

\item We study the function $\tilde{\omega}(x)$, a ``centered'' version of the function $\omega(x)$ that is first introduced in \secref{sec:rad-centered-recbal}. We show that $\tilde{\omega}(x)$ relates to the expansion \eqref{eq:fn-expansion-intro} in several interesting ways, and give an explicit description of its sequence of Taylor coefficient
(Theorem~\ref{thm:omega-tilde-taylor} in \secref{sec:omega-tilde-properties}) in terms of a recently studied integer sequence.
\end{itemize}

\bigskip
\item \textbf{Appendix~\ref{appendix:orthogonal}:}

This appendix contains a summary of mostly known properties of several families of orthogonal polynomials. In \secref{sec:orth-fngn-relation} we prove two new summation identities relating the two polynomial families $(f_n)_{n=0}^\infty$ and $(g_n)_{n=0}^\infty$.

\end{itemize}

\section{Previous work involving the polynomials $f_n$}

\label{sec:intro-relatedwork}

Our work on the Hermite expansion of the Riemann xi function is, as mentioned above, a natural continuation of Tur\'an's work, and also relates to the existing literature on the De Bruijn-Newman constant. By contrast, our results on the expansion of the Riemann xi function in the polynomial families $f_n$ and $g_n$ in Chapters~\ref{ch:fn-expansion}--\ref{ch:gn-expansion} do not appear to follow up on any established line of research. It seems worth mentioning however that the polynomials $f_n$ did in fact make an appearance in a few earlier works in contexts involving the Riemann zeta and xi functions. 

The earliest such work we are aware of is the paper by Bump and Ng \cite{bump-ng}, which discusses polynomials that are (up to a trivial reparametrization) the polynomials $f_n$ in connection with some Mellin transform calculations related to the zeta function. The follow-up papers by Bump et al.~\cite{bump-etal} and Kurlberg \cite{kurlberg} discuss these polynomials further, in particular interpreting their property of having only real zeros in terms of a phenomenon that the authors term the ``local Riemann hypothesis.'' The idea of using these polynomials as a basis in which to expand the Riemann xi function (or any other function) does not appear in these papers, but they seem nonetheless to be the first works that contain hints that the polynomials $f_n$ may hold some significance for analytic number theory.

In another paper \cite{kuznetsov2} (see also \cite{kuznetsov1}), Kuznetsov actually does consider an expansion in the polynomial basis $f_n(t/2)$---the same basis we use for our expansion of $\Xi(t)$---of a modified version of the Riemann xi function, namely the function $e^{-\pi t/4}\Xi(t)$, and finds formulas for the coefficients in the expansion in terms of the Taylor coefficients of an elementary function. Kuznetsov's result gives yet more clues as to the special role played in the theory of the Riemann xi function by the polynomials $f_n$. It is however unclear to us how his results relate to ours.

Finally, in a related direction, Inoue, apparently motivated by the work of Kuznetsov, studies in a recent preprint \cite{inoue} the expansion of the completed zeta function $\pi^{-s/2}\Gamma(s/2)\zeta(s)$ in the polynomials $f_n(t/2)$, and proves convergence of the expansion in the critical strip.

\section{How to read this paper}

The main part of this paper consists of Chapters~\ref{ch:hermite}--\ref{ch:gn-expansion}. These chapters are arranged in two conceptually distinct parts: Chapter~\ref{ch:hermite}, which deals with the Hermite expansion of the Riemann xi function and its connection to the De Bruijn-Newman constant, forms the first part; and Chapters~\ref{ch:fn-expansion}--\ref{ch:gn-expansion}, which develop the theory of the expansion of the Riemann xi function in the orthogonal polynomial families $(f_n)_{n=0}^\infty$ and $(g_n)_{n=0}^\infty$, form the second. The second part is largely independent of the first, so it would be practical for a reader to start reading directly from Chapter~\ref{ch:fn-expansion} and only refer back to Chapter~\ref{ch:hermite} as needed on a few occasions.

Following those chapters, we prove some additional results in Chapter~\ref{ch:misc}, and conclude in Chapter~\ref{ch:summary} with some final remarks.

The work makes heavy use of known properties of several classical, and less classical, families of orthogonal polynomials: the Chebyshev polynomials of the second kind, Hermite polynomials, Laguerre polynomials, Meixner-Pollaczek polynomials, and continuous Hahn polynomials. Appendix~\ref{appendix:orthogonal} contains reference sections summarizing the relevant properties of each of these families, and ends with a section in which we prove a new pair of identities relating the polynomial families $(f_n)_{n=0}^\infty$ and $(g_n)_{n=0}^\infty$.

We assume the reader is familiar with the basic theory of orthogonal polynomials, as described, e.g., in Chapters~2--3 of Szeg\H{o}'s classical book \cite{szego} on the subject. We also assume familiarity with standard special functions such as the \nobreak{Euler} gamma function $\Gamma(s)$ and Gauss hypergeometric function $\gausshyper(a,b;c;z)$ (see~\cite{andrews-askey-roy}), and of course with basic results and facts about the Riemann zeta function~\cite{edwards}. For background on Mellin transforms, of which we make extensive use, the reader is invited to refer to \cite{paris-kaminski}.

\section{Acknowledgements}

The author is grateful to Jim Pitman for many helpful comments and references, and for pointing out a simpler approach to proving Proposition~\ref{prop:stieltjes-from-laplace} than the one used in an earlier version of this paper.

\newpage

\chapter{The Hermite expansion of $\Xi(t)$}

\label{ch:hermite}

The goal of this chapter is to expand on Tur\'an's work in \cite{turan1952, turan1954, turan1959} on the series expansion of $\Xi(t)$ in Hermite polynomials. In \secref{sec:hermite-mainresults} we state a precise version of Tur\'an's claims about the existence of the expansion, showing that it holds on the entire complex plane and giving a quantitative rate of convergence estimate. This is proved in \secref{sec:hermite-proofmain}. In \secref{sec:hermite-asym} we prove an asymptotic formula for the coefficients $b_{2n}$ appearing in the expansion. In \secref{sec:hermite-poissonpolya} we show how the Hermite expansion leads naturally to a one-parameter family of approximations to the Riemann xi function, which we will show is (up to a trivial transformation) the same family studied in the works of De Bruijn, Newman and subsequent authors on what came to be known as the De Bruijn-Newman constant.

\section{The basic convergence result for the Hermite expansion}

\label{sec:hermite-mainresults}

Following Tur\'an \cite{turan1959}, we define numbers $(b_n)_{n=0}^\infty$ by
\begin{equation}
\label{eq:hermite-coeffs-def}
b_n = \frac{1}{2^n n!} \int_{-\infty}^\infty x^n e^{-\frac{x^2}{4}} \Phi(x)\,dx
\end{equation}
with $\Phi(x)$ defined in \eqref{eq:phix-def}.
Since $\Phi(x)$ is even and positive, we see that $b_{2n+1} = 0$ and $b_{2n} >0$ for all $n\ge0$. The following result is a more precise version of Tur\'an's remarks in \cite{turan1952} about the expansion of $\Xi(t)$ in Hermite polynomials.

\begin{thm}[Hermite expansion of $\Xi(t)$]
\label{THM:HERMITE-EXPANSION}
The Riemann xi function has the infinite series representation
\begin{equation}
\label{eq:hermite-expansion}
\Xi(t) = \sum_{n=0}^\infty (-1)^n b_{2n} H_{2n}(t),
\end{equation}
which converges uniformly on compacts for all $t\in\C$. More precisely, for any compact set $K\subset \C$ there exist constants $C_1, C_2>0$ depending on $K$ such that
\begin{equation}
\label{eq:hermite-expansion-errorbound}
\left|\Xi(t) - \sum_{n=0}^N (-1)^n b_{2n} H_{2n}(t) \right| \leq C_1 e^{-C_2 N \log N}
\end{equation}
holds for all $N\ge1$ and $t\in K$.
\end{thm}

We note for the record the unsurprising fact that the coefficients $b_{2n}$ can also be computed as Fourier coefficients of $\Xi(t)$ associated with the orthonormal basis of Hermite polynomials in the function space $L^2(\R,e^{-t^2}\,dt)$.

\begin{corollary}
\label{cor:hermite-coeff-innerproduct}
An alternative expression for the coefficients $b_{2n}$ is
\begin{equation}
\label{eq:hermite-coeff-innerproduct}
b_{2n} = \frac{(-1)^n}{\sqrt{\pi}2^{2n}(2n)!}
\int_{-\infty}^\infty \Xi(t) e^{-t^2} H_{2n}(t)\,dt.
\end{equation}
\end{corollary}

We give the easy proof of Corollary~\ref{cor:hermite-coeff-innerproduct} at the end of the next section following the proof of Theorem~\ref{THM:HERMITE-EXPANSION}.

\section{Preliminaries}

Recall the easy fact that the series \eqref{eq:omegax-def}--\eqref{eq:phix-def} defining $\omega(x)$ and $\Phi(x)$ are asymptotically dominated by their first summands as $x\to\infty$, and that this remains true if the series are summed starting at $m=2$. This leads to the following standard estimates (with the second one also relying on \eqref{eq:theta-functional-equation}), which will be used several times in this and the following chapters.

\begin{lem}
The functions $\omega(x)$ and $\Phi(x)$ satisfy the asymptotic estimates
\begin{align}
\label{eq:omegax-asym-xinfty}
\omega(x) &= 
O\left(x^2 e^{-\pi x}\right) & \textrm{as }x\to\infty,
\\
\label{eq:omegax-asym-xzero}
\omega(x) &= 
O\left(x^{-5/2} e^{-\pi/x}\right) & \textrm{as }x\to 0+,\negphantom{,}
\\
\label{eq:omegaxdiff-asym-xinfty}
\omega(x) - 
(2\pi^2 x^2 - 3\pi x)e^{-\pi x}
& = 
O\left(x^2 e^{-4\pi x}\right) & \textrm{as }x\to\infty,
\\
\label{eq:phix-asym-xinfty}
\Phi(x) & =
O\left(\exp\left(\frac{9x}{2}-\pi e^{2x}\right)\right)
& \textrm{as }x\to \infty,
\end{align}
and
\begin{align}
\label{eq:phixdiff-asym-xinfty}
\Phi(x) 
- 2\left(2\pi^2 e^{9x/2}-3\pi e^{5x/2}\right) & \exp\left(-\pi e^{2x}\right)
\\ & =
O\left(\exp\left(\frac{9x}{2}-4\pi e^{2x}\right)\right)
& \textrm{as }x\to \infty,
\nonumber
\end{align}
\end{lem}

\section{Proof of Theorem~\ref{THM:HERMITE-EXPANSION}}

\label{sec:hermite-proofmain}

We start by deriving an easy (and far from sharp, but sufficient for our purposes) bound on the rate of growth of $H_n(t)$ as a function of $n$.

\begin{lem} 
\label{lem:hermite-easybound}
The Hermite polynomials satisfy the bound
\begin{equation}
\label{eq:hermite-easybound}
|H_n(t)|\leq C \exp\left(\frac34 n \log n\right)
\end{equation}
for all $n\ge 1$, uniformly as $t$ ranges over any compact set $K\subset \C$, with $C>0$ being a constant that depends on $K$ but not on $n$.
\end{lem}

\begin{proof}
Fix the compact set $K$, and denote $M=2\max_{t\in K} |t|$. Let $N_0$ be a positive integer whose value will be fixed shortly. Let $C>0$ be a constant for which \eqref{eq:hermite-easybound} holds for all $t\in K$ and $1\le n\le N_0$. We prove by induction that the inequality holds for all $n\ge 1$, using as the induction base the case $n=N_0$. For the inductive step, let $n\ge N_0$ and assume that we have proved all cases up to the $n$th case. Then for $t\in K$ we can bound $|H_{n+1}(t)|$ using the recurrence relation \eqref{eq:hermite-recurrence} for the Hermite polynomials, which, together with the inductive hypothesis, gives that
\begin{align*}
|H_{n+1}(t)| &\leq
2 |t| \cdot|H_n(t)| + 2n |H_{n-1}(t)|
\\ & \leq
MC  \exp\left(\frac34 n \log n \right)
+ 2C n \exp\left(\frac34 (n-1)\log (n-1)\right)
\\ & =
C \exp\left(\frac34 n \log n + \log (M) \right)
+
C \exp\left(\frac34 (n-1) \log (n-1) + \log(2n) \right).
\end{align*}
We see that it is easy to complete the induction by fixing $N_0$ to be large enough as a function of $M$, specifically setting, say,
$N_0 = \max(128,\lceil(2M)^{4/3}\rceil) $. With this definition we then get (remembering the assumption $n\ge N_0$) that
\begin{align*}
|H_{n+1}(t)| & \leq
C \exp\left(\frac34 (n+1) \log n -\log 2 \right)
+
C \exp\left(\frac34 (n+1) \log (n-1) -\log 2 \right)
\\ & \leq
\frac{C}{2} \exp\left( \frac34 (n+1)\log (n+1) \right)
+ 
\frac{C}{2} \exp\left( \frac34 (n+1)\log (n+1) \right)
\\ & =
C \exp\left( \frac34 (n+1)\log (n+1) \right),
\end{align*}
which finishes the proof.
\end{proof}

Define the Lambert $W$-function to be the unique increasing function $W:[0,\infty)\to [0,\infty)$ satisfying the equation
\begin{equation*}
W(xe^x) = x.
\end{equation*}

In what follows, we will make use of the following asymptotic formula for $W(x)$ for large $x$. The result is a weaker version of eq.~(4.19) in \cite{corless-etal}.
\begin{thm}[Corless et al. \cite{corless-etal}]
The asymptotic behavior of $W(x)$ as $x\to\infty$ is given by
\begin{equation}
\label{eq:lambert-w-asym}
W(x) = \log x - \loglog x + \frac{\loglog x}{\log x}
+ O\left( \left(\frac{\loglog x}{\log x}\right)^2 \right).
\end{equation}
\end{thm}

The Lambert $W$-function and its asymptotics will be quite important for our analysis. A hint of why this is so can already be glimpsed in the proof of the following technical lemma.

\begin{lem}
\label{lem:hermite-easybound2}
For any number $B\ge 1$ there is a constant $C>0$ such that 
\begin{align}
\label{eq:hermite-easybound2}
\int_0^\infty x^n & \exp\left(-B e^x\right)\,dx
\leq 
\exp\left[
n\loglog n - \frac{n \loglog n}{\log n} - (\log B+1) \frac{n}{\log n}
+ C 
\frac{n (\log\log n)^2}{(\log n)^2}
\right]
\end{align}
for all $n\ge 3$.
\end{lem}

\begin{proof}
Denote the integral on the left-hand side of \eqref{eq:hermite-easybound2} by $I_n$. It is convenient to rewrite this integral as
\begin{equation*}
I_n = \int_0^\infty \exp(\psi_n(x))\,dx,
\end{equation*}
where we denote
\begin{equation}
\label{eq:hermite-psindef}
\psi_n(x) = n\log x - B e^x.
\end{equation}
To obtain an effective bound on this integral, it is natural to seek the point where $\psi_n(x)$ is maximized. 
Examining its derivative $\psi_n'(x) = \frac{n}{x}-B e^x$, we see that is positive for $x$ positive and close to $0$, negative for large values of $x$, and crosses zero when
\begin{equation*}
\frac{n}{x}-B e^x = 0
\quad \iff \quad
x e^x = \frac{n}{B},
\end{equation*}
an equation that has a unique solution, which we denote $x_n$, that is expressible in terms of the Lambert $W$-function, namely as
\begin{equation*}
x_n = W\left(\frac{n}{B}\right).
\end{equation*}
Thus $x_n$ is the unique global maximum point of $\psi_n(x)$. 
By \eqref{eq:lambert-w-asym}, the asymptotic behavior of $x_n$ for large $n$ (with $B$ fixed) is given by
\begin{align}
\label{eq:lambertw-value-asym}
x_n &=
\log\left(\frac{n}{B}\right) - \loglog\left(\frac{n}{B}\right) + \frac{\loglog\left(\frac{n}{B}\right)}{\log\left(\frac{n}{B}\right)} + 
O\left(
\left( \frac{\loglog\left(\frac{n}{B}\right)}{\log\left(\frac{n}{B}\right)} \right)^2
\right)
\\ \nonumber & =
(\log n - \log B) - 
\left(\loglog n
- \frac{\log B}{\log n} + 
O\left( \left(\frac{\log B}{\log n}\right)^2\right)
\right)
\\ \nonumber & \qquad\qquad +
\frac{
\left(\loglog n
- \frac{\log B}{\log n} + 
O\left( \left(\frac{\log B}{\log n}\right)^2\right)
\right)
}{\log n - \log B}
+ O\left(
\left( \frac{\loglog n}{\log n} \right)^2
\right)
\\ \nonumber & =
\log n - \loglog n - \log B + \frac{\log B}{\log n}
+ \frac{\loglog n}{\log n}
+ O\left( \left( \frac{\loglog n}{\log n}\right)^2 \right)
\quad (n\to\infty).
\end{align}
Denote $A_n = \psi_n(x_n)$, and observe that we can use the defining relation $x_n e^{x_n} = \frac{n}{B}$ for $x_n$ to rewrite $A_n$ in the form
\begin{align}
\label{eq:psin-maxvalue-lambert}
A_n &= n \log x_n - B e^{x_n}
=
n \log \left(x_n e^{x_n}\right) - n x_n - \frac{B}{x_n} \left(x_n e^{x_n}\right)
\\ \nonumber & =
n\log\left(\frac{n}{B}\right) - n x_n - \frac{n}{x_n}
=
n\left(\log n - \log B - x_n - \frac{1}{x_n} \right).
\end{align}
This form for $A_n$ makes it straightforward to derive an asymptotic formula for $A_n$: first, estimate the term $1/x_n$ separately as 
\begin{align}
\label{eq:lambertw-valuerecip-asym}
\frac{1}{x_n} &=
\frac{1}{\log n - \loglog n -\log B + O\left(\frac{\loglog n}{\log n}\right)}
\\ \nonumber &=
\frac{1}{\log n}\left(1-\frac{\loglog n}{\log n} - \frac{\log B}{\log n} + O\left(\frac{\loglog n}{\log n}\right) \right)^{-1}
=
\frac{1}{\log n} + O\left(\frac{\loglog n}{(\log n)^2}\right).
\end{align}
Then inserting \eqref{eq:lambertw-value-asym} and \eqref{eq:lambertw-valuerecip-asym} into \eqref{eq:psin-maxvalue-lambert} gives that
\begin{equation}
\label{eq:an-asym-largedevconst}
A_n =
n\loglog n - \frac{n \loglog n}{\log n} - (\log B+1)\frac{n}{\log n} +
O\left( \frac{n (\loglog n)^2}{(\log n)^2}
\right)
\quad (n\to\infty).
\end{equation}
We can now use these estimates to bound the integral $I_n$. First, split it into two parts, writing it as $I_n = I_n^{(1)} + I_n^{(2)}$, where we denote
\begin{align*}
I_n^{(1)} 
=
\int_0^{2\log n} \exp(\psi_n(x))\,dx, \qquad \quad
I_n^{(2)} 
=
\int_{2\log n}^\infty \exp(\psi_n(x))\,dx.
\end{align*}
Since $\psi_n(x)\leq A_n$ for all $x> 0$, for the first integral we have the trivial bound
\begin{align}
\label{eq:int1bound-hermiteeasy}
I_n^{(1)} &\leq 2\log n \cdot e^{A_n}
=
\exp\left[
n\loglog n - \frac{n \loglog n}{\log n} - (\log B+1) \frac{n}{\log n}
+ O\left(
\frac{n (\log\log n)^2}{(\log n)^2}
\right)
\right].
\end{align}
To bound the second integral, observe that $\psi_n(x)$ is a concave function (since its second derivative is everywhere negative), so in particular it is bounded from above by its tangent line at $x=2\log n$; that is, we have the inequality
\begin{equation*}
\psi_n(x) \leq \psi_n(2\log n) + \psi_n'(2\log n)(x-2\log n) \qquad (x>0).
\end{equation*}
The constants $\psi_n(2\log n)$, $\psi_n'(2\log n)$ in this inequality satisfy, for $n$ large enough,
\begin{align*}
\psi_n(2\log n) & =
n\log(2\log n) - B n^2 \leq -\frac{B}{2} n^2
\leq -\frac12 n^2,
\\
\psi_n'(2\log n) & =
\frac{n}{2\log n} - B n^2 \leq -\frac{B}{2} n^2 
\leq -\frac12 n^2.
\end{align*}
This then implies that, again for large $n$, we have
\begin{align}
\label{eq:int2bound-hermiteeasy}
I_n^{(2)} & \leq
\int_{2\log n}^\infty \exp\left(
\psi_n(2\log n) + \psi_n'(2\log n)(x-2\log n)
\right)\,dx
\\ \nonumber & =
\exp\left(\psi_n(2\log n)\right) \int_0^\infty
\exp\left(\psi_n'(2\log n)t\right)\,dt
\\ \nonumber  & =
\frac{1}{-\psi_n'(2\log n)} \exp\left(\psi_n(2\log n)\right)
\leq \frac{2}{n^2} e^{-n^2/2}
= O(1).
\end{align}
Combining the two bounds \eqref{eq:int1bound-hermiteeasy} and \eqref{eq:int2bound-hermiteeasy} gives the claimed bound \eqref{eq:hermite-easybound2}.
\end{proof}

We are ready to prove \eqref{eq:hermite-expansion-errorbound}. First, consider the following slightly informal calculation that essentially explains how the expansion \eqref{eq:hermite-expansion} arises out of the definition \eqref{eq:hermite-coeffs-def} of the coefficients $b_n$. Recalling the formula \eqref{eq:hermite-genfun} for the generating function for the Hermite polynomials, we have that
\begin{align}
\label{eq:hermiteexp-informal-der}
\sum_{n=0}^\infty (-1)^n b_{2n} H_{2n}(t)
& = \sum_{n=0}^\infty i^n b_n H_n(t)
=
\sum_{n=0}^\infty \frac{i^n}{2^n n!} \int_{-\infty}^\infty x^n e^{-\frac{x^2}{4}} \Phi(x)\,dx \cdot H_n(t)
\\ \nonumber &=
\int_{-\infty}^\infty \left(\sum_{n=0}^\infty \frac{i^n}{2^n n!}  x^n H_n(t)\right) e^{-\frac{x^2}{4}} \Phi(x)\,dx
\\ \nonumber &=
\int_{-\infty}^\infty \exp\left(2t\cdot \frac{ix}{2} - \left(\frac{i x}{2}\right)^2\right) e^{-\frac{x^2}{4}} \Phi(x)\,dx
=
\int_{-\infty}^\infty e^{itx} \Phi(x) \,dx = \Xi(t),
\end{align}
which is \eqref{eq:hermite-expansion}.
Note that at the heart of this calculation is the simple identity
\begin{equation}
\label{eq:fourier-kernel-hermite}
e^{itx} = e^{-\frac{x^2}{4}} \sum_{n=0}^\infty \frac{i^n x^n}{2^n n!} H_n(t),
\end{equation}
a trivial consequence of \eqref{eq:hermite-genfun}, which expands the Fourier transform integration kernel $e^{itx}$ as an infinite series in the Hermite polynomials. Thus, to get the more precise statement \eqref{eq:hermite-expansion-errorbound}, all that's left to do is to perform the same calculation a bit more carefully, using the results of Lemmas~\ref{lem:hermite-easybound} and~\ref{lem:hermite-easybound2} to get more explicit error bounds when summing this infinite series and integrating. Namely, using~\eqref{eq:fourier-kernel-hermite} we can estimate the left-hand side of \eqref{eq:hermite-expansion-errorbound} as
\begin{align}
\label{eq:hermite-expansion-proofchain}
\Bigg| \Xi(t) - & \sum_{n=0}^N (-1)^n b_{2n} H_{2n}(t) \Bigg|
=
\left| \Xi(t) - \sum_{n=0}^{2N} i^n b_n H_n(t) \right|
\\ \nonumber & =
\left| 
\int_{-\infty}^\infty
\Phi(x) \left(e^{itx}
- e^{-\frac{x^2}{4}}\sum_{n=0}^{2N}
\frac{i^n x^n}{2^n n!} H_n(t)
\right)\,dx
\right|
\\ \nonumber & 
=
\left| 
\int_{-\infty}^\infty
\Phi(x) e^{-\frac{x^2}{4}} \sum_{n=2N+1}^\infty
\frac{i^n x^n}{2^n n!} H_n(t)\,dx
\right|
\\ \nonumber & \leq
\sum_{n=2N+1}^\infty
\frac{1}{2^n n!}
\left(\int_{-\infty}^\infty
\Phi(x) e^{-\frac{x^2}{4}}
|x|^n \,dx
\right)|H_n(t)|
\\ \nonumber & =
\sum_{n=2N+1}^\infty
\frac{1}{2^{n-1} n!}
\left(\int_0^\infty
\Phi(x) e^{-\frac{x^2}{4}}
x^n \,dx
\right)|H_n(t)|
\\ \nonumber & \leq
\sum_{n=2N+1}^\infty
\frac{1}{2^{n-1} n!}
C \exp\left(\frac34 n \log n\right)
\int_0^\infty
\Phi(x) e^{-\frac{x^2}{4}}
x^n \,dx,
\end{align}
for all $t$ ranging over some fixed compact set $K \subset \C$, and where in the last step we invoked Lemma~\ref{lem:hermite-easybound}, with $C$ denoting the positive constant given by that lemma (depending on the compact set $K$).

Now, since $\Phi(x) = O\left(\exp\left(-3 e^{2x}\right)\right)$ as $x\to\infty$ by \eqref{eq:phix-asym-xinfty},
we can use Lemma~\ref{lem:hermite-easybound2} with $B=3$ to bound the integral in the last sum, and therefore conclude that this sum in \eqref{eq:hermite-expansion-proofchain} is bounded from above by
\begin{equation*}
C \sum_{n=2N+1}^\infty \frac{1}{2^n n!}
\exp \left(\frac34 n \log n\right)
\times \frac{1}{2^n}\exp\left( n \loglog n - \frac{n\loglog n}{\log n} + O\left(\frac{n}{\log n}\right)
\right).
\end{equation*}
By Stirling's formula this is $O\left(\exp\left(-\frac{1}{5} N \log N \right)\right)$, which is the bound we need. 
The proof of Theorem~\ref{THM:HERMITE-EXPANSION} is complete.
\qed

\begin{proof}[Proof of Corollary~\ref{cor:hermite-coeff-innerproduct}]

The Hermite polynomials form an orthogonal basis of the Hilbert space $L^2(\R,e^{-t^2}\,dt)$. By 
Lemma~\ref{lem:hermite-easybound2} we also get an upper bound for the coefficients $b_{2n}$ (which will be superseded by a more precise asymptotic result in the next section, but is still useful), namely the statement that 
\begin{equation*}
b_{2n} \le \frac{C}{2^{2n} (2n)!}\exp\left(2n \loglog (2n)\right)
\end{equation*}
for some constant $C>0$ and all $n\ge3$. Together with the fact that the squared $L^2$-norm of $H_n(t)$ is $\sqrt{\pi}2^n n!$ (see \eqref{eq:hermite-orthogonality}),
this implies that the infinite series on the right-hand side of \eqref{eq:hermite-expansion} converges in the sense of the function space $L^2(\R,e^{-t^2}\,dt)$ to an element of this space. Since $L^2$-convergence implies almost everywhere convergence along a subsequence, the $L^2$-limit must be equal to the pointwise limit, that is, the function $\Xi(t)$. Thus, the relation \eqref{eq:hermite-expansion} holds in the sense of $L^2$, and it follows that the coefficients in the expansion can be extracted in the standard way as inner products in the $L^2$-space, which (again because of \eqref{eq:hermite-orthogonality}) leads to the formula \eqref{eq:hermite-coeff-innerproduct}.
\end{proof}

\section{An asymptotic formula for the coefficients $b_{2n}$}

\label{sec:hermite-asym}

We now refine our analysis of the Hermite expansion by deriving an asymptotic formula for the coefficients $b_{2n}$. These asymptotics are most simply expressed in terms of the Lambert $W$-function.

\begin{thm}[Asymptotic formula for the coefficients $b_{2n}$]
\label{thm:hermite-coeff-asym}
The coefficients $b_{2n}$ satisfy the asymptotic formula
\begin{align}
\label{eq:hermite-coeff-asym}
b_{2n} & =
\left(1+O\left(\frac{\loglog n}{\log n}\right)\right)
\frac{\pi^{1/4}}{2^{4n-\frac52} (2n)!} \left(\frac{2n}{\log (2n)}\right)^{7/4}
\\ \nonumber & \qquad\qquad \times \exp\left[
2n\left(\log \left(\frac{2n}{\pi}\right) - W\left(\frac{2n}{\pi}\right) - \frac{1}{W\left(\frac{2n}{\pi}\right)} \right)
-\frac{1}{16} W\left(\frac{2n}{\pi}\right)^2\right]
\end{align}
as $n\to\infty$.
\end{thm}

The appearance of the non-elementary, implicitly-defined function $W(x)$ in the asymptotic formula \eqref{eq:hermite-coeff-asym} may make it somewhat difficult to use or gain intuition from, but with the help of the asymptotic formula \eqref{eq:lambert-w-asym} for the Lambert $W$-function, or its stronger version \cite[eq.~(4.19)]{corless-etal} mentioned above, we can extract the asymptotically dominant terms from inside the exponential to get an asymptotic formula involving more familiar functions (unfortunately, at a cost of having a much larger error term---but this seems like an unavoidable tradeoff that comes about as a result of the unusual asymptotic expansion of the Lambert $W$-function). For example, as an immediate corollary we get the following more explicit, but weaker, result.

\begin{corollary}[Asymptotic formula for the logarithm of the coefficients $b_{2n}$]
We have the relation
\begin{align}
\label{eq:hermite-coeffasym-simplified}
\log b_{2n} = 
-2n \log (2n) + 2n \loglog \left( \frac{2n}{\pi} \right)
+ O\left( n \right)
\end{align}
as $n\to\infty$.
\end{corollary}

\begin{proof}[Proof of Theorem~\ref{thm:hermite-coeff-asym}]
Define numbers $Q_n, r_n$ by
\begin{align}
\label{eq:qn-int-def}
Q_n & =
\int_0^\infty x^{2n} e^{-\frac{x^2}{4}} e^{\frac{5x}{2}}\left(e^{2x}-\frac{3}{2\pi}\right) \exp\left(-\pi e^{2x}\right)\,dx, \\
\label{eq:rn-int-def}
r_n & =
\int_0^\infty x^{2n} e^{-\frac{x^2}{4}} 
e^{\frac{5x}{2}} \sum_{m=2}^\infty \left(m^4 e^{2x}-\frac{3m^2}{2\pi}\right) \exp\left(-\pi m^2 e^{2x}\right)\,dx,
\end{align}
so that, by \eqref{eq:phix-def} and~\eqref{eq:hermite-coeffs-def}, the relation
\begin{equation}
\label{eq:bof2n-qn-rn}
b_{2n} = \frac{\pi^2}{2^{2n-3} (2n)!}( Q_n + r_n )
\end{equation}
holds. 
We will analyze the asymptotic behavior of $Q_n$ and then show that the contribution of $r_n$ is asymptotically negligible relative to that of $Q_n$.

\medskip
\textbf{Part 1: analysis of $Q_n$ using Laplace's method.}
Define a function
\begin{equation*}
f(x) = e^{-\frac{x^2}{4}} e^{\frac{5x}{2}}\left(e^{2x}-\frac{3}{2\pi}\right).
\end{equation*}
Then $Q_n$ can be rewritten in the form
\begin{equation}
\label{eq:qnint-rewritten}
Q_n = \frac{1}{2^{2n}}\int_0^\infty f(x) \exp\left( \psi_{2n}(2x)\right)\,dx,
\end{equation}
where $\psi_{2n}(x)$ is defined in \eqref{eq:hermite-psindef}, with the specific parameter value $B=\pi$.
This representation makes it possible to use Laplace's method to understand the asymptotic behavior of $Q_n$ as $n$ grows large. Proceeding as in the proof of Lemma~\ref{lem:hermite-easybound2}, we recall our observation that the function $\psi_{2n}(x)$ has a unique global maximum point at 
\begin{equation*}
x_{2n} = W\left(\frac{2n}{\pi}\right).
\end{equation*}
Now let quantities $\alpha_n$, $\beta_n$, $\gamma_n$ be defined by
\begin{align}
\label{eq:alphan-hermite-def}
\alpha_n & = \psi_{2n}(x_{2n}), \\
\label{eq:betan-hermite-def}
\beta_n & = -\psi_{2n}''(x_{2n}), \\
\label{eq:gamman-hermite-def}
\gamma_n & = f(x_{2n}/2).
\end{align}
Examining these quantities a bit more closely, note that $\alpha_n = A_{2n}$ in the notation used in the proof of Lemma~\ref{lem:hermite-easybound2} (again with the parameter $B=\pi$), so that, as in \eqref{eq:psin-maxvalue-lambert}, we have
\begin{equation}
\label{eq:alphan-simplified}
\alpha_n =
2n\left(\log (2n) - \log \pi - x_{2n} - \frac{1}{x_{2n}} \right).
\end{equation}
For $\beta_n$ we have that
\begin{equation}
\label{eq:betan-simplified-asym}
\beta_n = 
\frac{2n}{x_{2n}^2} + \pi e^{x_{2n}}
= 
\frac{2n}{x_{2n}^2} + \frac{2n}{x_{2n}}
= \left(1+O\left(\frac{1}{\log n}\right)\right) \frac{2n}{x_{2n}},
\end{equation}
and for $\gamma_n$ we can write
\begin{align}
\label{eq:gamman-asym}
\gamma_n & = 
\left(1+ O\left(
e^{-x_{2n}}\right)\right)\exp\left(-\frac{1}{16} x_{2n}^2 + \frac94 x_{2n} \right) 
=
\left(1+ O\left( \frac{\log n}{n}\right)\right)
\left( \frac{2n}{\pi x_{2n}} \right)^{9/4}
\exp\left(-\frac{1}{16} x_{2n}^2 \right).
\end{align}
With these preparations, Laplace's method in its heuristic form predicts that the integral on the right-hand side of \eqref{eq:qnint-rewritten} is given approximately for large $n$ by the expression
\begin{equation}
\label{eq:laplacemethod-heuristic}
\frac{\sqrt{\pi}}{\sqrt{-2\psi_{2n}''(x_{2n})}} f(x_{2n}/2) \exp\left(\psi_{2n}(x_n)\right) =
\frac{\sqrt{\pi}}{\sqrt{2\beta_n}} \gamma_n\exp(\alpha_n).
\end{equation}
Our goal is to establish this rigorously, with a precise rate of convergence estimate; substituting the formulas \eqref{eq:alphan-simplified}--\eqref{eq:gamman-asym} into \eqref{eq:laplacemethod-heuristic} will then give the desired formula for~$Q_n$.

It will be convenient to split up the integral defining $Q_n$ into three parts and estimate each part separately. Denote $\mu_n = n^{-2/5}$, and denote by $J_n$ the interval $\left[\frac12x_{2n}-\mu_n,\frac12x_{2n}+\mu_n\right]$. Now let
\begin{align*}
Q_n^{(1)} & =
\int_0^{\frac12 x_{2n}- \mu_n} f(x)\exp(\psi_{2n}(2x))\,dx,
\\
Q_n^{(2)} & =
\int_{J_n} f(x)\exp(\psi_{2n}(2x))\,dx,
\\
Q_n^{(3)} & =
\int_{\frac12 x_{2n}+ \mu_n}^\infty f(x)\exp(\psi_{2n}(2x))\,dx,
\end{align*}
so that $Q_n = \frac{1}{2^{2n}}\left(Q_n^{(1)}+Q_n^{(2)}+Q_n^{(3)}\right)$. Our estimates will rely on the following useful calculus observations (the first two of which were already noted in the proof of Lemma~\ref{lem:hermite-easybound2}).
\begin{enumerate}
\item The function $\psi_{2n}(x)$ is increasing on $(0,x_{2n})$ and decreasing on $(x_{2n},\infty)$.

\smallskip
\item The function $\psi_{2n}(x)$ is concave.

\smallskip
\item 
\label{item:phin-third-deriv}
$\sup_{x\in J_n} |\psi_{2n}'''(2x)| = 
O(n)$ as $n\to\infty$. 

\smallskip \noindent \textbf{Proof.}
$\psi_{2n}'''(2x) = \frac{n}{2x^3}-\pi e^{2x}$, so, for $x\in J_n$, using the fact that (by \eqref{eq:lambertw-value-asym}) for $n$ large enough we have the relation
$J_n \subseteq \left[ \frac14 \log n, \frac12 \log n\right]$,
we get that
\begin{align*}
|\psi_{2n}'''(2x)|
& \leq 
\frac{n}{2x^3} + \pi e^{2x}
\leq \frac{4^3 n}{2\log^3 n}
+ \pi e^{\log n} = O(n).
\end{align*}

\smallskip
\item 
As a consequence of the last observation, the Taylor expansion of $\psi_{2n}(2x)$ around $x=\frac12 x_{2n}$ in the interval $J_n$ has the form
\begin{equation}
\label{eq:psi2n-taylor-jn}
\psi_{2n}(2x) =
\alpha_n - 2 \beta_n \left(x-\frac{x_{2n}}{2} \right)^2 + O\left(n \left|x-\frac{x_{2n}}{2}\right|^3
\right) \qquad (x\in J_n),
\end{equation}
where the constant implicit in the big-$O$ term does not depend on $n$ or $x$.

\smallskip
\item $\sup_{x\in J_n} \left| \frac{f(x)}{f(x_{2n}/2)}
-1\right| = O\left(\frac{\log n}{n^{2/5}}\right)$.

\smallskip \noindent \textbf{Proof.}
Noting that $x_n\to\infty$ as $n\to\infty$, so that
$f(x_{2n}/2) \geq \frac12 e^{-\frac{x_{2n}^2}{16}+\frac94 x_{2n}}$ if $n$ is large, we have that
\begin{align}
\label{eq:fofx-messy}
\left| \frac{f(x)}{f(x_{2n}/2)}-1\right|
& =
\left| 
\frac{e^{-\frac{x^2}{4}+\frac92 x}-\frac{3}{2\pi}e^{-\frac{x^2}{4}+\frac52 x}}{e^{-\frac{x_{2n}^2}{16}+\frac94 x_{2n}}-\frac{3}{2\pi}e^{-\frac{x_{2n}^2}{16}+\frac54 x_{2n}}}
-1
\right|
\\ \nonumber
& =
\left| 
\frac{
\left(
e^{-\frac{x^2}{4}+\frac92 x}-\frac{3}{2\pi}e^{-\frac{x^2}{4}+\frac52 x}
\right)
-\left(
e^{-\frac{x_{2n}^2}{16}+\frac94 x_{2n}}-\frac{3}{2\pi}e^{-\frac{x_{2n}^2}{16}+\frac54 x_{2n}}
\right)
}{e^{-\frac{x_{2n}^2}{16}+\frac94 x_{2n}}-\frac{3}{2\pi}e^{-\frac{x_{2n}^2}{16}+\frac54 x_{2n}}}
\right|
\\ \nonumber & \leq
\left|
\frac{
e^{-\frac{x^2}{4}+\frac92 x} -
e^{-\frac{x_{2n}^2}{16}+\frac94 x_{2n}}
}{
\frac12 e^{-\frac{x_{2n}^2}{16}+\frac94 x_{2n}}
}
\right|
+
\frac{3}{2\pi}
\left|
\frac{
e^{-\frac{x^2}{4}+\frac52 x} -
e^{-\frac{x_{2n}^2}{16}+\frac54 x_{2n}}
}{
\frac12 e^{-\frac{x_{2n}^2}{16}+\frac94 x_{2n}}
}
\right|
\\ \nonumber & =
2 \left|
e^{(-\frac{x^2}{4}+\frac92 x)-(-\frac{x_{2n}^2}{16}+\frac94 x_{2n})}
-1
\right|
+ \frac{6}{2\pi}
e^{-x_{2n}} \left|
e^{(-\frac{x^2}{4}+\frac52 x)-(-\frac{x_{2n}^2}{16}+\frac54 x_{2n})}
-1
\right|.
\end{align}
Now observe that $e^{-x_{2n}} = O(1)$ and that for $x\in J_n$ we have that
\begin{align*}
\Bigg|
\left(-\frac{x^2}{4}+\frac92 x\right)- & \left(-\frac{x_{2n}^2}{16}+\frac94 x_{2n}\right)
\Bigg|
\leq \frac92 \left|x-\frac{x_{2n}}{2}\right| + \frac14 \left|x-\frac{x_{2n}}{2}\right|\cdot
\left|x+\frac{x_{2n}}{2}\right|
 = O\left(\frac{\log n}{n^{2/5}}\right)
\end{align*}
(with a uniform constant implicit in the big-$O$ term),
and similarly
\begin{equation*}
\left|
\left(-\frac{x^2}{4}+\frac52 x\right)-\left(-\frac{x_{2n}^2}{16}+\frac54 x_{2n}\right)
\right| = O\left(\frac{\log n}{n^{2/5}}\right),
\end{equation*}
so that \eqref{eq:fofx-messy} implies the claimed bound.

\end{enumerate}

We are ready to evaluate the integrals $Q_n^{(1)}$, $Q_n^{(2)}$, $Q_n^{(3)}$, starting with the middle integral $Q_n^{(2)}$, which is the one that is the most significant asymptotically. The standard idea is that the exponential term $\exp(\psi_{2n}(2x))$ can be approximated by a Gaussian centered around the point $\frac12 x_{2n}$. This follows from the Taylor expansion \eqref{eq:psi2n-taylor-jn}.
Indeed, making the change of variables
$u=\sqrt{\beta_n}\left(x-\frac12x_{2n}\right)$ in the integral, we have that
\begin{align}
\label{eq:qn2-longcalc}
Q_n^{(2)} & =
\int_{J_n} f(x) \exp(\psi_{2n}(2x))\,dx 
\\ \nonumber & =
\int_{-\mu_n\sqrt{\beta_n}}^{\mu_n\sqrt{\beta_n}}
f\left(\frac{x_{2n}}{2} + \frac{u}{\sqrt{\beta_n}}\right)
\exp\left(\psi_{2n}\left(
x_{2n} + \frac{2u}{\sqrt{\beta_n}}
\right)
\right)\, \frac{du}{\sqrt{\beta_n}}
\\ \nonumber & =
\frac{1}{\sqrt{\beta_n}} 
\int_{-\mu_n\sqrt{\beta_n}}^{\mu_n\sqrt{\beta_n}}
\left(1+O\left(\frac{\log n}{n^{2/5}}\right)\right)f\left(\frac{x_{2n}}{2}\right)
\exp\left[
\alpha_n - 2 u^2 + O\left(n \frac{|u|^3}{\beta_n^{3/2}} \right)
\right] \,du
\\ \nonumber & =
\left(1+O\left(\frac{\log n}{n^{2/5}}\right)\right) 
\left(1+O\left(\frac{1}{n^{1/5}}\right)\right)
\frac{1}{\sqrt{\beta_n}}
e^{\alpha_n} \gamma_n
\int_{-\mu_n\sqrt{\beta_n}}^{\mu_n\sqrt{\beta_n}} e^{-2 u^2}\,du
\\ \nonumber & =
\left(1+O\left(\frac{1}{n^{1/5}}\right)\right) \frac{1}{\sqrt{\beta_n}}
\gamma_n e^{\alpha_n} \left[\sqrt{\frac{\pi}{2}} -
O\left(
\frac{1}{\mu_n \sqrt{\beta_n}} \exp\left(-2 \mu_n^2 \beta_n\right)
\right)
 \right]
\\ \nonumber & =
\left(1+O\left(\frac{1}{n^{1/5}}\right)\right) \frac{\sqrt{\pi}}{\sqrt{2\beta_n}}
\gamma_n e^{\alpha_n}, 
\end{align}
where in the penultimate step we used the standard inequality 
\begin{equation}
\label{eq:gaussian-tail-bound}
\int_t^\infty e^{-u^2/2}\,du \leq \frac1t e^{-t^2/2}, \qquad (t>0).
\end{equation}

Next, to estimate $Q_n^{(1)}$, we use the fact that $\psi_{2n}(2x)$ is increasing on the interval of integration, bounding the integral by the length of the integration interval multiplied by an upper bound for $f(x)$ and the value of $\exp(\psi_{2n}(2x))$ at the rightmost end of the interval. Note that $f(x)$ is bounded from above by the numerical constant
\begin{equation*}
K_0 := \left(\sup_{x\ge 0} e^{-\frac{x^2}{4}+\frac92 x}
+ \frac{3}{2\pi}
\sup_{x\ge 0} e^{-\frac{x^2}{4}+\frac52 x}\right)
= 
e^{81/4} + \frac{3}{2\pi} e^{25/4}.
\end{equation*}
Thus, using \eqref{eq:betan-simplified-asym}, \eqref{eq:gamman-asym} and \eqref{eq:psi2n-taylor-jn}, we get that
\begin{align}
\label{eq:qn1-longcalc}
Q_n^{(1)} & 
\leq
K_0 \left(\frac{x_{2n}}{2}-\mu_n\right)
\times \exp\left(\psi_{2n}\left(x_{2n}-2\mu_n\right)\right)
\leq
\frac{K_0}{2} x_{2n} \exp\left(\alpha_n - 2 \beta_n \mu_n^2 + O(n \mu_n^3) \right)
\\ \nonumber & = 
O(\log n) e^{\alpha_n} O\left(\exp\left(-n^{1/10}\right) \right)
\left(1+O\left(\frac{1}{n^{1/5}}\right)\right)
=
O\left(
e^{-\frac12 n^{1/10}}
\frac{\sqrt{\pi}}{\sqrt{2 \beta_n}} \gamma_n e^{\alpha_n}
\right).
\end{align}

Next, to estimate $Q_n^{(3)}$, note that, as in the proof of Lemma~\ref{lem:hermite-easybound2}, by the concavity of $\psi_{2n}(2x)$ the graph of $\psi_{2n}(2x)$ is bounded from above by the tangent line to the graph at $x=\frac{x_{2n}}{2}+\mu_n$. In other words, we have the inequality
\begin{equation*}
\psi_{2n}(2x) \le \psi_{2n}\left(x_{2n}+2\mu_n\right) + 2\psi_{2n}'\left(x_{2n}+2\mu_n\right)\left(x-\frac{x_{2n}}{2}-\mu_n\right) \qquad (x> 0).
\end{equation*}
Moreover, it is useful to note that the derivative value $\psi_{2n}'(x_{2n}+2\mu _n)$ satisfies
\begin{align*}
\psi_{2n}'(x_{2n}+2\mu _n)
& = 
\frac{2n}{x_{2n}+2\mu_n} - \pi e^{x_{2n}+2\mu_n} 
= 
\frac{2n}{x_{2n}+2\mu_n} - e^{2\mu_n}\frac{2n}{x_{2n}}
\leq 
\frac{2n}{x_{2n}}\left(1 - e^{2\mu_n}\right)
\leq -\frac{4n \mu_n}{x_{2n}},
\end{align*}
so in particular $\psi_{2n}'(x_{2n}+2\mu _n) \leq -1$ if $n$ is large enough. These observations imply that as $n\to\infty$, $Q_n^{(3)}$ satisfies the bound
\begin{align}
\label{eq:qn3-longcalc}
Q_n^{(3)} 
&\leq
K_0\int_{\frac12 x_{2n}+\mu_n}^\infty 
\exp\left(\psi_{2n}\left(x_{2n}+2\mu_n\right)
+ 2\psi_{2n}'\left(x_{2n}+2\mu_n\right) \left(x-\frac{x_{2n}}{2}-\mu_n\right)\right)\,dx
\\ \nonumber & \leq
K_0 \int_{\frac12x_{2n}+\mu_n}^\infty \exp\left(\alpha_n - 2 \beta_n \mu_n^2 + O(n \mu_n^3)\right) 
\exp\left(
-\left(x-\frac{x_{2n}}{2}-\mu_n\right)
\right)\,dx
\\ \nonumber & =
K_0 \left(1+O\left(\frac{1}{n^{1/5}}\right)\right) e^{\alpha_n}
O\left(\exp\left(-n^{1/10}\right)\right) \int_0^\infty
\exp\left(-u\right)\,du
= O\left(e^{-n^{1/10}} \frac{\sqrt{\pi}}{\sqrt{2 \beta_n}} \gamma_n e^{\alpha_n}\right).
\end{align}
Combining \eqref{eq:qn2-longcalc}, \eqref{eq:qn1-longcalc} and \eqref{eq:qn3-longcalc}, we have finally that
\begin{equation}
\label{eq:qn-asym-preformula}
Q_n =
\left(1+O\left(\frac{1}{n^{1/5}}\right)\right)
\frac{\sqrt{\pi}}{2^{2n}\sqrt{2\beta_n}} \gamma_n e^{\alpha_n}.
\end{equation}
Using \eqref{eq:alphan-simplified}--\eqref{eq:gamman-asym} we now get the asymptotic formula
\begin{align}
\label{eq:qn-asym-formula}
Q_n &= 
\left(1+O\left(\frac{1}{\log n}\right)\right)
\frac{1}{2^{2n+\frac12}} \left(\frac{2n}{\pi x_{2n}}\right)^{7/4}
\exp\left[
2n\left(\log (2n) - \log \pi - x_{2n} - \frac{1}{x_{2n}} \right)
-\frac{1}{16} x_{2n}^2\right]
\end{align}
that holds as $n\to\infty$.
In particular, for the purpose of comparing $Q_n$ to $r_n$, it is useful to note that the exponential factors $\frac{1}{2^{2n}}$ and $e^{\alpha_n}$ are asymptotically the most significant ones in~\eqref{eq:qn-asym-preformula}. More precisely, recalling 
\eqref{eq:lambertw-value-asym}, \eqref{eq:an-asym-largedevconst}, \eqref{eq:betan-simplified-asym} and \eqref{eq:gamman-asym}, we get that
\begin{align}
\label{eq:qofn-exponential-approx}
Q_n &=
\frac{1}{2^{2n}}\exp\left(\alpha_n + O((\log n)^2)\right) 
\\ \nonumber &=
\frac{1}{2^{2n}}\exp\Bigg[
2n\loglog (2n) - \frac{2n \loglog (2n)}{\log (2n)} 
- (\log \pi+1)\frac{2n}{\log (2n)}
+O\left( \frac{n (\loglog n)^2}{(\log n)^2}
\right)
\Bigg]
\end{align}
as $n\to\infty$.

\medskip
\textbf{Part 2: estimating $r_n$.}
We proceed with asymptotically bounding $r_n$.
Observe that, by \eqref{eq:phixdiff-asym-xinfty}, $r_n$ satisfies
\begin{equation*}
|r_n| \leq
C_1 \int_0^\infty x^{2n} \exp\left(-3\pi e^{2x}\right)\,dx
=
\frac{C_1}{2^{2n+1}} \int_0^\infty u^{2n} \exp\left(-3\pi e^u\right)\,du
\end{equation*}
for some constant $C_1>0$.
We can therefore once again apply Lemma~\ref{lem:hermite-easybound2}, this time with the parameter $B=3\pi$, to get that, for all $n\ge3$ and some constant $C_2>0$,
\begin{align*}
|r_n| & \leq
\frac{1}{2^{2n}} \exp\Bigg[
2n\loglog (2n) - \frac{2n \loglog (2n)}{\log (2n)} 
- (\log (3\pi)+1) \frac{2n}{\log (2n)}
+ C_2 
\frac{2n (\log\log (2n))^2}{(\log (2n))^2}
\Bigg].
\end{align*}
Comparing this to \eqref{eq:qofn-exponential-approx}, we see that for large $n$ the relation
\begin{equation}
\label{eq:rn-asym-bound}
|r_n| \leq \exp\left(-\frac12 \log(3) \frac{2n}{\log (2n)}\right)Q_n 
\end{equation}
holds. Thus $r_n$ is indeed asymptotically negligible compared to $Q_n$. 

\medskip
\textbf{Finishing the proof.}
Combining \eqref{eq:bof2n-qn-rn}, \eqref{eq:qn-asym-formula} and \eqref{eq:rn-asym-bound} we arrive finally at the desired formula for $b_{2n}$:
\begin{align*}
b_{2n} & =
\frac{\pi^2}{2^{2n-3} (2n)!} 
\left(
1+O\left(\exp\left(-\log(3) \frac{n}{\log (2n)}\right)\right)
\right) Q_n 
\\ & =
\left(1+O\left(\frac{1}{\log n}\right)\right)
\frac{\pi^2}{2^{2n-3} (2n)!} \times
\frac{1}{2^{2n+\frac12}} \left(\frac{2n}{\pi x_{2n}}\right)^{7/4}
\exp\left[
2n\left(\log (2n) - \log \pi - x_{2n} - \frac{1}{x_{2n}} \right)
-\frac{1}{16} x_{2n}^2\right].
\end{align*}
Since $x_{2n} = W\left(\frac{2n}{\pi}\right) = \left(1+O\left(\frac{\loglog n}{\log n}\right)\right)\log(2n)$, this gives \eqref{eq:hermite-coeff-asym} and completes the proof of Theorem~\ref{thm:hermite-coeff-asym}.
\end{proof}

\section[The Poisson flow and the P\'olya-De Bruijn flow]{The Poisson flow, P\'olya-De Bruijn flow and the De Bruijn-Newman constant}

\label{sec:hermite-poissonpolya}

One recurring theme in the study of the Riemann hypothesis is the idea that in order to understand the zeros of the Riemann xi (or zeta) function, one might start by looking at suitable approximations to it that have a simpler structure---for example, being polynomials instead of entire or meromorphic functions---and trying to understand the location of the zeros of those approximations first. The hope is that there exists some good approximation that would have the feature that the zeros of the approximating functions can be understood, and, in an ideal scenario, shown to all lie on the real line (or on the critical line $\re(s)=1/2$, depending on the coordinate system used). In the setting of a discrete sequence of approximations, this approach has been applied for example  to the partial sums of the Taylor series of $\Xi(t)$ \cite{jenkins-mclaughlin} and to the partial sums of the Dirichlet series of $\zeta(s)$ \cite{haselgrove, spira, turan1948}. While those attacks can involve the use of some rather ingenious and sophisticated tools, they have not resulted in any easily quantifiable progress on the original question of RH. 

Instead of looking at a discrete sequence of approximations, certain other contexts naturally suggest instead a family of approximations indexed by a continuous parameter. We refer to such a family informally, especially in the case when the family satisfies a partial differential equation or some other sort of dynamical evolution law (all the approximation families we consider will be of this type) as a \firstmention{flow}.

One very natural and well-studied example of such a flow is the family of functions
\begin{equation}
\label{eq:polya-flow-def}
\Xi_\lambda(t) = \int_{-\infty}^\infty e^{\lambda x^2} \Phi(x) e^{itx}\,dx \qquad (\lambda \in \R).
\end{equation}
For $\lambda=0$, we have $\Xi_0\equiv\Xi$, so the family $\Xi_\lambda$ is a flow passing through the Riemann xi function at $\lambda=0$. We refer to it as the \firstmention{P\'olya-De Bruijn flow} associated with the xi function; this term seems appropriate in view of P\'olya's discovery of universal factors described in \secref{sec:intro-background} and its extension by De Bruijn. Specifically, P\'olya's result described in the introduction to the effect that $e^{\lambda x^2}$ is a universal factor implies that the P\'olya-De Bruijn flow preserves \firstmention{hyperbolicity} (the property of an entire function of having no non-real zeros) in the positive direction of the ``time'' parameter $\lambda$: that is, if $\Xi_\lambda$ is hyperbolic for some specific value of $\lambda$, then so is $\Xi_\mu$ for any $\mu>\lambda$, and in particular, if it could be shown that $\Xi_\lambda$ is hyperbolic for some \emph{negative} value $\lambda<0$, the Riemann hypothesis would follow. Moreover, showing that $\Xi_\lambda$ is hyperbolic for \emph{positive} values of $\lambda$ (which by the same logic ought to be both more likely to be true, and easier to prove if true) may still be beneficial, since if for instance it could somehow be shown that $\Xi_\lambda$ were hyperbolic for \emph{all} $\lambda>0$, once again RH would follow by a straightforward approximation argument.

De Bruijn extended P\'olya's work in an important way when he showed that in fact $\Xi_\lambda$ is hyperbolic for all $\lambda \geq 1/8$. His result was later strengthened slightly by Ki, Kim and Lee \cite{ki-kim-lee}, who showed that the same would be true for $\lambda\geq 1/8-\epsilon$ for some (fixed, but non-explicit) $\epsilon>0$. In the negative direction, Newman \cite{newman} proved that $\Xi_\lambda$ is \emph{not} hyperbolic if $\lambda$ is a negative number of sufficiently large magnitude, and conjectured that the same statement holds true for \emph{all} $\lambda < 0$---this is usually formulated as the statement that the \firstmention{De Bruijn-Newman constant}, denoted by $\Lambda$ and defined as four times the greatest lower bound of the set of $\lambda$'s for which $\Xi_\lambda$ is hyperbolic,\footnote{The multiplication by four is a quirk associated with the different notational conventions used by different authors in the literature on the subject. See p.~63 and Table 5.2 on p.~68 of \cite{broughan} for further discussion and a comparison of the different conventions.} is nonnegative. An explicit numerical lower bound of $-50$ for the De Bruijn-Newman constant was later established by Csordas, Norfolk and Varga \cite{csordas-norfolk-varga1988}. The lower bound was pushed upwards further in a succession of papers \cite{csordas-odlyzko-etal, csordas-ruttan-varga, csordas-smith-varga, norfolk-ruttan-varga, odlyzko, saouter-etal, teriele}, with the established bounds gradually growing extremely close to $0$ on the negative side. Most recently Rodgers and Tao \cite{rodgers-tao} succeeded in proving Newman's conjecture that $\Lambda\ge 0$, and recent work by the Polymath15 project \cite{polymath15} strengthened the result of Ki, Kim and Lee mentioned above by proving the sharper upper bound $\Lambda \le 0.22$.

We now come to a key idea that relates the above discussion to our theme of expansions of the Riemann xi function in families of orthogonal polynomials, and the Hermite polynomials in particular. Specifically, it is the idea that any series expansion of the Riemann xi function in a system of orthogonal polynomials comes equipped with its own flow based on the standard construction of the so-called \firstmention{Poisson kernel} in the theory of orthogonal polynomials. We call this the \firstmention{Poisson flow}.

To define the Poisson flow, recall that the Poisson kernel for a family $\phi=(\phi_n)_{n=0}^\infty$ of polynomials that are orthogonal with respect to a weight function $w(x)$ is defined by
\begin{equation}
\label{eq:poisson-ker-def}
p_r^\phi(x,y) = \sum_{n=0}^\infty \frac{r^n}{\int_\R \phi_n(u)^2 w(u)\,du} \phi_n(x)\phi_n(y) \qquad (|r|<1).
\end{equation}
Its essential feature is the equation
\begin{equation*}
\int_{-\infty}^\infty p_r^\phi(x,y) \phi_n(y) w(y)\,dy = 
r^n \phi_n(x),
\end{equation*}
which is trivial to verify by evaluating the integral termwise.
That is, the associated integral kernel operator $\Pi_r^\phi:f\mapsto 
\int_{\R} p_r^\phi(x,y) f(y) w(y)\,dy$
acting on $L^2(\R, w(x)\,dx)$
sends a function with Fourier expansion $f(x) = \sum_n \gamma_n \phi_n$ to $\sum_n \gamma_n r^n \phi_n$, with the $n$th ``harmonic'' in the expansion being attenuated by a factor $r^n$. Note that one can also consider the limiting case $r\to 1$, in which case the definition \eqref{eq:poisson-ker-def} of the Poisson kernel $p_r^\phi$ no longer makes sense, but the operator $\Pi_1^\phi$ can be defined simply as the identity operator, which is clearly the limit of the $\Pi_r$'s (and $p_1^\phi(x,y)$ can be thought of as the distribution $\delta(x-y)$).

We can now define the \firstmention{Poisson flow associated with the Riemann xi function for the orthogonal polynomial sequence $\phi=(\phi_n)$} to be the family of functions
\begin{equation}
\label{eq:poisson-flow-def}
X_r^{\phi}(t) =
\Pi_r^\phi(\Xi)(t)
=
\begin{cases}
\int_{-\infty}^\infty p_r(t,\tau)\Xi(\tau)w(\tau)\,d\tau & \textrm{if }0<r<1,
\\
\Xi(t) & \textrm{if }r=1
\end{cases}
\qquad (0 < r \le 1).
\end{equation}
Alternatively, if $\Xi(t)$ is expressed in terms of its Fourier series expansion $\Xi(t) = \sum_{n=0}^\infty \gamma_n \phi_n(t)$ in the orthogonal polynomial family $(\phi_n)_{n=0}^\infty$  (in the sense of the function space $L^2(\R,w(x)\,dx)$), we can write the Poisson flow equivalently as
\begin{equation}
\label{eq:poisson-flow-expansion}
X_r^\phi(t) = \sum_{n=0}^\infty r^n \gamma_n \phi_n(t).
\end{equation}

Denote the family of Hermite polynomials by $\mathcal{H}=(H_n)_{n=0}^\infty$, so that $p_r^{\mathcal{H}}(x,y)$ and $\Pi_r^{\mathcal{H}}$ now denote the Poisson kernel and integral operator associated with the Hermite polynomials, respectively, and $X_r^{\mathcal{H}}(t)$ denotes the corresponding flow associated with the Riemann xi function. Our main result for this section, relating the different concepts we introduced above, is the following.

\begin{thm}[Connection between the P\'olya-De Bruijn and Poisson flows]
\label{thm:hermite-poisson-polya}
The Poisson flow for the Hermite polynomials is related to the P\'olya-De Bruijn flow \eqref{eq:polya-flow-def} via
\begin{equation}
\label{eq:hermite-poisson-polya}
X_r^{\mathcal{H}}(t) =
\Xi_{(r^2-1)/4}(rt) \qquad (0<r\le 1).
\end{equation}
\end{thm}

\begin{proof}
This is a straightforward calculation that generalizes \eqref{eq:hermiteexp-informal-der} by weighting each of the summands in the expansion by the factor $r^n$. Once again using the generating function formula \eqref{eq:hermite-genfun}, we have that
\begin{align}
\label{eq:hermite-poissonpolya-calc}
X_r^{\mathcal{H}}(t) & =
\sum_{n=0}^\infty i^n r^n b_n H_n(t)
=
\sum_{n=0}^\infty \frac{i^n r^n}{2^n n!}
\int_{-\infty}^\infty x^n e^{-\frac{x^2}{4}} \Phi(x)\,dx \cdot
H_n(t)
\\ \nonumber &=
\int_{-\infty}^\infty \left( \sum_{n=0}^\infty \frac{1}{n!} \left(\frac{i r x}{2} \right)^n H_n(t)\right) e^{-\frac{x^2}{4}} \Phi(x)\,dx
=
\int_{-\infty}^\infty 
\exp\left( 2t\cdot \frac{ir x}{2}-\left(\frac{irx}{2}\right)^2\right)
e^{-\frac{x^2}{4}}\Phi(x)\,dx
\\ \nonumber &=
\int_{-\infty}^\infty \Phi(x) \exp\left((r^2-1)\frac{x^2}{4}\right)e^{ir t x}\,dx
= \Xi_{(r^2-1)/4}(r t),
\end{align}
as claimed.
\end{proof}

Theorem~\ref{thm:hermite-poisson-polya} ties together in an interesting way the different threads of research into RH begun with the work of P\'olya on universal factors (and continued with the extensive subsequent investigations into the De Bruijn-Newman constant by De Bruijn, Newman and others) on the one hand, and Tur\'an's ideas on the Hermite expansion on the other hand. Incidentally, hints of this connection seem to have already been noted in a less explicit way in the literature; see in particular \cite[Section~3]{bleecker-csordas}.

One key point to take away from this discussion is that the Poisson flow appears to be a natural device with which to try to approximate the Riemann xi function. And while Theorem~\ref{thm:hermite-poisson-polya} shows that the Poisson flow associated with the Hermite polynomials is equivalent to an already well-studied construction, the point is that Poisson flows are a method of approximation that allows us a considerable freedom in choosing the system of orthogonal polynomials to use, and it is conceivable that this might lead to new families of approximations with useful properties. Indeed, in Chapter~\ref{ch:fn-expansion}, when we consider the expansion of $\Xi(t)$ in the family of Meixner-Pollaczek orthogonal polynomials $f_n$, we will revisit the Poisson flow in the context of this new expansion and show that it has some quite natural and interesting properties in that setting.

As a final remark, we recall that one of several notable features of the P\'olya-De Bruijn flow, first pointed out in \cite{csordas-smith-varga}, is that it satisfies the time-reversed heat equation
\begin{equation}
\label{eq:timerev-heat-equation}
\frac{\partial \Xi_\lambda(t)}{\partial \lambda}
=
- \frac{\partial^2 \Xi_\lambda(t)}{\partial t^2},
\end{equation}
a fact that follows immediately from the representation \eqref{eq:polya-flow-def} by differentiating under the integral sign, and which played a useful role in the study of the De Bruijn-Newman constant (see \cite[Sec.~5.5]{broughan}, \cite{csordas-smith-varga}, \cite{rodgers-tao}).
It is of some interest to note that the same result can also be obtained by using the relation \eqref{eq:hermite-poisson-polya} interpreting the P\'olya-De Bruijn flow as a reparametrized version of the Poisson flow, together with basic properties of the Hermite polynomials. To see this, start by inverting \eqref{eq:hermite-poisson-polya} to express $\Xi_\lambda(t)$ in terms of the Poisson flow as
\begin{equation*}
\Xi_\lambda(t) =
X_{\sqrt{1+4\lambda}}^\mathcal{H}\left(\frac{t}{\sqrt{1+4\lambda}}\right).
\end{equation*}
Now expanding the Poisson flow as in \eqref{eq:poisson-flow-expansion}, we differentiate and then use the classical ordinary differential equation \eqref{eq:hermite-ode} satisfied by the Hermite polynomials, to get that
\begin{align*}
\frac{\partial \Xi_\lambda(t)}{\partial \lambda}
& =
\frac{\partial}{\partial \lambda}
\left(
X_{\sqrt{1+4\lambda}}^\mathcal{H}\left(\frac{t}{\sqrt{1+4\lambda}}\right) \right)
\\
&=
\frac{\partial}{\partial \lambda}
\left(
\sum_{n=0}^\infty i^n b_n (1+4\lambda)^{\frac{n}{2}} H_n\left(\frac{t}{\sqrt{1+4\lambda}}\right)
\right)
\\ & =
\sum_{n=0}^\infty i^n b_n
\Bigg[
4\frac{n}{2} (1+4\lambda)^{\frac{n}{2}-1} 
H_n\left(\frac{t}{\sqrt{1+4\lambda}}\right)
- (1+4\lambda)^{\frac{n}{2}} 
\frac{4t}{2(1+4\lambda)^{3/2}}
H_n'\left(\frac{t}{\sqrt{1+4\lambda}}\right)
\Bigg]
\\ & =
\sum_{n=0}^\infty i^n b_n (1+4\lambda)^{\frac{n}{2}}
\Bigg[
\frac{1}{1+4\lambda}
\left(
- H_n''\left(\frac{t}{\sqrt{1+4\lambda}}\right)
+ \frac{2t}{\sqrt{1+4\lambda}} 
H_n'\left(\frac{t}{\sqrt{1+4\lambda}}\right)
\right)
\\ & \hspace{240pt}
-\frac{2t}{(1+4\lambda)^{3/2}} 
H_n'\left(\frac{t}{\sqrt{1+4\lambda}}\right)
\Bigg]
\\ & =
-\sum_{n=0}^\infty i^n b_n (1+4\lambda)^{\frac{n}{2}}
\frac{\partial^2}{\partial t^2}
\left( H_n\left(\frac{t}{\sqrt{1+4\lambda}}\right) \right)
= 
-\frac{\partial^2}{\partial t^2}
\left( X_{\sqrt{1+4\lambda}}^\mathcal{H}\left(\frac{t}{\sqrt{1+4\lambda}}\right) \right)
 =
-\frac{\partial^2 \Xi_\lambda(t)}{\partial t^2},
\end{align*}
recovering \eqref{eq:timerev-heat-equation} as expected. (Incidentally, at the heart of this calculation is the simple observation that each of the two-variable functions $h_n(\tau,x) = \tau^{n/2} H_n\left(-\frac{x}{\sqrt{\tau}}\right)$ solves the time-reversed heat equation $(h_n)_\tau = -\frac14 (h_n)_{xx}$. With a bit of hindsight, this fact coupled with knowledge of \eqref{eq:timerev-heat-equation} could have been seen as yet another clue foreshadowing the connection we made explicit in Theorem~\ref{thm:hermite-poisson-polya}.)

One reason why the above derivation was worth noting is that it has a nice analogue in the context of the expansion of the Riemann xi function in the orthogonal polynomial family $(f_n)_{n=0}^\infty$---see Theorem~\ref{thm:poissonflow-dde} in \secref{sec:fn-poissonflow}.

\chapter{Expansion of $\Xi(t)$ in the polynomials $f_n$}

\label{ch:fn-expansion}

Recall that in the Introduction we discussed a family of polynomials $(f_n)_{n=0}^\infty$ defined as
\begin{equation*}
f_n(x) = P_n^{(3/4)}(x;\pi/2)
=
\frac{(3/2)_n}{n!} 
i^n \gausshyper\left(-n, \frac34+ix; \frac32; 2 \right),
\end{equation*}
where $P_n^{(\lambda)}(x;\phi)$ denotes the Meixner-Pollaczek polynomials with parameters $\lambda, \phi$. The polynomials $f_n$ form a family of orthogonal polynomials with respect to the weight function $\left|\Gamma\left(\frac34+ix\right)\right|^2$ on $\R$. Their properties are summarized in \secref{sec:orth-fn}.
Our main goal in this chapter is to derive the expansion \eqref{eq:fn-expansion-intro} for $\Xi(t)$ in the (trivially rescaled) orthogonal polynomials $f_n(t/2)$, which we refer to as the $f_n$-expansion, and to investigate some of its key properties. After proving two main results about the existence of the expansion and the asymptotic behavior of the coefficients, we will see that thinking about the $f_n$-expansion leads to a natural family of approximations to $\Xi(t)$ arising out of the Poisson flow of the orthogonal polynomial family $(f_n)_{n=0}^\infty$. The ideas in this chapter will also prepare the ground for much additional theory developed in Chapters~\ref{ch:radial} and~\ref{ch:gn-expansion}.

\section{Main results}

\label{sec:fnexp-mainresults}

We start by identifying the numbers $c_{2n}$ that will play the role of the coefficients in the $f_n$-expansion. We define more generally numbers $(c_n)_{n=0}^\infty$ by
\begin{equation}
\label{eq:def-fn-coeffs}
c_n = 2\sqrt{2}\int_0^\infty \frac{\omega(x)}{(x+1)^{3/2}} \left(\frac{x-1}{x+1}\right)^n\,dx.
\end{equation}
The integral converges absolutely, by \eqref{eq:omegax-asym-xinfty}--\eqref{eq:omegax-asym-xzero}.
Moreover, the functional equation \eqref{eq:theta-functional-equation} implies through a trivial change of variables $u=1/x$ that
\begin{equation} \label{eq:omega-int-powersn-symmetry}
\int_0^1 \frac{\omega(x)}{(x+1)^{3/2}} \left(\frac{x-1}{x+1}\right)^n\,dx
= (-1)^n \int_1^\infty \frac{\omega(u)}{(u+1)^{3/2}} \left(\frac{u-1}{u+1}\right)^n\,du.
\end{equation}
It follows that $c_{2n+1} = 0$ for all $n$, and that the even-indexed numbers $c_{2n}$ can be alternatively expressed as
\begin{equation}
\label{eq:c2n-formula}
c_{2n} = 4\sqrt{2}\int_1^\infty \frac{\omega(x)}{(x+1)^{3/2}} \left(\frac{x-1}{x+1}\right)^{2n}\,dx.
\end{equation}
Since the integrand in \eqref{eq:c2n-formula} is positive on $(1,\infty)$, the numbers $c_{2n}$ are positive.

With these preliminary remarks, we can formulate the main result on the expansion \eqref{eq:fn-expansion-intro}.

\begin{thm}[Infinite series expansion for $\Xi(t)$ in the polynomials $f_n$]
\label{THM:FN-EXPANSION}
The Riemann xi function has the infinite series representation
\begin{equation}
\label{eq:fn-expansion}
\Xi(t) = \sum_{n=0}^\infty (-1)^n c_{2n} f_{2n}\left(\frac{t}{2}\right),
\end{equation}
which converges uniformly on compacts for all $t\in\C$. More precisely, for any compact set $K\subset \C$ there exist constants $C_1, C_2>0$ depending on $K$ such that
\begin{equation}
\label{eq:fn-expansion-errorbound}
\left|\Xi(t) - \sum_{n=0}^N (-1)^n c_{2n} f_{2n}\left(\frac{t}{2}\right) \right| \leq C_1 e^{-C_2 \sqrt{N}}
\end{equation}
holds for all $N\ge0$ and $t\in K$.
\end{thm}

We will also prove a formula describing the asymptotic behavior of the coefficient sequence $c_{2n}$.

\begin{thm}[Asymptotic formula for the coefficients $c_{2n}$]
\label{THM:FN-COEFF-ASYM}
The asymptotic behavior of $c_{2n}$ for large $n$ is given by
\begin{equation}
\label{eq:fn-coeff-asym}
c_{2n} = \left(1+O\left(n^{-1/10}\right)\right) 16 \sqrt{2} \pi^{3/2}\, \sqrt{n} \,\exp\left(-4\sqrt{\pi n}\right)
\end{equation}
as $n\to\infty$.
\end{thm}

A corollary of the above results,
analogous to Corollary~\ref{cor:hermite-coeff-innerproduct},
is the following.

\begin{corollary}
\label{cor:fn-coeff-innerproduct}
The coefficients $c_n$ can be alternatively expressed as
\begin{equation}
\label{eq:fn-coeff-innerproduct}
c_n = (-i)^n \frac{\sqrt{2} \, n!}{\pi^{3/2} (3/2)_n}
\int_{-\infty}^\infty \Xi(t) f_n\left(\frac{t}{2}\right)\left|\Gamma\left(\frac34+\frac{it}{2}\right)\right|^2\,dt.
\end{equation}
\end{corollary}

\begin{proof}
This is analogous to the proof of Corollary~\ref{cor:hermite-coeff-innerproduct}.
\end{proof}

\section{Proof of Theorem~\ref{THM:FN-EXPANSION}}

\label{sec:fnexp-proofexpansion}

The next two lemmas establish technical bounds that will be useful for our analysis and play a similar role to the one played in the previous chapter by Lemmas~\ref{lem:hermite-easybound} and~\ref{lem:hermite-easybound2}.

\begin{lem} 
\label{lem:fn-easybound}
The polynomials $f_n(x)$ satisfy the bound
\begin{equation}
\label{eq:fn-easybound}
|f_n(x)|\leq C_1 e^{C_2 n^{1/3}}
\end{equation}
for all $n\ge 0$, uniformly as $x$ ranges over any compact set $K\subset \C$, with $C_1,C_2>0$ being constants that depend on $K$ but not on $n$.
\end{lem}

\begin{proof}
Fix the compact $K\subset \C$, and denote $M=2\max_{x\in K} |x|$.
Fix an integer $N_0\ge \max(4,(3M)^3)$. Let $C_1,C_2>0$ be constants for which \eqref{eq:fn-easybound} holds for all $x\in K$ and $0\le n\le N_0$, and such that $C_2\ge 1$. Note that for all $n\ge N_0$ we have the inequality
$
n^{1/3} - (n-2)^{1/3} \geq \frac{2}{3n^{2/3}},
$
which implies that
\begin{equation}
\label{eq:fn-easybound-auxineq}
e^{-C_2 (n^{1/3} - (n-2)^{1/3})} \leq e^{-(n^{1/3} - (n-2)^{1/3})} \leq 1-\frac{1}{3n^{2/3}}.
\end{equation}
(since $e^{-x}\leq 1-x/2$ if $0\le x\le 1$).
Then, assuming by induction that we have proved the bound \eqref{eq:fn-easybound} for all cases up to $n-1$, in the $n$th case (where $n> N_0$) we can use the recurrence \eqref{eq:fn-recurrence} and \eqref{eq:fn-easybound-auxineq} to write that, for all $x\in K$,
\begin{align*}
|f_n(x)| &\leq \frac{2|x|}{n} |f_{n-1}(x)| + \left(1-\frac{1}{2n}\right)|f_{n-2}(x)|
\leq 
\frac{M}{n} C_1 e^{C_2 (n-1)^{1/3}} + \left(1-\frac{1}{2n}\right) C_1 e^{C_2 (n-2)^{1/3}}
\\ &\leq C_1 e^{C_2 n^{1/3}}
\left[
\frac{M}{n} e^{-C_2 (n^{1/3}-(n-1)^{1/3})}
+ \left(1-\frac{1}{2n}\right) e^{-C_2 (n^{1/3}-(n-2)^{1/3})}
\right]
\\ &\leq C_1 e^{C_2 n^{1/3}}
\left[
\frac{M}{n} 
+ \left(1-\frac{1}{2n}\right) \left(1-\frac{1}{3n^{2/3}}\right)
\right]
\\ &\leq C_1 e^{C_2 n^{1/3}}
\left[
\frac{1}{3n^{2/3}} 
+ \left(1-\frac{1}{3n^{2/3}}\right)
\right] = C_1 e^{C_2 n^{1/3}}.
\end{align*}
This completes the inductive step.
\end{proof}

\begin{lem}
\label{lem:fn-easybound2}
For any number $B\ge1$ there is a constant $C>0$ such that 
\begin{equation}
\label{eq:fn-easybound2}
\int_1^\infty e^{-B x} \left( \frac{x-1}{x+1}\right)^n\,dx \leq C e^{-2 \sqrt{B n}}
\end{equation}
for all $n\ge 0$.
\end{lem}

\begin{proof}
The integral can be expressed as 
\begin{align*}
\int_1^\infty \exp\left(\psi_n(x)\right)\, dx,
\end{align*}
where we define
\begin{equation}
\label{eq:largedev-function}
\psi_n(x) = -B x + n \log \left( \frac{x-1}{x+1}\right).
\end{equation}
By solving the equation $\psi_n'(x) = 0$, it is easy to check that $\psi_n(x)$ has a unique global maximum point $x_n$ in $[1,\infty)$, namely
\begin{equation*}
0 = \psi_n'(x_n) = -B + \frac{2n}{x_n^2-1} \quad \iff \quad 
x_n = \sqrt{\frac{2n}{B}+1}, 
\end{equation*}
which asymptotically as $n\to\infty$ behaves as
\begin{equation*}
x_n = \sqrt{\frac{2n}{B}} + O\left(\frac{1}{\sqrt{n}}\right).
\end{equation*}
The value at the maximum point is
\begin{align*}
\psi_n(x_n) &= -B x_n + n \log\left(\frac{x_n-1}{x_n+1}\right)
= 
-B x_n + n \log\left(\frac{1-1/x_n}{1+1/x_n}\right)
\\ &= - \sqrt{2B n} +O\left(\frac{1}{\sqrt{n}}\right) 
+ n\left(-2\cdot\frac{1}{x_n} + O\left(\frac{1}{x_n^3}\right)\right)
=
-2 \sqrt{2B n} +O\left(\frac{1}{\sqrt{n}}\right)
\end{align*}
as $n\to\infty$.
We conclude that
\begin{align*}
\int_1^\infty \exp\left(\psi_n(x)\right)\, dx
&= 
\int_1^{2x_n} \exp\left(\psi_n(x)\right)\, dx
+ \int_{2x_n}^\infty \exp\left(\psi_n(x)\right)\, dx
\\ &\leq
2x_n \exp\left(\psi_n(x_n)\right) + \int_{2x_n}^\infty e^{-B x}\,dx
\\ &=
\left(1+O\left(\frac{1}{\sqrt{n}}\right)\right) 
2\sqrt{\frac{2n}{B}}
\exp\left(-2\sqrt{2B n}\right)
+ \frac{1}{B} e^{-2 B x_n}
= O\left( e^{-2\sqrt{B n}}\right)
\end{align*}
as $n\to\infty$,
as claimed.
\end{proof}

Denote $s=\frac12+it$, and observe that with this substitution the Mellin transform representation \eqref{eq:riemannxi-mellintrans} for $\xi(s)$ becomes the statement that 
\begin{equation*}
\Xi(t) =\int_0^\infty \omega(x)x^{-\frac34+\frac{it}{2}}\,dx.
\end{equation*} 
The idea behind the expansion \eqref{eq:fn-expansion} is the simple yet powerful fact that the integration kernel $x^{-\frac34+\frac{it}{2}}$ can be expanded in a very specific way in an infinite series related to the generating function \eqref{eq:fn-genfun}. More precisely, for any $x>0$ we have that
\begin{align}
\label{eq:int-kernel-expansion}
x^{s/2-1} & =
x^{-\frac34+\frac{it}{2}} = \frac{2\sqrt{2}}{(x+1)^{3/2}}
\left(\frac{2x}{x+1}\right)^{-\frac34+\frac{it}{2}}
\left(\frac{2}{x+1}\right)^{-\frac34-\frac{it}{2}}
\\ &=
\frac{2\sqrt{2}}{(x+1)^{3/2}}
\cdot \left((1-iz)^{-\frac34+\frac{it}{2}}
(1+iz)^{-\frac34-\frac{it}{2}}
\right)_{\Big| z=i\frac{x-1}{x+1}}
\nonumber \\ &=
\frac{2\sqrt{2}}{(x+1)^{3/2}}
\sum_{n=0}^\infty 
f_n\left(\frac{t}{2}\right) \left(i \,\frac{x-1}{x+1}\right)^n 
=
\frac{2\sqrt{2}}{(x+1)^{3/2}}
\sum_{n=0}^\infty 
i^n f_n\left(\frac{t}{2}\right) \left(\,\frac{x-1}{x+1}\right)^n. 
\nonumber
\end{align}
One can now get \eqref{eq:fn-expansion} as a formal identity by multiplying the first and last expressions in this chain of equations by $\omega(x)$ and integrating both sides over $(0,\infty)$, then using the fact that the odd-indexed coefficients $c_{2n+1}$ vanish. 

To rigorously justify this formal calculation and obtain the more precise rate of convergence estimate \eqref{eq:fn-expansion-errorbound}, 
we now make use of the technical bounds from Lemmas~\ref{lem:fn-easybound} and \ref{lem:fn-easybound2}. Using the above infinite series representation of the kernel $x^{-\frac34+\frac{it}{2}}$, we see that
\begin{align*}
\Bigg|\Xi(t) - \sum_{n=0}^N & (-1)^{2n}c_{2n} f_{2n}\left(\frac{t}{2}\right)\Bigg|
=
\left|\Xi(t) - \sum_{n=0}^{2N} i^n c_n f_n\left(\frac{t}{2}\right)\right|
\\ &=
\left|\int_0^\infty
\omega(x)
\left(x^{-\frac34+\frac{it}{2}} - 
\frac{2\sqrt{2}}{(x+1)^{3/2}}
\sum_{n=0}^{2N} f_n\left(\frac{t}{2}\right)\left(
i \frac{x-1}{x+1}\right)^n
\right)
\,dx\right|
\\ &\leq
\int_0^\infty
\omega(x)
\left|x^{-\frac34+\frac{it}{2}} - 
\frac{2\sqrt{2}}{(x+1)^{3/2}}
\sum_{n=0}^{2N} f_n\left(\frac{t}{2}\right)\left(
i \frac{x-1}{x+1}\right)^n
\right|
\,dx
\\ &=
2\sqrt{2} \int_0^\infty
\frac{\omega(x)}{(x+1)^{3/2}}
\left|
\sum_{n=2N+1}^{\infty} f_n\left(\frac{t}{2}\right)
\left(i \frac{x-1}{x+1}
\right)^n
\right|
\,dx
\\ &\leq
2\sqrt{2} \int_0^\infty
\frac{\omega(x)}{(x+1)^{3/2}}
\sum_{n=2N+1}^{\infty} \left|f_n\left(\frac{t}{2}\right)\right| \cdot\left|
i \frac{x-1}{x+1}
\right|^n
\,dx
\\ &\leq
2\sqrt{2} \int_0^\infty
\frac{\omega(x)}{(x+1)^{3/2}}
\sum_{n=2N+1}^{\infty} C_1 e^{C_2 n^{1/3}} \left|
\frac{x-1}{x+1}
\right|^n
\,dx
\\ &=
2\sqrt{2} \sum_{n=2N+1}^{\infty} C_1  e^{C_2 n^{1/3}} \int_0^\infty
\frac{\omega(x)}{(x+1)^{3/2}}
\left|
\frac{x-1}{x+1}
\right|^n
\,dx
\\ &
=
4\sqrt{2} \sum_{n=2N+1}^{\infty} C_1  e^{C_2 n^{1/3}} \int_1^\infty
\frac{\omega(x)}{(x+1)^{3/2}}
\left(
\frac{x-1}{x+1}
\right)^n
\,dx,
\end{align*}
where $C_1,C_2$ are the constants from Lemma~\ref{lem:fn-easybound} (associated with the compact set $K$ over which we are allowing $t$ to range); the last step follows from \eqref{eq:omega-int-powersn-symmetry}. Now note that, by \eqref{eq:omegax-asym-xinfty}, $\frac{\omega(x)}{(x+1)^{3/2}} = O\left(\sqrt{x}e^{-\pi x}\right) = O\left(e^{-\pi x/2}\right)$ as $x\to\infty$, so we can apply Lemma~\ref{lem:fn-easybound2} (with $B=\pi/2$) to the integrals, to get that the last expression in the above chain of inequalities is bounded by
$$
4\sqrt{2} \sum_{n=2N+1}^{\infty} C_1  e^{C_2 n^{1/3}} \cdot C e^{-2 \sqrt{\pi n/2}},
$$
and this is easily seen to be $O\left(e^{-\sqrt{\pi N}}\right)$ as $N\to\infty$. This proves \eqref{eq:fn-expansion-errorbound} and completes the proof of Theorem~\ref{THM:FN-EXPANSION}.
\qed

\section{Proof of Theorem~\ref{THM:FN-COEFF-ASYM}}

\label{sec:fnexp-proofasym}

Define a function $\phi(x)$, and numbers $Z_n$ and $\varepsilon_n$, by
\begin{align*}
\phi(x) & = \frac{\pi x(2\pi x-3)}{(x+1)^{3/2}}, \\
Z_n & = \int_1^\infty \phi(x)e^{-\pi x}
\left(\frac{x-1}{x+1}\right)^{2n}\,dx, \\
\varepsilon_n & = 
\int_1^\infty \left(\frac{\omega(x)}{(x+1)^{3/2}}-\phi(x)e^{-\pi x}\right)
\left(\frac{x-1}{x+1}\right)^{2n}\,dx,
\end{align*}
so that $c_{2n}$ in \eqref{eq:c2n-formula} can be rewritten as $c_{2n} = 4\sqrt{2}(Z_n + \varepsilon_n)$. We consider separately the asymptotic behavior of $Z_n$ and $\varepsilon_n$. For $\varepsilon_n$, note that we have
\begin{equation}
\label{eq:omegax-firstorder-asym}
\left|\frac{\omega(x)}{(x+1)^{3/2}} - \phi(x)e^{-\pi x}\right| = 
O\left(e^{-3\pi x}\right)
\quad \textrm{as }x\to\infty,
\end{equation}
by \eqref{eq:omegaxdiff-asym-xinfty}.
Thus, Lemma~\ref{lem:fn-easybound2} implies that
\begin{equation} \label{eq:asym-estimate-epsilonn}
|\varepsilon_n| = O\left(e^{-2\sqrt{6\pi n}}\right) \quad \textrm{as }n\to\infty.
\end{equation}
The main asymptotic contribution to $c_{2n}$ comes from $Z_n$, and can be found using Laplace's method. 
Start by rewriting $Z_n$ as 
\begin{equation*}
Z_n = \int_1^\infty \phi(x) \exp(\psi_{2n}(x))\,dx,
\end{equation*}
where $\psi_n(x)$ is the function defined in \eqref{eq:largedev-function} with $B=\pi$. Noting that, as was discussed in the proof of Lemma~\ref{lem:fn-easybound2}, $\psi_{2n}(x)$ has a unique global maximum point at
\begin{equation*} 
\tau_n := x_{2n} = \sqrt{\frac{4n}{\pi}+1} = \sqrt{\frac{4n}{\pi}} + O\left(\frac{1}{\sqrt{n}}\right) \qquad \textrm{as }n\to\infty,
\end{equation*}
we further split this integral up into three parts, by writing
$Z_n = Z_n^{(1)} + Z_n^{(2)} + Z_n^{(3)}$, with
\begin{align}
Z_n^{(1)} &= \int_1^{\tau_n-n^{3/10}}
\phi(x) \exp(\psi_{2n}(x))\,dx,
\label{eq:int-z1}
\\
Z_n^{(2)} &=
\int_{I_n}
\phi(x) \exp(\psi_{2n}(x))\,dx,
\label{eq:int-z2}
\\
Z_n^{(3)} &=
\int_{\tau_n+n^{3/10}}^\infty
\phi(x) \exp(\psi_{2n}(x))\,dx,
\label{eq:int-z3}
\end{align}
where $I_n$ denotes the interval $[\tau_n-n^{3/10},\tau_n+n^{3/10}]$.

The following calculus facts are straightforward to check; their verification is left to the reader:
\begin{enumerate}
\item $\phi(x)$ is monotone increasing on $[1,\infty)$.
\item $\psi_{2n}(x)$ is monotone increasing on $[1,\tau_n]$ and monotone decreasing on $[\tau_n,\infty)$.
\item $\psi_{2n}(x)$ is concave on $[1,\infty)$.
\item We have the asymptotic relations
\begin{align*}
V_n &:= \psi_{2n}(\tau_n) = -\pi \tau_n + 2n 
\log\left( \frac{\tau_n-1}{\tau_n+1}\right)
= -4\sqrt{\pi n} + O\left(\frac{1}{\sqrt{n}}\right), \\
D_n &:= -\psi_{2n}''(\tau_n) = \frac{\pi^2}{2n} \tau_n
= \pi^{3/2} \frac{1}{\sqrt{n}} + O\left(\frac{1}{n^{3/2}}\right),
\\
E_n &:= \phi(\tau_n) 
= 2\sqrt{2} \pi^{7/4} n^{1/4} + O\left(\frac{1}{n^{1/4}}\right),
\\
K_n &:= \psi_{2n}'(\tau_n + n^{3/10}) = -\pi^{3/2} n^{-1/5} + O\left(\frac{1}{n^{2/5}}\right)
\end{align*}
as $n\to\infty$.

\item We have the relation
$
\psi_{2n}'''(x) = \frac{8n(3x^2+1)}{(x^2-1)^3}. 
$
Consequently
\begin{equation*}
\sup_{x\in I_n} |\psi_{2n}'''(x)|
= O\left(\frac{1}{n}\right) \qquad
\textrm{as }n\to\infty,
\end{equation*}
which implies that the Taylor expansion of $\psi_{2n}(x)$ around $x=\tau_n$ can be written as
\begin{align*}
\psi_{2n}(x) = V_n -\frac12 D_n (x-\tau_n)^2 + O\left(\frac{|x-\tau_n|^3}{n}\right), \qquad (x\in I_n)
\end{align*}
where the constant implicit in the big-$O$ term does not depend on $x$ or $n$.

\item We have
\begin{equation*}
\sup_{x\in I_n} \left| \frac{\phi(x)}{E_n}-1\right| = O\left(\frac{1}{n^{1/5}}\right) \qquad \textrm{as }n\to\infty.
\end{equation*}

\end{enumerate}

We now estimate the integrals \eqref{eq:int-z1}--\eqref{eq:int-z3}. For $Z_n^{(1)}$, since $\phi(x)$ and $\psi_{2n}(x)$ are increasing on $[1,\tau_n]$, we have that
\begin{align}
\label{eq:asym-estimate-z1}
|Z_n^{(1)}| &\leq 
(\tau_n-n^{3/10}-1) \phi(\tau_n-n^{3/10}) \exp\left(
\psi_{2n}(\tau_n-n^{3/10}) \right)
\\ &\leq O\left(n^{3/4}\right) \exp\left(V_n
-\frac12 D_n n^{3/5} + 
O\left( \frac{n^{9/10}}{n}
\right)
\right)
= O\left( \frac{1}{n^2} e^{-4\sqrt{\pi n}} \right)
\nonumber
\end{align}
as $n\to\infty$.

For $Z_n^{(3)}$, we use the fact that 
\begin{equation*}
\psi_{2n}(x) \leq \psi_{2n}(\tau_n+n^{3/10}) + K_n (x-\tau_n-n^{3/10})
\end{equation*} 
for all $x\ge \tau_n+n^{3/10}$ (which follows from the concavity of $\psi_{2n}(x)$) to
write 
\begin{align}
\label{eq:asym-estimate-z3}
|Z_n^{(3)}| &\leq
\int_{\tau_n+n^{3/10}}^\infty
\phi(x) \exp\Bigg[ \psi_{2n}(\tau_n+n^{3/10}) 
+ K_n
(x-\tau_n-n^{3/10})\Bigg]
\,dx
\\ \nonumber &\leq
\exp\left[
V_n - \frac12 D_n n^{3/5} + O\left(\frac{n^{9/10}}{n}\right)
\right]
\\ \nonumber &\ \ \ \quad\qquad \times 
\int_{\tau_n+n^{3/10}}^\infty O\left( \sqrt{x} \right)
\exp\Bigg[
-\left(1+O\left(\frac{1}{n^{1/5}}\right)\right)
\pi^{3/2} n^{-1/5} (x-\tau_n-n^{3/10})
\Bigg]
\,dx
\\ \nonumber &=
e^{-4\sqrt{\pi n}} O\left(e^{-n^{1/20}}\right)
=
O\left(\frac{1}{n^2}e^{-4\sqrt{\pi n}}\right).
\end{align}

Finally, to obtain the asymptotics of $Z_n^{(2)}$, we make the change of variables $u=\sqrt{D_n}(x-\tau_n)$ in the integral \eqref{eq:int-z2}, to get that
\begin{align}
\label{eq:asym-estimate-z2}
Z_n^{(2)} &=
\int_{-\sqrt{D_n} n^{3/10}}^{\sqrt{D_n} n^{3/10}} \phi\left(\tau_n + \frac{u}{\sqrt{D_n}}
\right)
\exp\left( \psi_{2n}\left(\tau_n + \frac{u}{\sqrt{D_n}}
\right) \right)\,\frac{du}{\sqrt{D_n}}
\\ &=
\frac{1}{\sqrt{D_n}}\int_{-\sqrt{D_n} n^{3/10}}^{\sqrt{D_n} n^{3/10}} \phi\left(\tau_n + \frac{u}{\sqrt{D_n}}
\right)
\exp\left[ V_n - \frac12 u^2 + O\left(
\frac{|u|^3}{n D_n^{3/2}}
\right)
\right]
\,du
\nonumber 
\\ &=
\left(1+O\left(\frac{1}{n^{1/4}}\right)\right) \pi^{-3/4} n^{1/4} 
\times \left(1+O\left(\frac{1}{n^{1/5}}\right)\right) E_n
\nonumber \\ & \qquad 
\times \left(1+O\left(\frac{1}{\sqrt{n}}\right)\right) e^{-4\sqrt{\pi n}}
\times \int_{-\sqrt{D_n} n^{3/10}}^{\sqrt{D_n} n^{3/10}}
\exp\left[-u^2/2 + O\left(\frac{1}{n^{1/10}}\right)\right]\,du
\nonumber \\ &=
\left(1+O\left(\frac{1}{n^{1/10}}\right)\right) 
\pi^{-3/4} n^{1/4} 
\times 2\sqrt{2}\pi^{7/4} n^{1/4}
\times e^{-4\sqrt{\pi n}}
\left(\sqrt{2\pi}-
O\left(\exp\left(-\frac12 D_n n^{3/5}\right)\right)\right)
\nonumber \\ &= 
\left(1+O\left(\frac{1}{n^{1/10}}\right)\right) 
4 \pi^{3/2} n^{1/2} e^{-4\sqrt{\pi n}},
\nonumber
\end{align}
where we once again used \eqref{eq:gaussian-tail-bound} to account for the error arising from adding the tails of the Gaussian integral.

Since $c_{2n} = 4\sqrt{2}\left(\varepsilon_n + Z_n^{(1)}+Z_n^{(2)}+Z_n^{(3)}\right)$, combining \eqref{eq:asym-estimate-epsilonn}, \eqref{eq:asym-estimate-z1}, \eqref{eq:asym-estimate-z3} and \eqref{eq:asym-estimate-z2} gives the asymptotic formula
\eqref{eq:fn-coeff-asym}. The proof of Theorem~\ref{THM:FN-COEFF-ASYM} is complete.
\qed

\section{The Poisson flow associated with the $f_n$-expansion}

\label{sec:fn-poissonflow}

Motivated by the developments of \secref{sec:hermite-poissonpolya}, we now consider the Poisson flow 
\eqref{eq:poisson-flow-def} associated with the family $(f_n)_{n=0}^\infty$ of orthogonal polynomials, which in this section we will denote by $\mathcal{F}$. Recall that in the case of the Hermite expansion, we showed that the Poisson flow could be understood as the family of Fourier transforms of functions obtained from the function $\Phi(x)$ by performing a simple operation (refer to \eqref{eq:hermite-poissonpolya-calc}). One might wonder if something similar (or perhaps even more interesting) happens in the case of the Poisson flow associated with the family $\mathcal{F}$. The answer is given in the following result.

\begin{thm}[Mellin transform representation of the Poisson flow]
\label{thm:fn-poisson-mellin}
For $0<r<1$,
the function $X_r^{\mathcal{F}}(t)$ has the Mellin transform representation
\begin{equation}
\label{eq:fn-poisson-mellin}
X_r^{\mathcal{F}}(t) 
=
\int_0^\infty \omega_r(x) x^{-\frac34+\frac{i t}{2}}\,dx
\qquad (t\in \C),
\end{equation}
where we define 
\begin{equation}
\label{eq:fn-omega-compressed}
\omega_r(x) = \begin{cases}
\frac{1+\eta}{\sqrt{1-\eta}} \frac{1}{\sqrt{1-\eta x}}
\omega\left(\frac{x-\eta}{1-\eta x}\right) & \textrm{if }\eta < x<1/\eta, \\
0 & \textrm{otherwise,}
\end{cases}
\end{equation}
making use of the notation
\begin{equation}
\label{eq:compression-eta-notation}
\eta = \frac{1-r}{1+r}.
\end{equation}
\end{thm}

Note that the map $x\mapsto \frac{x-\eta}{1-\eta x}$ maps the interval $(\eta,1/\eta)$ bijectively onto $(0,\infty)$, so the function $\omega_r(x)$ contains the same ``frequency information'' as $\omega(x)$, but compressed into a finite interval. In particular, 
a notable feature of this result, which stands in contrast to  what we saw in the case of the Poisson flow associated with the Hermite polynomials, is that for $r<1$ the function $X_r^{\mathcal{F}}(t)$ is now the Fourier transform of a function with bounded support; that is, $X_r^{\mathcal{F}}(t)$ is an entire function of exponential type. It is intriguing to speculate that this might make the problem of understanding where the zeros of $X_r^{\mathcal{F}}(t)$ are located easier than for the case of the original xi function $\Xi(t)$. 

\begin{proof}
The derivation starts with the formula \eqref{eq:poisson-flow-expansion}. Specializing this to the expansion \eqref{eq:fn-expansion} and substituting the defining formula \eqref{eq:def-fn-coeffs} for the coefficients $c_n$, we have that
\begin{align*}
X_r^{\mathcal{F}}(t) &= \sum_{n=0}^\infty i^n c_n r^n f_n\left(\frac{t}{2}\right)
=
2\sqrt{2} \sum_{n=0}^\infty \int_0^\infty \frac{\omega(x)}{(x+1)^{3/2}} \left(\frac{x-1}{x+1}\right)^n\,dx \cdot r^n f_n\left(\frac{t}{2}\right)
\\&= 
2\sqrt{2} \int_0^\infty \frac{\omega(x)}{(x+1)^{3/2}} 
\sum_{n=0}^\infty f_n\left(\frac{t}{2}\right) \left(i r \frac{x-1}{x+1}\right)^n\,dx 
\\&
= 
2\sqrt{2} \int_0^\infty \frac{\omega(x)}{(x+1)^{3/2}} 
\left(\sum_{n=0}^\infty f_n\left(\frac{t}{2}\right) {z^n}\right)_{\raisebox{3pt}{$\Big|$}z=i r \frac{x-1}{x+1}}\,dx.
\end{align*}
Inside the integrand we have an expression involving the generating function \eqref{eq:fn-genfun} of the polynomials $f_n(x)$. Substituting the formula for this generating function (as we did in \eqref{eq:int-kernel-expansion}, which is essentially the special case $r=1$ of the current computation) gives that
\begin{align*}
X_r^{\mathcal{F}}(t) &=
2\sqrt{2} \int_0^\infty \frac{\omega(x)}{(x+1)^{3/2}} 
\left((1-iz)^{-\frac34+\frac{it}{2}} (1+iz)^{-\frac34+\frac{it}{2}} \right)
_{\raisebox{3pt}{$\Big|$}z=i r \frac{x-1}{x+1}}
\,dx 
\\ &=
2\sqrt{2} \int_0^\infty \frac{\omega(x)}{(x+1)^{3/2}} 
\left(
\frac{1-r+x(1+r)}{x+1}
\right)^{-\frac34+\frac{it}{2}}
\left(
\frac{1+r+x(1-r)}{x+1}
\right)^{-\frac34-\frac{it}{2}}
\,dx 
\\ &=
2\sqrt{2} \int_0^\infty \omega(x)
(1+r)^{-\frac34+\frac{it}{2}} \left(x+\frac{1-r}{1+r}\right)^{-\frac34+\frac{it}{2}}
(1+r)^{-\frac34-\frac{it}{2}} \left(1+\frac{1-r}{1+r}x\right)^{-\frac34-\frac{it}{2}}
\,dx 
\\ &=
\frac{2\sqrt{2}}{(1+r)^{3/2}} \int_0^\infty \omega(x) 
\frac{1}{(1+\eta x)^{3/2}} \left( \frac{x+\eta}{1+\eta x}\right)^{-\frac34+\frac{it}{2}}\,dx
\\ &=
(1+\eta)^{3/2} \int_0^\infty \omega(x) 
\frac{1}{(1+\eta x)^{3/2}} \left( \frac{x+\eta}{1+\eta x}\right)^{-\frac34+\frac{it}{2}}\,dx.
\end{align*}
We have thus expressed $X_r^{\mathcal{F}}(t)$ as a sort of modified Mellin transform of $\omega(x)$. But this last integral formula can be transformed to an ordinary Mellin transform by  making the change of variables 
$x=\frac{u-\eta}{1-\eta u}$ in the last integral. The reader can verify without difficulty that this yields the Mellin transform \eqref{eq:fn-poisson-mellin} of the function given in \eqref{eq:fn-omega-compressed}.
\end{proof}

In the next result we show that the Poisson flow satisfies an interesting dynamical evolution law, analogous to the time-reversed heat equation \eqref{eq:timerev-heat-equation} satisfied by the P\'olya-De Bruijn flow. In this case the evolution law is not a partial differential equation, but rather a \firstmention{differential difference equation (DDE)}. To make the equation homogeneous in the ``time'' variable, it is most convenient to perform a change of variables, reparametrizing the time variable $r$ by denoting $r = e^{-\tau}$ (with $\tau\ge0$).

\begin{thm}[Differential difference equation for the Poisson flow]
\label{thm:poissonflow-dde}
The function $M(\tau,t):=X_{e^{-\tau}}^{\mathcal{F}}(t)$ satisfies the differential difference equation
\begin{equation}
\label{eq:poissonflow-dde}
\frac{\partial M}{\partial \tau} = 
\frac34 M(\tau,t) 
- \frac12\left(\frac34-\frac{it}{2}\right) M(\tau,t+2i)
- \frac12\left(\frac34+\frac{it}{2}\right) M(\tau,t-2i)
\quad (\tau >0, t\in \C).
\end{equation}
\end{thm}

\begin{proof}
The computation is analogous to the derivation of the time-reversed heat equation at the end of \secref{sec:hermite-poissonpolya}, except that instead of using the differential equation satisfied by the Hermite polynomials we use the difference equation \eqref{eq:fn-diffeq} satisfied by the polynomials $f_n(x)$. We have, again starting with \eqref{eq:poisson-flow-expansion} with the substitution $r=e^{-\tau}$,
\begin{align*}
\frac{\partial M}{\partial \tau}
&=
\frac{\partial}{\partial \tau}  
\left(
\sum_{n=0}^\infty i^n c_n e^{-n\tau} f_n\left(\frac{t}{2}\right)
\right)
=
\sum_{n=0}^\infty i^n c_n (-n) e^{-n\tau} f_n\left(\frac{t}{2}\right)
\\ &=
\sum_{n=0}^\infty i^n c_n e^{-n\tau}
\bigg(
\frac34 f_n\left(\frac{t}{2}\right)
-\frac12\left(\frac34+\frac{it}{2}\right) f_n\left(\frac{t}{2}-i\right)
-\frac12\left(\frac34-\frac{it}{2}\right) f_n\left(\frac{t}{2}+i\bigg)
\right)
\\ &=
\frac34 \sum_{n=0}^\infty i^n c_n e^{-n\tau} f_n\left(\frac{t}{2}\right)
-\frac12\left(\frac34+\frac{it}{2}\right)
\sum_{n=0}^\infty i^n c_n e^{-n\tau} f_n\left(\frac{t}{2}-i\right)
-\frac12\left(\frac34-\frac{it}{2}\right)
\sum_{n=0}^\infty i^n c_n e^{-n\tau} f_n\left(\frac{t}{2}+i\right)
\\ &=
\frac34 M(\tau,t) 
- \frac12\left(\frac34-\frac{it}{2}\right) M(\tau,t+2i)
- \frac12\left(\frac34+\frac{it}{2}\right) M(\tau,t-2i).
\end{align*}

\end{proof}

\section{Evolution of the zeros under the Poisson flow}

\label{sec:zeros-evolution}

The differential difference equation \eqref{eq:poissonflow-dde} opens up the way to an analysis of the dynamical evolution of the zeros of the functions $X_r^{\mathcal{F}}(t)$ as a function of $r$, in a manner analogous to how the time-reversed heat equation \eqref{eq:timerev-heat-equation} made it possible to write a system of coupled ODEs satisfied by the P\'olya-De Bruijn flow, which played a useful role in the investigations of the De Bruijn-Newman constant (see \cite[Lemma~5.18, p.~83]{broughan}). Our next goal is to derive this evolution law, again using the more convenient time parameter $\tau$. To avoid technicalities involving the behavior of entire functions (and to generalize the question slightly, which also seems potentially useful), we switch in this section from the Riemann xi function to the simpler setting of polynomials.

Let $z_1,\ldots,z_n \in \C$ be distinct complex numbers.
Let 
\begin{equation}
\label{eq:poly-initial-condition}
p(t) = \prod_{k=1}^n (t-z_k),
\end{equation} 
and consider the function $M_p(\tau,t)$ defined as the solution to the DDE \eqref{eq:poissonflow-dde} with initial condition $M_p(0,t)=p(t)$. To see that such an object exists, write the expansion
\begin{equation*}
p(t) = \sum_{k=0}^n \gamma_k f_k\left(\frac{t}{2}\right)
\end{equation*}
in the linear basis of polynomials $(f_k(t/2))_{k=0}^n$. Then $M_p(\tau,t)$ is given by
\begin{equation}
\label{eq:poly-poisson-evolution}
M_p(\tau,t) = \sum_{k=0}^n \gamma_k e^{-k \tau} f_k\left(\frac{t}{2}\right)
\end{equation}
(the proof is a repetition of the calculation in the proof of Theorem~\ref{thm:poissonflow-dde} above, with both proofs being based on the simple observation that each of the functions $m_k(\tau,t)= e^{-k\tau} f_k(t/2)$ for $k\ge 0$ is a solution to \eqref{eq:poissonflow-dde}). Proving uniqueness is left as an exercise. We refer to the function $M_p(\tau,t)$ as the \firstmention{Poisson flow (associated with the polynomial family $\mathcal{F}$) with initial condition $p$}.

For any fixed $\tau \in \R$, the function $t\mapsto M_p(\tau,t)$ is a polynomial of degree $n$ with leading coefficient $e^{-n\tau}$ (to see this, compare \eqref{eq:poly-poisson-evolution} at times $\tau$ and $0$, taking into account \eqref{eq:poly-initial-condition}). Denote its zeros by $Z_1(\tau),\ldots,Z_n(\tau)$, and note that while they are defined only up to ordering, in the neighborhood of any fixed time $\tau_0$ for which the zeros are distinct one can pick the ordering so that $Z_k(\tau)$ are smooth functions of~$\tau$.

\begin{thm}[Evolution equations for the zeros under the Poisson flow]
In the neighborhood of any $\tau_0$ as above, 
the functions $(Z_k(\tau))_{k=1}^n$ satisfy the system of coupled ordinary differential equations
\begin{align*}
\frac{dZ_k}{d\tau} &= 
\frac12 \Bigg[
\left( Z_k + \frac{3i}{2} \right)\prod_{1\le j\le n \atop j\neq k} \left(1+\frac{2i}{Z_k-Z_j}\right)
+
\left( Z_k - \frac{3i}{2} \right)\prod_{1\le j\le n \atop j\neq k} \left(1-\frac{2i}{Z_k-Z_j}\right)
\Bigg]
\qquad (1\le k\le n).
\end{align*}
\end{thm}

\begin{proof}
The fundamental relation defining the $k$th zero $Z_k$ is
\begin{equation*}
M_p(\tau, Z_k(\tau)) = 0.
\end{equation*}
Differentiating this with respect to $\tau$ gives
\begin{align*}
0 &= \frac{d}{d\tau}\Big(M_p(\tau,Z_k(\tau))\Big) = \frac{\partial M_p}{\partial \tau}(\tau,Z_k(\tau)) + \frac{\partial M_p}{\partial t}(\tau,Z_k(\tau))\frac{dZ_k}{d\tau}
\end{align*}
(where $\frac{\partial M_p}{\partial t}$ refers to the partial derivative with respect to the second variable). 
By \eqref{eq:poissonflow-dde}, this expands to
\begin{align*}
0 &= \frac34 M_p(\tau,Z_k)
- \frac12\left(\frac34-\frac{iZ_k}{2}\right) M_p(\tau,Z_k+2i)
- \frac12\left(\frac34+\frac{iZ_k}{2}\right) M_p(\tau,Z_k-2i)
+ \frac{\partial M_p}{\partial t}(\tau,Z_k(\tau))\frac{dZ_k}{d\tau}
\\ &
=
- \frac12\left(\frac34-\frac{iZ_k}{2}\right) M_p(\tau,Z_k+2i)
- \frac12\left(\frac34+\frac{iZ_k}{2}\right) M_p(\tau,Z_k-2i)
+
\frac{\partial M_p}{\partial t}(\tau,Z_k(\tau))\frac{dZ_k}{d\tau}.
\end{align*}
Now, $M_p(\tau,t) = e^{-n \tau} \prod_{j=1}^n (t-Z_j(\tau))$, so
\begin{equation*}
\frac{\partial M_p}{\partial t}(\tau,Z_k(\tau))
= e^{-n\tau} \prod_{1\le j\le n \atop j\neq k} (Z_k-Z_j).
\end{equation*}
It follows that
\begin{align*}
\frac{dZ_k}{d\tau}
&= \frac12 \frac{1}{\frac{\partial M_p}{\partial t}(\tau,Z_k(\tau))}
\left[
\left(\frac34-\frac{iZ_k}{2}\right) M_p(\tau,Z_k+2i)
+ \left(\frac34+\frac{iZ_k}{2}\right) M_p(\tau,Z_k-2i)
\right]
\\ &= \frac12 \prod_{1\le j\le n \atop j\neq k} (Z_k-Z_j)^{-1}
\left[
\left(\frac34-\frac{iZ_k}{2}\right) \prod_{j=1}^n (Z_k+2i-Z_j)
+ \left(\frac34+\frac{iZ_k}{2}\right) \prod_{j=1}^n (Z_k-2i-Z_j)
\right]
\\ &= \frac12 
\Bigg[
2i\left(\frac34-\frac{iZ_k}{2}\right) \prod_{1\le j\le n\atop j\neq k} \frac{Z_k+2i-Z_j}{Z_k-Z_j}
+
(-2i)\left(\frac34+\frac{iZ_k}{2}\right) \prod_{1\le j\le n\atop j\neq k} \frac{Z_k-2i-Z_j}{Z_k-Z_j}
\Bigg]
\\
&= \frac12 \Bigg[
\left( Z_k + \frac{3i}{2} \right)\prod_{1\le j\le n \atop j\neq k} \left(1+\frac{2i}{Z_k-Z_j}\right)
+
\left( Z_k - \frac{3i}{2} \right)\prod_{1\le j\le n \atop j\neq k} \left(1-\frac{2i}{Z_k-Z_j}\right)
\Bigg],
\end{align*}
as claimed.
\end{proof}

Our final result for this section is of a negative sort, illustrating another way in which the Poisson flow associated with the family $\mathcal{F}$ of orthogonal polynomials behaves differently from the P\'olya-De Bruijn flow.
Specifically, it was mentioned in \secref{sec:hermite-poissonpolya} that the P\'olya-De Bruijn flow preserves the property of hyperbolicity.
Our result shows that the Poisson flow associated with the family $\mathcal{F}$ does \emph{not}.

\begin{prop}
\label{prop:poisson-notpreshyp}
There exists a polynomial
\begin{equation*}
P(t) = \sum_{k=0}^n \gamma_k f_k\left(\frac{t}{2}\right),
\end{equation*}
and numbers $\tau_1>0$ and $\tau_2<0$, such that
$P(t)$ has only real zeros, but the polynomials $t\mapsto M_P(\tau_1,t)$ and $t\mapsto M_P(\tau_2,t)$ both have non-real zeros.
\end{prop}

\begin{proof}
Take
\begin{equation*}
P(t) = (x-2)(x-2.01)(x-4)
= \sum_{k=0}^4 \sigma_k f_k\left(\frac{t}{2}\right),
\end{equation*}
where
\begin{equation*}
(\sigma_0, \sigma_1, \sigma_2, \sigma_3) =
\left( -\frac{5619}{1600},
\frac{83}{25},
- \frac{801}{400},
\frac34 \right),
\end{equation*}
and $\tau_1 = 0.1$ and $\tau_2 = -0.05$. Direct calculation of the zeros of $M_P(\tau_1,t)$ and $M_P(\tau_2,t)$ verifies the claim.
\end{proof}

One conclusion from Proposition~\ref{prop:poisson-notpreshyp} is that there does not seem to be an obvious way to define an analogue of the De Bruijn-Newman constant in the context of the $f_n$-expansion of the Riemann xi function.

\chapter{Radial Fourier self-transforms}

\label{ch:radial}

In this chapter we continue to probe deeper into the theory of the $f_n$-expansion of the Riemann xi function, by developing what will turn out to be an entirely new way of thinking about the expansion as arising out of the expansion of an elementary function $A(r)$ (described below in \eqref{eq:aofr-explicit}) in a natural orthogonal basis of functions related to the Laguerre polynomials $L_n^{1/2}(x)$. Along the way we will encounter several interesting new special functions and develop some new ideas, which are of independent interest, related to radial functions that are eigenfunctions of the Fourier transform, and their connections to a class of functions satisfying a symmetry property similar to (but weaker than) that satisfied by modular forms.

\section{Radial Fourier self-transforms on $\R^d$ and their construction from balanced functions}

\label{sec:rad-selftransforms}

A function $F:\R^d\to\R$ is called a \firstmention{radial function} if $F(\mathbf{x})$ depends only on the Euclidean norm $|\mathbf{x}|$. Given a radial function $F$, it is common to abuse notation slightly and write $F(\mathbf{x})=F(|x|)$, that is, we use the same symbol to denote the function on $\R^d$ and the function (on $[0,\infty)$) of the norm through which the original radial function can be computed. Conversely, given a function $F:[0,\infty)\to\R$ it will sometimes be convenient to regard $F$ as a radial function on $\R^d$ for some specified value of $d$.

Let $\mathcal{F}_d$ denote the Fourier transform on $\R^d$, with the normalization
\begin{equation*}
\mathcal{F}_d(F)(\mathbf{y}) = \int_{\R^d} F(\mathbf{x}) e^{-2 \pi i\langle \mathbf{y}, \mathbf{x}\rangle} \, d\mathbf{x}.
\end{equation*}
It is well-known that the $d$-dimensional Fourier transform $\mathcal{F}_d(F)$ of a radial function $F$ is also a radial function, and can be expressed as a Hankel transform, namely as
\begin{equation}
\label{eq:radial-fourier-d}
\mathcal{F}_d(F)(\rho) = 2\pi \rho^{1-d/2} \int_0^\infty F(r) r^{d/2}
J_{d/2-1}(2\pi r \rho)\, dr
\qquad (\rho \ge0),
\end{equation}
where 
\begin{equation*}
J_\alpha(z) = 
\sum_{n=0}^\infty \frac{(-1)^n}{n! \, \Gamma(n+\alpha+1)} \left(\frac{z}{2}\right)^{2n+\alpha}
\end{equation*}
denotes the Bessel function; see \cite[Sec.~B.5]{grafakos}. The cases $d=1$ and $d=3$ of \eqref{eq:radial-fourier-d} are particularly simple (and of relevance to us, as we shall see). In those cases, the standard identities
\begin{equation*}
J_{1/2}(x) = \sqrt{\frac{2}{\pi}} \frac{\sin(x)}{\sqrt{x}}, 
\qquad
J_{-1/2}(x) = \sqrt{\frac{2}{\pi}} \frac{\cos(x)}{\sqrt{x}} 
\end{equation*}
mean that \eqref{eq:radial-fourier-d} can be rewritten as
\begin{align}
\label{eq:radial-fourier-1d}
\mathcal{F}_1(F)(\rho) & = 
2 \int_0^\infty F(r) \cos(2\pi r\rho)\,dr,
\\
\label{eq:radial-fourier-3d}
\mathcal{F}_3(F)(\rho) & = 
\frac{2}{\rho} \int_0^\infty F(r) r \sin(2\pi r\rho)\,dr.
\end{align}
Note that the case $d=1$ is simply a cosine transform; indeed, a radial function on $\R^d$ for $d=1$ is the same as an even function.

A function $F:\R^d\to\R$ is called a \firstmention{(Fourier) self-transform} if $\mathcal{F}_d(F) = F$. The Gaussian $F(r) = e^{-\pi r^2}$ is an important example of a self-transform (in any dimension!) which is also a radial function. More generally, through a trivial rescaling operation we see that the Fourier transform of a \emph{scaled} Gaussian $e^{-\pi c r^2}$ (with $c>0$) is given by
\begin{equation*}
\mathcal{F}_d(e^{-\pi c r^2})(\rho) = c^{-d/2} e^{-\pi \rho^2/c}.
\end{equation*}
This relation provides a general means for constructing a large class of radial self-transforms in $\R^d$ by taking a weighted average of scaled Gaussians (or a ``scale mixture'' of Gaussians, in probabilistic language), using a weighting function in which the contribution of the Gaussian scaled by a given scalar $c$ is suitably matched by that coming from the reciprocal scalar $1/c$. This sort of construction can be found for example in works by Hardy and Titschmarsh \cite{hardy-titschmarsh} and Barndorff-Nielsen et al \cite{barndorff-nielsen}. As discussed by Cohn \cite{cohn}, the same construction in the case where the weighting functions are modular forms motivated recent progress on the sphere packing problem (see also \cite{viazovska}).

For our purposes, the weighting functions we will consider are related to modular forms but are more general. Let $\alpha\ge 0$. If a function $f:(0,\infty)\to\R$ satisfies the relation
\begin{equation*}
f\left(\frac{1}{x}\right) = x^\alpha f(x) \qquad (x>0),
\end{equation*}
we say that $f$ is a \firstmention{reciprocally balanced function of weight $\alpha$}. (Usually, for convenience we will omit the adverb ``reciprocally'' and simply refer to $f$ as a balanced function of weight $\alpha$.) The following  result is a trivial variant of the observation made in \cite[Eq.~(2.3)]{barndorff-nielsen}.

\begin{lem}[Constructing self-transforms from balanced functions]
\label{lem:recbal-selftrans}
Let $d\in\N$. Let $f(x)$ be a reciprocally balanced function of weight $2-d/2$, and define an associated function
\begin{equation}
\label{eq:gauss-trans}
F(r) = \int_0^\infty f(x) e^{-\pi x r^2}\,dx
\qquad (r>0).
\end{equation}
Then $F$, considered as a radial function on $\R^d$, is a Fourier self-transform, assuming its Fourier transform is well-defined.
\end{lem}

\begin{proof}
\begin{align*}
\mathcal{F}_d(F)(\rho) & = 
\int_0^\infty f(x) \mathcal{F}_d\left(e^{-\pi x r^2}\right)(\rho) \,dx
= \int_0^\infty f(x) x^{-d/2} e^{-\pi \rho^2/x}\,dx
\\ & =
\int_0^\infty f(1/y)
y^{d/2} e^{-\pi y\rho^2}\, \frac{dy}{y^2}
=
\int_0^\infty f(y)
e^{-\pi y\rho^2} \,dy = F(\rho).
\end{align*}

\end{proof}

Note that the relationship between $f(x)$ and $F(r)$ in \eqref{eq:gauss-trans} is simply that $F(r)$ is the Laplace transform $(\mathcal{L} F)(u)$ of $f(x)$, with the change of coordinates $u=\pi r^2$. It can also be interpreted as a group-theoretic convolution operation of the Gaussian function $r\mapsto e^{-\pi r^2}$ with the function $x\mapsto x^{-1}f(x^{-1})$ with respect to the multiplicative group structure on $\R_+$ equipped with the multiplicative Haar measure $\frac{dx}{x}$.

\section{The radial function $A(r)$ associated to $\omega(x)$}

\label{sec:rad-aofr}

We have encountered two balanced functions that play an important role in the study of the Riemann xi function: the function $\theta(x)$ (which is in fact a modular form), and the function $\omega(x)$ derived from it; both of those functions are balanced of weight $1/2$. We are mainly interested in $\omega(x)$, because it has better integrability properties and because the Riemann xi function is its Mellin transform. 
Define 
\begin{equation*}
A(r) = \int_0^\infty \omega(x) e^{-\pi x r^2}\,dx.
\end{equation*}
Since $\omega(x)$ is balanced of weight $1/2$, Lemma~\ref{lem:recbal-selftrans} implies that $A(r)$ is a Fourier self-transform when considered as a radial function on $\R^3$. The next result gives an explicit formula for $A(r)$.

\begin{prop}
\label{prop:aofr-explicit}
$A(r)$ is given explicitly by
\begin{equation}
\label{eq:aofr-explicit}
A(r) = \frac{d^2}{dr^2}\left( \frac{r}{4}\coth(\pi r)\right) =
-\frac{\pi}{2}\frac{1}{\sinh^2(\pi r)} + \frac{\pi^2 r}{2} \frac{\cosh(\pi r)}{\sinh^3(\pi r)}.
\end{equation}
\end{prop}

We give two short proofs of Proposition~\ref{prop:aofr-explicit}. As we remarked in the last paragraph of the previous section, this result has an obvious interpretation as a calculation of the Laplace transform of $\omega(x)$. Several closely related calculations have appeared in the literature; see \cite[pp.~23--24]{bellman}, \cite[eq.~(2.17)]{biane-pitman-yor}, \cite[pp.~168]{chung}, and especially eq.~(91) of \cite{pitman-yor-ejp}, which can be seen using the results of \cite{chung} to be equivalent to \eqref{eq:aofr-explicit}.

\begin{proof}[First proof of Proposition~\ref{prop:aofr-explicit}]
We have directly from the definitions that
\begin{align*}
\int_0^\infty \omega(x) e^{-\pi r^2 x}\,dx
& =
\sum_{n=1}^\infty \Bigg[
2\pi^2 n^4 \int_0^\infty x^2 e^{-\pi (r^2+n^2) x} \, dx
- 3\pi n^2 \int_0^\infty x e^{-\pi (r^2+n^2) x} \, dx
\Bigg]
\\ & =
\sum_{n=1}^\infty \left(\frac{4\pi^2 n^4}{\pi^3(r^2+n^2)^3}
- \frac{3\pi n^2}{\pi^2(r^2+n^2)^2}\right)
\\& =
\frac{1}{\pi}\sum_{n=1}^\infty
\left(
\frac{4n^4}{(r^2+\pi^2)^3}
-
\frac{3n^2}{(r^2+\pi^2)^2}
\right)
=
\frac{1}{2\pi}
\sum_{n=1}^\infty
\frac{d^2}{dr^2} \left(
\frac{r^2}{r^2+n^2}
\right)
\\ & =
\frac{d^2}{dr^2}\left(\frac{1}{4\pi}+\frac{1}{2\pi}\sum_{n=1}^\infty \frac{r^2}{r^2+n^2}\right)
=
\frac{d^2}{dr^2}\left( \frac{r}{4}\coth(\pi r)\right).
\end{align*}
Here, the last equality follows from the classical identity
\begin{equation*}
\pi \coth(\pi r) =  \frac{1}{r}+\sum_{n=1}^\infty \frac{2r}{r^2+n^2},
\end{equation*}
the partial fraction decomposition of the hyperbolic cotangent function \cite[p.~12]{andrews-askey-roy}. This proves the first equality in \eqref{eq:aofr-explicit}; the second equality is a trivial verification, which we leave to the reader.
\end{proof}

An alternative proof of Proposition~\ref{prop:aofr-explicit} is based on a calculation of the moments of $\omega(x)$, which seems worth recording separately.

\begin{lem}
\label{lem:omega-moments}
For $n\ge 0$ we have the relation
\begin{equation}
\label{eq:omega-moments}
\int_0^\infty \omega(x) x^n\,dx 
=
\frac{(-1)^n (4\pi)^{n+1} n!}{4(2n)!}B_{2n+2},
\end{equation}
where $(B_k)_{k=0}^\infty$ denotes the Bernoulli numbers.
\end{lem}

The relation \eqref{eq:omega-moments} is equivalent to the bottom-right entry in Table~1 of \cite[p.~442]{biane-pitman-yor} (see also \cite{pitman-yor-cjm} where several analogous formulas are derived).

\begin{proof}
Recalling Euler's formula
\begin{equation*}
\zeta(2m) = \frac{(-1)^{m-1}(2\pi)^{2m}}{2(2m)!}B_{2m},
\end{equation*}
we observe that for integer $n\ge 0$,
\begin{align*}
\int_0^\infty \omega(x) x^n\,dx 
& = \left[\int_0^\infty \omega(x) x^{s/2-1}\,dx \right]_{\raisebox{3pt}{$\Big|$}s=2n+2}
= \xi(2n+2) 
= \frac12(2n+2)(2n+1) \pi^{-n-1} \Gamma(n+1)\zeta(2n+2)
\\ &= \frac{n!}{2\pi^{n+1}}(2n+2)(2n+1)\frac{(-1)^n(2\pi)^{2n+2}}{2(2n+2)!}B_{2n+2} 
=
\frac{(-1)^n (4\pi)^{n+1} n!}{4(2n)!}B_{2n+2},
\end{align*}
as claimed.
\end{proof}

Another easy fact that we record is the Taylor expansion of the function on the right-hand side of \eqref{eq:aofr-explicit}.

\begin{lem}
We have the Taylor expansion
\begin{equation*}
\frac{d^2}{dr^2} \left(
\frac{r}{4}\coth\left(\pi r\right)
\right)
=
\sum_{n=0}^\infty \frac{(2\pi)^{2n+1} B_{2n+2}}{2(2n)!} r^{2n}
\qquad (|r|<1).
\end{equation*}
\end{lem}

\begin{proof}
Recall the standard generating function identity
\begin{equation*}
\frac{z}{2}\coth\left(\frac{z}{2}\right) =
\sum_{n=0}^\infty \frac{B_{2n}}{(2n)!}z^{2n} \qquad \left(|z|<2\pi\right)
\end{equation*}
(see \cite[p.~12]{andrews-askey-roy}). Making the substitution $z=2\pi r$ and differentiating twice gives the result.
\end{proof}

\begin{proof}[Second proof of Proposition~\ref{prop:aofr-explicit}]
From the above two lemmas we see that, calculating formally at least,
\begin{align*}
\int_0^\infty \omega(x) e^{-\pi r^2 x}\,dx 
& =
\int_0^\infty \omega(x) 
\sum_{n=0}^\infty \frac{(-1)^n \pi^n r^{2n}}{n!} x^n\,dx 
\\ & =
\sum_{n=0}^\infty \frac{(-1)^n \pi^n r^{2n}}{n!}
\int_0^\infty \omega(x) x^n\,dx 
=
\sum_{n=0}^\infty \frac{(2\pi)^{2n+1} B_{2n+2}}{2(2n)!} r^{2n}
=
\frac{d^2}{dr^2} \left(
\frac{r}{4}\coth\left(\pi r\right)
\right).
\end{align*}
To justify this rigorously, note that, by \eqref{eq:omegax-asym-xinfty}--\eqref{eq:omegax-asym-xzero}, the function $\omega(x)\exp(-\pi r^2 x)$ is absolutely integrable on $[0,\infty)$ for any \emph{complex} number $r$ satisfying $|r|<1$.
We have thus established the identity \eqref{eq:aofr-explicit} for those values of $r$, and, since $A(r)$ can be regarded as an analytic function of a complex variable $r$ on some open set containing the positive real axis, the result follows for general $r\ge 0$ by analytic continuation.
\end{proof}

\section{An orthonormal basis for radial self-transforms}

\label{sec:rad-orthradial}

Recall that the Laguerre polynomials $L_n^\alpha(x)$ are, for fixed $\alpha>-1$, a family of orthogonal polynomials with respect to the weight function $e^{-x} x^\alpha$ on $[0,\infty)$. Their main properties are summarized in 
\secref{sec:orth-laguerre}. 
We can use them to construct functions suitable for representing radial functions on $\R^d$ by defining
\begin{equation*}
G_n^{(d)}(r) = e^{-\pi r^2} L_n^{d/2-1}(2\pi r^2)
\qquad (r>0).
\end{equation*}
One main reason why this is a useful definition is that the $G_n^{(d)}$ satisfy the orthogonality relation
\begin{equation}
\label{eq:gnd-orthogonality}
\int_0^\infty G_m^{(d)}(r) G_n^{(d)}(r) r^{d-1}\,dr 
= \frac{\Gamma(n+d/2)}{2(2\pi)^{d/2} n!} \delta_{m,n},
\end{equation}
which, as the reader can verify, is immediate from the standard orthogonality relation \eqref{eq:laguerre-orthogonality} for the Laguerre polynomials, by a change of variables. Equivalently, recalling that we are thinking of the $G_n^{(d)}$ as functions on $\R^d$, we can write this as an orthogonality relation with respect to the ordinary Lebesgue measure on $\R^d$ by interpreting the integral on the left-hand side of \eqref{eq:gnd-orthogonality} as an integral in polar coordinates, which gives the equivalent relation
\begin{equation*}
\int_{\R^d} G_m^{(d)}(\mathbf{x}) G_n^{(d)}(\mathbf{x}) \, d\mathbf{x}
= \kappa_{d,n} \delta_{m,n},
\end{equation*}
where
\begin{equation*}
\kappa_{d,n} = d\cdot V_d \frac{\Gamma(n+d/2)}{2(2\pi)^{d/2} n!}
= \frac{d\cdot \Gamma(n+d/2)}{2^{1+d/2} n! \Gamma(1+d/2)},
\end{equation*}
and $V_d = \frac{\pi^{d/2}}{\Gamma\left(\frac{d}{2}+1\right)}$ denotes the volume of the unit ball in $\R^d$. 

The orthogonal family $(G_n^{(d)})_{n=0}^\infty$ is especially useful for representing radial \emph{self-transforms} such as the function $A(r)$, thanks to the following result.

\begin{thm}[{\cite[Secs.~4.20,~4.23]{lebedev}}]
The functions $G_n^{(d)}(r)$, considered as radial functions on $\R^d$, form an orthogonal basis of the subspace $L_{\textnormal{rad}}^2(\R^d)$ of $L^2(\R^d)$ consisting of square-integrable radial functions. 
Moreover, this orthogonal basis diagonalizes the radial Fourier transform \eqref{eq:radial-fourier-d}; more precisely, we have the property
\begin{equation*}
\mathcal{F}_d\left(G_n^{(d)}\right) = (-1)^n G_n^{(d)}.
\end{equation*}
\end{thm}

The theorem implies in particular that the even-indexed functions $G_{2n}^{(d)}(r)$ form an orthogonal basis for the subspace of $L_{\textnormal{rad}}^2(\R^d)$ consisting of square-integrable radial Fourier self-transforms. This gives a new way of representing radial self-transforms as linear combinations of the form $\sum_n \gamma_n G_{2n}^{(d)}(r)$. Thus, we have now shown two general ways to construct radial Fourier self-transforms: first, as weighted mixtures of scaled Gaussians, and second, as linear combinations of the basis elements $G_{2n}^{(d)}$. As the next result starts to illustrate, the interplay between these two approaches turns out to be very fruitful.

\begin{prop}
The $f_n$-expansion coefficients $c_n$ defined in~\eqref{eq:def-fn-coeffs} can be alternatively expressed as
\begin{equation}
\label{eq:fn-coeffs-radialform}
c_n = 
\frac{8\sqrt{2} \pi n!}{(3/2)_n}
\int_0^\infty A(r) r^2 G_n^{(3)}(r)\,dr.
\end{equation}
\end{prop}

\begin{proof}
We start by evaluating a simpler integral, namely, for integer $m\ge1$,
\begin{align*}
\int_0^\infty A(r) e^{-\pi r^2} r^{2m}\,dr
& =
\int_0^\infty \left(\int_0^\infty \omega(x) e^{-\pi x r^2}\,dx \right) e^{-\pi r^2} r^{2m}\,dr
\\ & =
\int_0^\infty \omega(x) \left(\int_0^\infty e^{-\pi (x+1) r^2}r^{2m} \,dr \right) \,dx
\\ & =
\int_0^\infty \omega(x) \left(\int_0^\infty e^{-u} \left(\frac{\sqrt{u}}{\sqrt{\pi (x+1)}}\right)^{2m} \frac{du}{2\sqrt{\pi (x+1)u}}\right)\,dx
\\ & =
\frac{1}{2 \pi^{m+1/2}} \Gamma\left(m+\frac12\right)\int_0^\infty \omega(x)
\frac{1}{(x+1)^{m+1/2}}
\,dx
\\ &
=
\frac{(3/2)_{m-1}}{4 \pi^{m}} \int_0^\infty \omega(x)
\frac{1}{(x+1)^{m+1/2}}\,dx.
\end{align*}
Now using the formula \eqref{eq:laguerre-explicit} for the Laguerre polynomials, we have that
\begin{align*}
\int_0^\infty & A(r) r^2 G_n^{(3)}(r)\,dr
=
\int_0^\infty A(r) r^2 e^{-\pi r^2} L_n^{1/2}(2\pi r^2)\,dr
\\ & =
\int_0^\infty A(r) e^{-\pi r^2} \sum_{k=0}^n 
\frac{(-1)^k}{k!} \binom{n+1/2}{n-k} (2\pi)^k r^{2k+2}
\,dr
\\ & =
\sum_{k=0}^n \frac{(-2\pi)^k}{k!}\binom{n+1/2}{n-k}   \int_0^\infty A(r) e^{-\pi r^2} 
 r^{2k+2}
\,dr
\\ & =
\sum_{k=0}^n \frac{(-2\pi)^k}{k!}\binom{n+1/2}{n-k} 
\frac{(3/2)_k}{4\pi^{k+1}}
\int_0^\infty \omega(x) (x+1)^{-(k+3/2)}
\,dx
\\ & =
\int_0^\infty \omega(x) 
\left(
\sum_{k=0}^n \frac{(-2\pi)^k}{k!}
\binom{n+1/2}{n-k} \frac{(3/2)_k}{4\pi^{k+1}} (x+1)^{-(k+3/2)}
\right)\,
dx
\\ & =
\frac{1}{4\pi} \int_0^\infty \frac{\omega(x)}{(x+1)^{n+3/2}}
\left(
\sum_{k=0}^n \frac{(-2)^k}{k!}
\binom{n+1/2}{n-k} (3/2)_k (x+1)^{n-k}
\right)
dx.
\end{align*}
Noting the simple relation 
$\frac{(3/2)_k}{k!}\binom{n+1/2}{n-k} 
=
\frac{(3/2)_n}{n!} \binom{n}{k},
$
we see that 
the sum inside the integral simplifies as
\begin{align*}
\sum_{k=0}^n \frac{(3/2)_k}{k!}
\binom{n+1/2}{n-k} \cdot & (-2)^k (x+1)^{n-k}
=
\frac{(3/2)_n}{n!} \sum_{k=0}^n \binom{n}{k} (-2)^k (x+1)^{n-k}
\\ & =
\frac{(3/2)_n}{n!} (-2+(x+1))^n =
\frac{(3/2)_n}{n!} (x-1)^n.
\end{align*}
Thus, we get finally that
\begin{equation*}
\int_0^\infty A(r) r^2 G_n^{(3)}(r)\,dr
= \frac{(3/2)_n}{4\pi n!} 
\int_0^\infty \frac{\omega(x)}{(x+1)^{3/2}} \left(\frac{x-1}{x+1}\right)^n\,dx
= 
\frac{(3/2)_n}{8\sqrt{2} \pi n!}  c_n,
\end{equation*}
which proves \eqref{eq:fn-coeffs-radialform}.
\end{proof}

We have set the stage for one of the central results of this chapter.

\begin{thm}[Expansion of $A(r)$ in the orthogonal family $G_n^{(3)}(r)$]
\label{thm:aofr-selftrans-expansion}
The radial function $A(r)$ has the series expansion
\begin{equation}
\label{eq:aofr-selftrans-expansion}
A(r) = \sum_{n=0}^\infty c_n G_n^{(3)}(r),
\end{equation}
with $c_n$ given by \eqref{eq:def-fn-coeffs}, \eqref{cor:fn-coeff-innerproduct} and \eqref{eq:fn-coeffs-radialform}. The series in \eqref{eq:aofr-selftrans-expansion} converges pointwise and in $L^2(\R^3)$.
\end{thm}

\begin{proof}
The equation \eqref{eq:aofr-selftrans-expansion} is simply the Fourier expansion of the (clearly square-integrable) function $A(r)$, considered as a radial function on $\R^3$, in the orthogonal basis $G_n^{(3)}(r)$. The fact that the coefficients $c_n$ are the Fourier coefficients follows from \eqref{eq:gnd-orthogonality}~and~\eqref{eq:fn-coeffs-radialform} (together with the simple equality $\frac{(3/2)_n}{8\sqrt{2}\pi n!} = \frac{\Gamma(n+3/2)}{2(2\pi)^{3/2}n!}$ relating the normalization constants appearing in those two equations). The convergence in $L^2(\R^3)$ is immediate, and pointwise convergence follows from standard theorems about expansions of a function in Laguerre polynomials; see \cite[p.~88]{lebedev}.
\end{proof}

\section{Constructing new balanced functions from old}

\label{sec:rad-newrecbal}

Next, we show a simple operation that produces a balanced function of weight $2-\alpha$ starting with a balanced function of weight $\alpha$.

\begin{lem}
\label{lem:new-from-old}
Let $f:(0,\infty)\to \R$ be balanced of weight $0<\alpha<2$. Then the function $g:(0,\infty)\to\R$ defined by
\begin{equation}
\label{eq:gen-stieltjes}
g(x) = \int_0^\infty \frac{f(u)}{(u+x)^{2-\alpha}}\,du
\end{equation}
is a balanced function of weight $2-\alpha$.
\end{lem}

\begin{proof}
\begin{align*}
g\left(\frac{1}{x}\right) 
& = \int_0^\infty \frac{f(u)}{(u+1/x)^{2-\alpha}}\,du
= 
\int_0^\infty \frac{f(1/v)}{\left(\frac{1}{v}+\frac{1}{x}\right)^{2-\alpha}}\,\frac{dv}{v^2}
\\ & =
\int_0^\infty \frac{(vx)^{2-\alpha} v^\alpha f(v)}{(v+x)^{2-\alpha}}\,\frac{dv}{v^2}
 =
x^{2-\alpha} \int_0^\infty \frac{f(v)}{(v+x)^{2-\alpha}}\,dv
= x^{2-\alpha} g(x).
\end{align*}
\end{proof}

The integral transform in \eqref{eq:gen-stieltjes} is known as a \firstmention{generalized Stieltjes transform}. Its properties are discussed in \cite{karp-prilepkina, schwarz}. Lemma~\ref{lem:new-from-old} appears to be new.

\section{The functions $\nu(x)$ and $B(r)$}

\label{sec:rad-bofr}

Applying the construction of Lemma~\ref{lem:new-from-old} to our balanced function $\omega(x)$ (with $\alpha=1/2$), we define
\begin{equation*}
\nu(x) = \int_0^\infty \frac{\omega(u)}{(u+x)^{3/2}}\,du,
\end{equation*}
and note that $\nu(x)$ is a balanced function of weight $3/2$. Now define (as in \eqref{eq:gauss-trans}) the associated radial function
\begin{equation*}
B(r) = \int_0^\infty \nu(x) e^{-\pi x r^2}\,dx,
\end{equation*}
We will have much more to say about these functions below. Among other results, in \secref{sec:rad-properties-nu} we will show that $B(r)$ is given by the explicit formula $B(r) = 1-2r \psi'(r) - r^2 \psi''(r)$, where $\psi(x)$ is the digamma function. The function $\nu(x)$ and a closely related function $\tilde{\nu}(t) = \frac{1}{\sqrt{1+t}}\nu\left(\frac{1-t}{1+t}\right)$ will play an important role in our understanding of the $g_n$-expansion \eqref{eq:gn-expansion-intro} of the Riemann xi function, which we will show in Chapter~\ref{ch:gn-expansion} arises naturally from the expansion of $\tilde{\nu}(t)$ in the Chebyshev polynomials in the second kind.
The same function $\tilde{\nu}(t)$ also has an interpretation as a generating function for the $f_n$-expansion coefficients $c_n$; see \secref{sec:rad-properties-nu}.

\section{Some Mellin transform computations}

\label{sec:rad-mellintransforms}

In this section we compute the Mellin transforms of the functions $A(r)$, $B(r)$, $\nu(x)$, and derive two separate Mellin transform representations for the polynomials $f_n$. This will set the ground for the alternative derivation of the $f_n$-expansion mentioned at the beginning of the chapter, which will be given in the next section, and for additional results in Sections~\ref{sec:rad-centered-recbal}--\ref{sec:rad-properties-nu} and Chapter~\ref{ch:gn-expansion} that will shed additional light on the significance of the new functions we introduced.

\begin{prop}
\label{prop:aofr-mellin}
The Mellin transform of $A(r)$ is given by
\begin{equation}
\label{eq:aofr-mellin}
\int_0^\infty A(r) r^{s-1}\,dr =
\frac{1}{2(2\pi)^{s-1}} (s-1)(s-2) \Gamma(s-1)\zeta(s-1)
\qquad (\re{s} > 0).
\end{equation}
\end{prop}

\begin{proof}[First proof]
Denote $F(r) = \frac{r}{4}(\coth(\pi r)-1) =
\frac12 \frac{r}{e^{2\pi r}-1}$. Using integration by parts twice, we have
\begin{align*}
\int_0^\infty A(r) r^{s-1}\,dr &=
\int_0^\infty F''(r)
r^{s-1}\,dr
=
F'(r) r^{s-1}\Big|_{r=0}^{r=\infty}
- (s-1)\int_0^\infty 
F'(r)  r^{s-2}\,dr 
\\ &=
0- (s-1)F(r)r^{s-2}\Big|_{r=0}^{r=\infty}
+ (s-1)(s-2)\int_0^\infty 
F(r)
 r^{s-3}\,dr 
\\ &=
0+0+\frac12(s-1)(s-2) \int_0^\infty
\frac{1}{e^{2\pi r}-1} r^{s-2}\,dr
\\ &= 
\frac12(s-1)(s-2) (2\pi)^{-(s-1)} \Gamma(s-1)\zeta(s-1).
\end{align*}
\end{proof}

\begin{proof}[Second proof]
\begin{align}
\label{eq:aofr-explicit-secondproof}
\int_0^\infty A(r) r^{s-1}\,dr &=
\int_0^\infty \int_0^\infty \omega(x)e^{-\pi x r^2}\,dx \, r^{s-1}\,dr 
\\ &=
\int_0^\infty \omega(x)
\int_0^\infty e^{-\pi x r^2}r^{s-1}\,dr \, \,dx
\nonumber \\ &=
\int_0^\infty
\omega(x) 
\left( \int_0^\infty e^{-u} \left(\frac{\sqrt{u}}{\sqrt{\pi x}}\right)^{s-1}\, \frac{du}{2 \sqrt{\pi x u}} \right)
\,dx
\nonumber \\ &=
\frac{1}{2\sqrt{\pi}} \pi^{-(s-1)/2} \Gamma\left(\frac{s}{2}\right)
\int_0^\infty \omega(x) x^{-s/2}\,dx
\nonumber \\ &=
\frac12 \pi^{-s/2} \Gamma\left(\frac{s}{2}\right)
\xi(2-s)
=
\frac12 \pi^{-s/2} \Gamma\left(\frac{s}{2}\right)
\xi(s-1)
\nonumber \\ &=
\frac12 \pi^{-s/2} \Gamma\left(\frac{s}{2}\right)
\cdot \frac12 (s-1)(s-2) \pi^{-(s-1)/2} \Gamma\left(\frac{s-1}{2}\right)\zeta(s-1)
\nonumber \\ & =
\frac{1}{2(2\pi)^{s-1}} (s-1)(s-2)
\left( 
\frac{2^{s-2}}{\sqrt{\pi}}\Gamma\left(\frac{s-1}{2}\right)
\Gamma\left(\frac{s}{2}\right)
\right) 
\zeta(s-1).
\nonumber
\end{align}
This coincides with the right-hand side of \eqref{eq:aofr-mellin}
by the duplication formula \cite[p.~22]{andrews-askey-roy}
\begin{equation}
\label{eq:duplication-formula}
\Gamma(z)\Gamma(z+1/2) = \sqrt{\pi} 2^{1-2z} \Gamma(2z).
\end{equation}
\end{proof}

\begin{prop}
\label{prop:mellin-nu}
The Mellin transform of $\nu(x)$ is given by
\begin{equation}
\label{eq:nu-mellin-trans}
\int_0^\infty \nu(x) x^{s/2-1}\,dx =
\frac{2}{\sqrt{\pi}} 
\Gamma\left(\frac{s}{2}\right)\Gamma\left(\frac32-\frac{s}{2}\right) \xi(s-1)
\qquad (0 < \re{s} < 3).
\end{equation}
(Here, as with known formulas such as \eqref{eq:riemannxi-mellintrans}, it seems more aesthetically pleasing to use $s/2$ as the Mellin transform variable instead of $s$.)
\end{prop}

\begin{proof}
A textbook Mellin transform computation is the result that
for $\alpha>0$,
\begin{equation}
\label{eq:mellin-textbook-ex}
\int_0^\infty \frac{t^{s-1}}{(1+t)^\alpha}\,dt
= \frac{\Gamma(s)\Gamma(\alpha-s)}{\Gamma(\alpha)}
\qquad (0 < \re(s) < \alpha),
\end{equation}
since the integral on the left transforms into a beta integal $\int_0^1 u^{s-1}(1-u)^{\alpha-s-1}\,du$ upon making the substitution $t=u/(1-u)$.
Using this fact, we write
\begin{align*}
\int_0^\infty \nu(x) x^{s/2-1}\,dx &=
\int_0^\infty \int_0^\infty \frac{\omega(u)}{(u+x)^{3/2}}\,du \, x^{s/2-1}\,dx
=
\int_0^\infty \omega(u) \left( \int_0^\infty \frac{x^{s/2-1}}{(u+x)^{3/2}} \,dx \right)\,du
\\ &=
\int_0^\infty \frac{\omega(u)}{u^{3/2}} \int_0^\infty \frac{t^{s/2-1}}{(1+t)^{3/2}}\,dt\,u^{s/2}\,du
=
\int_0^\infty \omega(u) u^{s/2-3/2}\,du \cdot \frac{\Gamma\left(\frac{s}{2}\right)
\Gamma\left(\frac32-\frac{s}{2}\right)}{\Gamma\left(\frac32\right)}
\\ &= \frac{2}{\sqrt{\pi}} 
\Gamma\left(\frac{s}{2}\right)\Gamma\left(\frac32-\frac{s}{2}\right) \xi(s-1).
\end{align*}
Note that the assumption $0<\re(s)<3$ ensures that the double integral following the first equality is absolutely convergent (as can be seen by repeating the same chain of equalities with $s$ replaced by $\re(s)$), which justifies the change in the order of integration.
\end{proof}

\begin{prop}
The Mellin transform of $B(r)$ is given by
\begin{equation}
\label{eq:bofr-mellin}
\int_0^\infty B(r) r^{s-1}\,dr =
\frac{(2\pi)^{1-s}}{2\sin(\pi s/2)}
 s(s-1) \Gamma(s) \zeta(s)
\qquad (0 < \re{s} < 2).
\end{equation}
\end{prop}

\begin{proof}
Repeat the calculation in \eqref{eq:aofr-explicit-secondproof}, replacing $A(r)$ by $B(r)$ and replacing the Mellin transform of $\omega(x)$ with the Mellin transform of $\nu(x)$. We omit the details.
\end{proof}

The following result appears to be known, although its origins and proof seem difficult to trace (it is implicit in the results of section~1 of \cite{bump-etal}, and a more explicit version is mentioned without proof in \cite[eq.~(1)]{kuznetsov1} and \cite[p.~829]{kuznetsov2}). We include it along with a short, self-contained proof.

\begin{prop}[First Mellin transform representation of $f_n$]
\label{prop:fn-mellintrans1}
The Mellin transform of $G_n^{(3)}(r)$ is given by
\begin{equation}
\label{eq:fn-mellintrans1}
\int_0^\infty G_n^{(3)}(r) r^{s-1}\,dr =
\frac12 (-i)^n \pi^{-s/2} \Gamma\left(\frac{s}{2}\right)
f_n\left(\frac{1}{2i}\left(s-\frac{3}{2}\right)\right)\qquad (\re{s} > 0)
\end{equation}
\end{prop}

\begin{proof}
Again using the explicit formula \eqref{eq:laguerre-explicit} for the Laguerre polynomials, we have that
\begin{align*}
\int_0^\infty G_n^{(3)}(r) r^{s-1}\,dr 
&=
\int_0^\infty e^{-\pi r^2} L_n^{1/2}(2\pi r^2) r^{s-1}\,dr 
\\ &=
\sum_{k=0}^n \frac{(-1)^k}{k!} \binom{n+1/2}{n-k} (2\pi)^k \int_0^\infty e^{-\pi r^2} r^{2k} r^{s-1}\,dr 
\\
&=
\sum_{k=0}^n \frac{(-1)^k}{k!} \binom{n+1/2}{n-k} (2\pi)^k 
\int_0^\infty e^{-u} \left(\frac{\sqrt{u}}{\sqrt{\pi}}\right)^{s+2k-1}\, \frac{du}{2\sqrt{\pi u}}
\\ &=
\frac12 \pi^{-s/2}\sum_{k=0}^n \frac{(-2)^k}{k!} \binom{n+1/2}{n-k} 
\Gamma\left(\frac{s}{2}+k\right)
\\ &=
\frac12 \pi^{-s/2} \Gamma\left(\frac{s}{2}\right)
\sum_{k=0}^n \frac{(-2)^k}{k!} \binom{n+1/2}{n-k} 
\frac{s}{2}
\left( \frac{s}{2} + 1 \right)
\cdots 
\left( \frac{s}{2} + k-1 \right)
\\ &=
\frac12 \pi^{-s/2} \Gamma\left(\frac{s}{2}\right)
\sum_{k=0}^n (-2)^k \binom{n+1/2}{n-k} 
\binom{s/2+k-1}{k}
\\ &=
\frac12 \pi^{-s/2} \Gamma\left(\frac{s}{2}\right)
(-i)^n f_n\left(\frac{1}{2i}\left(s-\frac{3}{2}\right)\right),
\end{align*}
where in the last step we used the formula \eqref{eq:fn-explicit4} for $f_n(x)$. This gives the claimed formula.
\end{proof}

If we replace $s$ with the variable $t$ where $s=2\left(\frac34+it\right)$, \eqref{eq:fn-mellintrans1} can be rewritten in the form
\begin{equation}
\label{eq:fn-mellintrans2}
\int_0^\infty G_n^{(3)}(r) r^{\frac12+2it}\,dr =
\frac12 (-i)^n \pi^{-\frac34-it} \Gamma\left(\frac34+it\right)
f_n\left(t\right)\qquad 
\end{equation}
which is a useful integral representation for $f_n(t)$. The next result, which we have not found in the literature, gives yet another representation for $f_n(t)$ in terms of a Mellin transform.

\begin{prop}[Second Mellin transform representation of $f_n$]
\label{prop:fn-mellintrans3}
We have the relation
\begin{align}
\label{eq:fn-mellintrans-scoord}
\int_0^\infty &
\frac{1}{(x+1)^{3/2}} 
\left(\frac{x-1}{x+1}\right)^n x^{s-1}\,dx
\\ \nonumber & =
i^n\frac{2 n!}{\sqrt{\pi} (3/2)_n}
\Gamma(s)\Gamma\left(\frac32-s\right)
 f_n\left(\frac{1}{i}\left(s-\frac34\right)\right)
 \qquad \left(0<\re(s)<\frac32 \right).
\end{align}
Equivalently, with the substitution $s=\frac34+it$, this can be written in the form
\begin{align}
\label{eq:fn-mellintrans-tcoord}
\int_0^\infty 
\frac{1}{(x+1)^{3/2}} &
\left(\frac{x-1}{x+1}\right)^n x^{-\frac14+it}\,dx
=
i^n \frac{2 n!}{\sqrt{\pi}(3/2)_n} \Gamma\left(\frac34+it\right)
\Gamma\left(\frac34-it\right)
f_n(t).
\end{align}
\end{prop}

It is interesting to compare this result with Proposition~\ref{prop:gn-mellintrans} (page~\pageref{prop:gn-mellintrans}), which gives an analogous Mellin transform representation for the polynomials $g_n$.

\begin{proof}
Recalling~\eqref{eq:mellin-textbook-ex}, we have
\begin{align*}
\int_0^\infty & 
\frac{1}{(x+1)^{3/2}}
\left(\frac{x-1}{x+1}\right)^n
x^{s-1}\,dx
=
\int_0^\infty \frac{\sum_{k=0}^n (-1)^{n-k} \binom{n}{k} x^k}{(x+1)^{n+3/2}} x^{s-1}\,dx
\\ &
=
\sum_{k=0}^n (-1)^{n-k} \binom{n}{k}
\int_0^\infty \frac{x^{k+s-1}}{(x+1)^{n+3/2}} \,dx
=
\sum_{k=0}^n (-1)^{k-k} \binom{n}{k}
\mellin{\frac{1}{(x+1)^{n+3/2}}}(k+s)
\\& =
\sum_{k=0}^n (-1)^{n-k} \binom{n}{k}
\frac{\Gamma(k+s)\Gamma\left(n-k+\frac32-s\right)}{\Gamma\left(n+\frac32\right)}
\\& =
\frac{\Gamma(s)\Gamma\left(\frac32-s\right)}{\Gamma\left(n+\frac32\right)}
\sum_{k=0}^n (-1)^{n-k} \binom{n}{k}
\left( s(s+1)\cdots (s+k-1) \right)
\\ & \qquad\qquad\qquad\qquad\qquad \times
\left( \left(\frac32-s\right)\left(\frac32-s+1\right)\cdots \left(\frac32-s+n-k-1\right) \right)
\\ & =
\frac{\Gamma(s)\Gamma\left(\frac32-s\right)}{\Gamma\left(n+\frac32\right)}
\sum_{k=0}^n (-1)^{n-k}
\binom{n}{k} 
\times (-1)^k k!\binom{-s}{k}
\times (-1)^{n-k} (n-k)!\binom{s-3/2}{n-k}
\\ & =
\frac{\Gamma(s)\Gamma\left(\frac32-s\right)}{\Gamma\left(n+\frac32\right)}
n! \sum_{k=0}^n (-1)^k
\binom{-s}{k}
\binom{s-3/2}{n-k}
\\ & =
\frac{2 n!}{\sqrt{\pi} (3/2)_n}
\Gamma(s)\Gamma\left(\frac32-s\right)
\sum_{k=0}^n (-1)^k
\binom{-s}{k}
\binom{s-3/2}{n-k}
\\ & =
\frac{2 n!}{\sqrt{\pi} (3/2)_n}
\Gamma(s)\Gamma\left(\frac32-s\right)
i^n f_n\left(\frac{1}{i}\left(s-\frac34\right)\right),
\end{align*}
using \eqref{eq:fn-explicit4} and the symmetry $f_n(-x)= (-1)^n f_n(x)$ in the last step.
\end{proof}

\section{Alternative approach to the $f_n$-expansion of $\Xi(t)$}

\label{sec:rad-alt-fnexp}

The results of the preceding sections make it possible to conceive of a parallel approach to the development of the $f_n$-expansion of $\Xi(t)$ that is distinct from the approach taken in Chapter~\ref{ch:fn-expansion}, and relies entirely on the elementary function $A(r)$ rather than on its more sophisticated companion function $\omega(x)$ (or even the function $\theta(x)$ from which $\omega(x)$ is derived). 

The idea is to define $A(r)$ directly using \eqref{eq:aofr-explicit}, and then to use \eqref{eq:fn-coeffs-radialform} as the definition of the coefficients $c_n$. Theorem~\ref{thm:aofr-selftrans-expansion} and its proof remain valid. Propositions~\ref{prop:aofr-mellin}~and~\ref{prop:fn-mellintrans1} also remain valid (note that, conveniently, the first proof we gave for Proposition~\ref{prop:aofr-mellin} does not rely on the connection between $A(r)$ and $\omega(x)$). We now consider what happens when we formally take the Mellin transform of both sides of \eqref{eq:aofr-selftrans-expansion}; the result is the relation
\begin{align*}
\frac{1}{2(2\pi)^{s-1}} &(s-1)(s-2) \Gamma(s-1)\zeta(s-1)
=
\sum_{n=0}^\infty c_n \cdot
\frac12 (-i)^n \pi^{-s/2} \Gamma\left(\frac{s}{2}\right)
f_n\left(\frac{1}{2i}\left(s-\frac{3}{2}\right)\right).
\end{align*}
After substituting $s+1$ in place of $s$, the left-hand side starts to resemble $\xi(s)$, so, rearranging terms judiciously, we see that
\begin{align*}
\xi(s) & =
\frac12 s(s-1)\pi^{-s/2} \Gamma\left(\frac{s}{2}\right)
\zeta(s)
\\ & =
\left(\frac{\Gamma\left(\frac{s}{2}\right)}{\Gamma(s)}2^s \pi^{s/2} \right) \cdot \left(\frac{1}{2(2\pi)^{s}} s(s-1) \Gamma(s)\zeta(s)\right)
\\ & =
\left(\frac{\Gamma\left(\frac{s}{2}\right)}{\Gamma(s)}2^s \pi^{s/2} \right)
\sum_{n=0}^\infty c_n 
\cdot \frac12 
(-i)^n \pi^{-(s+1)/2} \Gamma\left(\frac{s+1}{2}\right)
f_n\left(\frac{1}{2i}\left(s-\frac{1}{2}\right)\right)
\\ & =
\sum_{n=0}^\infty
(-i)^n c_n
\left(
\frac{\Gamma\left(\frac{s}{2}\right)}{\Gamma(s)}2^s \pi^{s/2}
\times \frac12
\pi^{-(s+1)/2} \Gamma\left(\frac{s+1}{2}\right)
\right)
f_n\left(\frac{1}{2i}\left(s-\frac{1}{2}\right)\right)
\\ & =
\sum_{n=0}^\infty
(-i)^n c_n
f_n\left(\frac{1}{2i}\left(s-\frac{1}{2}\right)\right),
\end{align*}
with the cancellation in the last step following from the duplication formula \eqref{eq:duplication-formula}. Setting $s=\frac12+it$, we again recover the $f_n$-expansion \eqref{eq:fn-expansion}---a satisfying result.

Note that the above approach, while it is rather intuitive and provides useful insight into the meaning and significance of the $f_n$-expansion, nonetheless suffers from the drawback that using \eqref{eq:fn-coeffs-radialform} as the definition of the coefficients $c_n$ does not make the positivity property of the even-indexed coefficients evident, nor can we see a way to understand their asymptotic behavior directly from this definition. For this reason, the approach we originally took in Chapter~\ref{ch:fn-expansion} seems preferable as the initial basis for developing the theory of the $f_n$-expansion.

\section{Centered versions of balanced functions}

\label{sec:rad-centered-recbal}

If $f$ is a balanced function of weight $\alpha$, denote
\begin{equation}
\label{eq:balanced-centered-def}
\tilde{f}(t) = \frac{1}{(1+t)^\alpha} f\left(\frac{1-t}{1+t}\right) \qquad (|t|<1).
\end{equation}
We refer to $\tilde{f}$ as the \firstmention{centered version of $f$}.

\begin{lem}
The centered version $\tilde{f}$ of a balanced function of weight $\alpha$ is an even function.
\end{lem}

\begin{proof} It is trivial to verify that the equation $\tilde{f}(-t)=\tilde{f}(t)$ is algebraically equivalent to the functional equation \eqref{eq:balanced-centered-def}.
\end{proof}

We remark that, in the context of the modular form $\theta(x)$, the notion of its centered version was studied in \cite{romik} (see also \cite[Sec.~5.1]{zagier123} where a similar change of variables for modular forms was discussed). 

One illustration of the relevance of the above definition is the following simple fact concerning the centered version $\tilde{\omega}(t)$ of $\omega(x)$.

\begin{prop}
\label{prop:cn-as-moments}
The coefficients $c_n$ can be alternatively expressed as 
\begin{equation}
\label{eq:fn-coeffs-asmoments}
c_n = 2 \int_{-1}^1 \tilde{\omega}(u) u^n\,du
\qquad (n\ge 0).
\end{equation}
\end{prop}

\begin{proof}
This is a trivial reinterpretation of the defining formula
\eqref{eq:def-fn-coeffs} for $c_n$: simply make the change of variables $u=\frac{x-1}{x+1}$, and check that we get precisely \eqref{eq:fn-coeffs-asmoments}
\end{proof}

See \secref{sec:omega-tilde-properties} for further discussion of $\tilde{\omega}(t)$. The centered version $\tilde{\nu}(t)$ of $\nu(x)$ also has an important role to play in connection with both the $f_n$-expansion and the $g_n$-expansion of the Riemann xi function, as we shall see in the next section and later in Chapter~\ref{ch:gn-expansion}.

\section{Properties of $\nu(x)$ and $B(r)$}

\label{sec:rad-properties-nu}

\begin{prop}
\label{eq:nuofx-properties}
The function $\nu(x)$ has the following properties:

(i) $\nu(x)$ is positive and monotone decreasing.

(ii) $\nu(0)=\frac{\pi}{6}$.

(iii) $\nu(x)$ has the asymptotic expansions
\begin{align}
\nu(x) & =
\sum_{k=0}^n
\frac{(2k+1) \pi^{k+1} B_{2k+2}}{k!} x^k
+ O(x^{n+1})
\qquad\quad\,\ \ \qquad \textrm{as }x\to 0,
\label{eq:nutilde-asym1}
\\
\nu(x) & =
\sum_{k=0}^n
\frac{(2k+1) \pi^{k+1} B_{2k+2}}{k!} x^{-k-3/2}
+ O(x^{-n-5/2})
\qquad \textrm{as }x\to \infty,
\label{eq:nutilde-asym2}
\end{align}
\end{prop}

\begin{proof}
(i) is immediate from the definition of $\nu(x)$. (ii) is a special case of (iii), but also follows more directly from the definition of $\nu(x)$, since for $x=0$ we have
\begin{align*}
\nu(0) & = \int_0^\infty \frac{\omega(u)}{u^{3/2}}\,du 
= \xi(-1) = \xi(2) = \frac12 \cdot 2\cdot 1\cdot \pi^{-1} \Gamma(1) \zeta(2) = \frac{\pi}{6}.
\end{align*}

To prove (iii), first, note that \eqref{eq:nutilde-asym1} follows immediately from \eqref{eq:nutilde-asym2} using the balancedness property of $\nu(x)$. Now, to prove \eqref{eq:nutilde-asym2} we use the technique of Mellin transform asymptotics, described, e.g., in \cite[Appendix~B.7]{flajolet-sedgewick}.
Let $\varphi(s) = \frac{2}{\sqrt{\pi}} \Gamma(s) \Gamma(\frac32-s)\xi(2s-1)$ denote the Mellin transform of $\nu(x)$. From the Mellin inversion formula we have
\begin{equation*}
\nu(x) = \frac{1}{2\pi i} \int_{c-i\infty}^{c+i\infty} \varphi(s) x^{-s}\,ds,
\end{equation*}
where $c$ is an arbitrary number in $(0,3/2)$. We now shift the integration contour to the left to the line $\im(s)=-n-3/2$ and apply the residue theorem to calculate the change in the value of the integral resulting from this contour shift. Note that the contour is unbounded, so one has to justify this use of the residue theorem using a limiting argument involving the rate of decay of the integrand along vertical lines; this is not hard to do using the standard facts that the xi and gamma functions both decay exponentially fast along vertical lines
(see \cite[p.~21]{andrews-askey-roy}, \cite[p.~121]{paris-kaminski}).

With this technicality out of the way, we see that the result of the contour shift is that
for each pole of the integrand $\varphi(s)x^{-s}$ that this contour shift skips over, the integral (including the factor of $\frac{1}{2\pi i}$ in front) changes by an amount equal to its residue. The poles being skipped over in this case are at $s=0, -1,\ldots, -n-1$, and the residue for the pole at $s=-k$ is equal to $x^k$ multiplied by
\begin{align*}
\frac{2}{\sqrt{\pi}} &\frac{(-1)^k}{k!} \Gamma\left(k+\frac32 \right) \xi(-2k-1)
=
\frac{2(-1)^k}{k!}\cdot
\frac{(2k+2)!}{2^{2k+2}(k+1)!}\cdot
\xi(2k+2)
\\ &=
\frac{(-1)^k (2k+2)!}{2^{2k+1} k! (k+1)!} \cdot
\frac12 (2k+2)(2k+1) \pi^{-(k+1)}k! \zeta(2k+2)
\\ &=
\frac{(-1)^k (2k+2)!}{2^{2k+1} k! (k+1)!} \cdot
\frac12 (2k+2)(2k+1) \pi^{-(k+1)}k! \cdot
\frac{(-1)^{k}(2\pi)^{2k+2}}{2(2k+2)!}B_{2k+2}
\\ &=
\frac{(2k+1)\pi^{k+1}}{k!}
B_{2k+2}.
\end{align*}
Thus, following the contour shift we get the alternative expression
\begin{align*}
\nu(x) &= \frac{1}{2\pi i} \int_{-(n+3/2)-i\infty}^{-(n+3/2)+i\infty} \varphi(s) x^{-s}\,ds
+
\sum_{k=0}^{n+1} \frac{(2k+1)\pi^{k+1}}{k!} B_{2k+2} x^k
\\ & =
\sum_{k=0}^{n+1} \frac{(2k+1)\pi^{k+1}}{k!} B_{2k+2} x^k
+ x^{n+3/2} O\left(
\int_{-(n+3/2)-i\infty}^{-(n+3/2)+i\infty} |\varphi(s)|\,ds
\right)
\\ & =
\sum_{k=0}^{n} \frac{(2k+1)\pi^{k+1}}{k!} B_{2k+2} x^k
+ O(x^{n+1}),
\end{align*}
as claimed.
\end{proof}

The next result shows that $\tilde{\nu}$ can be thought of as a generating function for the coefficient sequence $(c_n)_{n=0}^\infty$, providing another illustration of why $\nu(x)$ and $\tilde{\nu}(t)$ are interesting functions to study.

\begin{thm}[The function $\tilde{\nu}(t)$ as a generating function for the sequence $(c_n)_{n=0}^\infty$]
The centered function $\tilde{\nu}(t)$ has the power series expansion
\begin{equation}
\label{eq:nutilde-taylor}
\tilde{\nu}(t) =
\frac{1}{2\sqrt{2}} \sum_{n=0}^\infty 
\frac{(-1)^n (3/2)_n}{n!} c_n t^n
\qquad (|t|<1).
\end{equation}
\end{thm}

\begin{proof}
Noting the Taylor expansion
\begin{equation*}
\sum_{n=0}^\infty \frac{(3/2)_n}{n!} z^n = (1-z)^{-3/2},
\end{equation*}
we write
\begin{align*}
\tilde{\nu}(t) & =
\frac{1}{(1+t)^{3/2}} \nu\left(\frac{1-t}{1+t}\right)
=
\frac{1}{(1+t)^{3/2}} \int_0^\infty 
\frac{\omega(u)}{\left(u+\frac{1-t}{1+t}\right)^{3/2}}\,du
\\ & =
\int_0^\infty \omega(u) \left( 
1-t+u+t u
\right)^{-3/2}\,du
=
\int_0^\infty \frac{\omega(u)}{(u+1)^{3/2}}\left( 
1+t\frac{u-1}{u+1}
\right)^{-3/2}\,du
\\ & =
\int_0^\infty \frac{\omega(u)}{(u+1)^{3/2}}\left( 
\sum_{n=0}^\infty \frac{(-1)^n (3/2)_n}{n!} 
\left( \frac{u-1}{u+1} \right)^n t^n
\right)
\,du
\\ & =
\sum_{n=0}^\infty 
\frac{(-1)^n (3/2)_n}{n!} \left(
\int_0^\infty \frac{\omega(u)}{(u+1)^{3/2}}
\left(\frac{u-1}{u+1}\right)^n\,du
\right) t^n
=
\frac{1}{2\sqrt{2}} \sum_{n=0}^\infty 
\frac{(-1)^n (3/2)_n}{n!} c_n t^n,
\end{align*}
as claimed.
\end{proof}

Note that the power series \eqref{eq:nutilde-taylor} also converges for $t=1$ (as is immediately apparent from the asymptotic rate of decay of the coefficients $c_{2n}$), and, by Proposition~\ref{eq:nuofx-properties}(ii), its value there is given explicitly by
\begin{align*}
\tilde{\nu}(1) = \frac{1}{(1+1)^{3/2}} \nu\left(\frac{1-1}{1+1}\right) = \frac{1}{2\sqrt{2}} \nu(0) = \frac{\pi}{12\sqrt{2}}.
\end{align*}
That is, we have the summation identity
\begin{equation*}
\sum_{n=0}^\infty \frac{(3/2)_{2n}}{(2n)!} c_{2n}
= \frac{\pi}{6}.
\end{equation*}

Next, we mention another curious integral representation expressing $\nu(x)$ in terms of $A(r)$.

\begin{prop}
\label{prop:stieltjes-from-laplace}
The function $\nu(x)$ can be expressed in terms of $A(r)$ as
\begin{equation*}
\nu(x) = 4\pi \int_0^\infty A(r)r^2
e^{-\pi x r^2}
\,dr.
\end{equation*}
\end{prop}

\begin{proof}
As pointed out to us by Jim Pitman, this is a special case of a standard relation expressing the generalized Stieltjes transform of a function in terms of its Laplace transform. In our notation, we have that
\begin{align*}
\int_0^\infty A(r) r^2 e^{-\pi x r^2}\,dr &=
\int_0^\infty \left(\int_0^\infty \omega(t) e^{-\pi t r^2} \,dt
\right) r^2  e^{-\pi xr^2}\,dr
\\ & =
\int_0^\infty \omega(t) \left(
\int_0^\infty r^2 e^{-\pi(x+t)r^2}\,dr
\right) \,dt
\\ & = \int_0^\infty \omega(t) \left(\frac{1}{2} \Gamma\left(\frac32\right) 
\frac{1}{\pi^{3/2}(x+t)^{3/2}}\right) \,dt = \frac{1}{4\pi} \nu(x).
\end{align*}
\end{proof}

Next, we turn to examining the function $B(r)$. As it turns out, it can be evaluated explicitly in terms of the digamma function.

\begin{prop}
\label{prop:bofr-explicit}
The function $B(r)$ is given explicitly by
\begin{equation}
\label{eq:bofr-explicit}
B(r) = 1-2r \psi'(r) - r^2 \psi''(r),\end{equation}
where $\psi(x) = \frac{\Gamma'(x)}{\Gamma(x)}$ is the digamma function.
\end{prop}

Combining this with Lemma~\ref{lem:recbal-selftrans}, we obtain the following result, which seems to be new (compare with \cite[Sec.~5.3]{berndt}, \cite[Sec.~9.12]{titschmarsh-fourier}, \cite{mathoverflow-fourier} where some results with a similar flavor are discussed).

\begin{corollary}
The function $1-2r \psi'(r)-r^2 \psi''(r)$ is a Fourier self-transform, considered as a radial function on $\R$. That is, it satisfies the integral equation
\begin{equation*}
F(\rho)  =
2 \int_0^\infty F(r) \cos(2\pi r\rho)\,dr \qquad (\rho\ge0).
\end{equation*}

\end{corollary}

\begin{proof}[First proof of Proposition~\ref{prop:bofr-explicit}]
Denote $\beta(r) = 1-2r \psi'(r) - r^2 \psi''(r)$. The proof that $B(r)=\beta(r)$ is based on computing the Mellin transform of $\beta(r)$. If we succeed in showing that this Mellin transform is defined for $0<\re(s)<2$ and is equal to the function given in \eqref{eq:bofr-mellin}, the equality $B(r)=\beta(r)$ will follow from the standard uniqueness theorem for the Mellin transform.

We recall some useful facts about the digamma function and its derivatives \cite[p.~260]{abramowitz-stegun}. Start with the well-known partial fraction decomposition
\begin{equation*}
\psi(r+1) = -\gamma + \sum_{n=1}^\infty \left( \frac1n - \frac{1}{r+n}\right),
\end{equation*}
where $\gamma$ denotes the Euler-Mascheroni constant.
Repeated differentiation gives the also-standard expansion
\begin{equation}
\label{eq:polygamma-partial-fractions}
\psi^{(m)}(r+1) =  (-1)^{m+1}\sum_{n=1}^\infty \frac{m!}{(r+n)^{m+1}} \qquad (m\ge 1).
\end{equation}
The Taylor expansion of $\psi^{(m)}(r+1)$ around $r=0$ is given by
\begin{equation}
\label{eq:polygamma-taylor}
\psi^{(m)}(r+1) = \sum_{k=0}^\infty (-1)^{m+k+1}  (k+1)(k+2)\cdots (k+m) \zeta(k+m+1) r^k
\quad\ \  (m\ge1),
\end{equation}
and its asymptotic expansion as $r\to\infty$ (with $m\ge 1$, $N\ge0$ fixed) is
\begin{equation}
\label{eq:polygamma-asym}
\psi^{(m)}(r+1) =
(-1)^{m+1} \sum_{k=0}^N \frac{(k+m-1)!}{k!} \frac{B_k}{(r+1)^{k+m}} 
+ O\left(\frac{1}{r^{N+m+1}}\right)
\qquad \textrm{($r\to\infty$).}
\end{equation}

With this preparation, the Mellin transform of derivatives of $\psi(r+1)$ can be evaluated through termwise integration of the terms of \eqref{eq:polygamma-partial-fractions}.
After a short computation using~\eqref{eq:mellin-textbook-ex} and the reflection formula 
$\Gamma(z)\Gamma(1-z) =\frac{\pi}{\sin(\pi z)}$, 
we get that
\begin{align}
\label{eq:polygamma-mellin}
\int_0^\infty \psi^{(m)}&(r+1) r^{s-1}\,dr 
=
-\frac{\pi (s-1)(s-2)\cdots (s-m)}{\sin(\pi s)} \zeta(m+1-s) 
\qquad 
(0 < \re(s) < m)
\end{align}
for $m\ge 1$
(a variant of this formula is mentioned in \cite[p.~659, eq.~6.473]{gradshteyn-ryzhik}; see also \cite[p.~325, eq.~(7)]{erdelyi}). Note that the strip of convergence for each of the individual Mellin transforms being summed is $0 < \re(s) < m+1$, and the requirement of absolute summability imposes the further restriction $\re(s) < m$ (apparent in the zeta-term $\zeta(m+1-s)$, which blows up as $\re(s)$ approaches $m$ from the left).

Building on the facts discussed above, we can approach the computation of the Mellin transform of $\beta(r)$. First, note that, because of the standard identity $\psi(x+1) = \psi(x) + \frac{1}{x}$, $\beta(r)$ can be rewritten as
\begin{equation*}
\beta(r) = 1-2r \psi'(r+1) - r^2 \psi''(r+1).
\end{equation*}
By \eqref{eq:polygamma-taylor}--\eqref{eq:polygamma-asym}, the asymptotic behavior of $\beta(r)$ for $r$ near $0$ and $\infty$ is given by
\begin{align}
\label{eq:betaofr-asym1}
\beta(r) & = 1 - \frac{\pi^2}{3} r + O(r^2)
\qquad \textrm{as }r\to 0,
\\
\label{eq:betaofr-asym2}
\beta(r) & = \frac{1}{6r^2} + O\left(r^{-4}\right)
\ \ \qquad \textrm{as }r\to \infty.
\end{align}
This implies that the Mellin transform of $\beta(r)$ is defined for $0 < \re(s) < 2$.
Also of some interest is the function $\beta(r)-1$ and its own Mellin transform. Again by \eqref{eq:betaofr-asym1}--\eqref{eq:betaofr-asym2},
it follows that the Mellin transform of $\beta(r)-1$ is defined for $-1 < \re(s) < 0$. Using \eqref{eq:polygamma-mellin}, it is readily evaluated for such $s$ to be
\begin{align}
\label{eq:bofr-minusone-mellin}
\int_0^\infty (\beta(r)-1) r^{s-1} \, dr & = 
\int_0^\infty \left(-2 \psi'(r+1)r^s - \psi''(r+1)r^{s+1}\right) \, dr 
\\ \nonumber & 
= 
\frac{\pi s(s-1)}{\sin(\pi s)} \zeta(1-s) 
\qquad\qquad \qquad
(-1 < \re(s) < 0).
\end{align}
Now, this isn't exactly what we want, both because of the irksome ``$-1$'' term in the integrand and because the range of validity of the formula is disjoint from the range $0<\re(s)<2$ we are interested in. We can nonetheless exploit this identity to get what we need through a trick involving analytic continuation. Namely, for $s$ satisfying $0<\re(s)<2$, we can try to evaluate the Mellin transform of $\beta(r)$ by performing an integration by parts, which gives that (under the stated assumptions on $s$)
\begin{align*}
\int_0^\infty \beta(r) r^{s-1}\,dr
& =
\frac{1}{s} \beta(r) r^s\Big|_{r=0}^{r=\infty}
- \frac{1}{s} \int_0^\infty \beta'(r)r^{s+1}\,dr 
= 
- \frac{1}{s} \int_0^\infty \beta'(r)r^{s+1}\,dr 
=: K(s).
\end{align*}
Moreover, examining the asymptotic behavior of $\beta'(r)$ near $r=0$ and $r=\infty$ (which, the reader can confirm using \eqref{eq:polygamma-taylor}--\eqref{eq:polygamma-asym}, is simply that obtained by differentiating the terms in \eqref{eq:betaofr-asym1}--\eqref{eq:betaofr-asym2} termwise, including the big-$O$ term),
we see that the Mellin transform of $\beta'(r)$ converges (and is an analytic function) in the strip $0 < \re(s) < 3$, which implies that
$K(s)$ is an analytic function in the strip $-1 < \re(s) < 2$.
But for $s$ satisfying $-1<\re(s)<0$, performing a similar integration by parts as the one above starting with the integral in \eqref{eq:bofr-minusone-mellin} gives
\begin{align*}
\int_0^\infty (\beta(r)-1) r^{s-1}\,dr
& =
\frac{1}{s} (\beta(r)-1) r^s\Big|_{r=0}^{r=\infty}
- \frac{1}{s} \int_0^\infty \frac{d}{dr}(\beta(r)-1) r^{s+1}\,dr 
\\ & =
- \frac{1}{s} \int_0^\infty \frac{d}{dr}(\beta(r)-1) r^{s+1}\,dr 
= 
- \frac{1}{s} \int_0^\infty \beta'(r)r^{s+1}\,dr = K(s),
\end{align*}
the \emph{same analytic function} evaluated on a different part of its domain of definition. Since $K(s)$ is analytic and given by the formula found in \eqref{eq:bofr-minusone-mellin} for $-1<\re(s)<0$, by the principle of analytic continuation it is also equal to the same expression for $0<\re(s)<2$. That is, we have shown that
\begin{equation}
\label{eq:betaofr-mellin-proof1}
\int_0^\infty \beta(r) x^{s-1} \, dx = 
\frac{\pi s(s-1)}{\sin(\pi s)} \zeta(1-s) 
\qquad
(0 < \re(s) < 2).
\end{equation}
Finally, using the functional equation of the Riemann zeta function it is easy to check that the function on the right-hand side of \eqref{eq:betaofr-mellin-proof1} is equal to the one on the right-hand side of \eqref{eq:bofr-mellin}. This shows that the Mellin transforms of $\beta(r)$ and $B(r)$ coincide (in their ``natural'' region of definition where the Mellin transform of both converges), and finishes the proof.
\end{proof}

\begin{proof}[Second proof of Proposition~\ref{prop:bofr-explicit}]
A second method is to use Mellin inversion to represent $B(r)$ in terms of its Mellin transform found in \eqref{eq:bofr-mellin}, namely as
\begin{equation*}
B(r) = \frac{1}{2\pi i} \int_{c-i\infty}^{c+i\infty} \frac{(2\pi)^{1-s}}{2\sin(\pi s/2)}
 s(s-1) \Gamma(s) \zeta(s)
 r^{-s}\,ds,
\end{equation*}
where $c$ is an arbitrary number in $(0,2)$. Fix an integer $n\ge1$, and shift the integration contour to the left to the line $\im(s)=-n-1/2$. As in the proof of Proposition~\ref{eq:nuofx-properties}, we can use the residue theorem (with the same arguments to justify its application in this setting involving unbounded contours) to calculate the change in the value of the integral by looking at the poles of the integrand being skipped over and their residues. The relevant poles in this case are at $s=0, -1,\ldots, -n$. We leave to the reader to verify that the residue at $s=0$ is equal to $1$, and that for any $k\ge 1$, the pole at $-k$ has residue $(-1)^k k (k+1) \zeta(k+1) r^k$. We therefore get that
\begin{align*}
B(r) = 
1 + \sum_{k=1}^n &(-1)^k k (k+1) \zeta(k+1) r^k
+\frac{1}{2\pi i} \int_{-(n+1/2)-i\infty}^{(n+1/2)+i\infty} \frac{(2\pi)^{1-s}}{2\sin(\pi s/2)}
 s(s-1) \Gamma(s) \zeta(s)
 r^{-s}\,ds.
\end{align*}
Assume that $0<r<1$. In this case it can be shown without much difficulty that the integral converges to $0$ as $n\to\infty$ (this becomes easier to do if one first replaces the Mellin transform of $B(r)$ in the integrand by the simpler expression on the right-hand side of \eqref{eq:betaofr-mellin-proof1}, which as we commented above in fact represents the same function). The conclusion is that $B(r)$ is represented for $0<r<1$ by a convergent Taylor series
\begin{equation*}
B(r) = 
1 + \sum_{k=1}^\infty (-1)^k k (k+1) \zeta(k+1) r^k.
\end{equation*}
But this is consistent with (and implies) \eqref{eq:bofr-explicit}: using \eqref{eq:polygamma-taylor}, one can check easily that the function $\beta(r) = 1-2r \psi'(r+1)-r^2 \psi''(r+1)$ has the same Taylor expansion.
This proves \eqref{eq:bofr-explicit} for $0<r<1$, and the claim follows for general $r$ by analytic continuation.
\end{proof}

\chapter{Expansion of $\Xi(t)$ in the polynomials $g_n$}

\label{ch:gn-expansion}

In this chapter we continue to build on the tools developed in Chapters~\ref{ch:fn-expansion} and~\ref{ch:radial}, in order to derive an infinite series expansion for the Riemann xi function in yet another family of orthogonal polynomials, the family $(g_n(x))_{n=0}^\infty$, and study its properties. As we discussed briefly in the Introduction, the polynomials $g_n$ are defined by
\begin{equation*}
g_n(x) = p_n\left(x; \frac34,\frac34,\frac34,\frac34\right)
=
i^n (n+1) \ {}_3F_2\left(-n,n+2,\frac34+ix;\frac32,\frac32; 1\right),
\end{equation*}
where $p_n(x;a,b,c,d)$ denotes the continuous Hahn polynomial with parameters $a,b,c,d$.
They form a family of orthogonal polynomials with respect to the weight function $\left|\Gamma\left(\frac34+ix\right)\right|^4$ on $\R$. Their main properties are summarized in \secref{sec:orth-gn}.

\section{Main results}

\label{sec:gnexp-mainresults}

As in Chapters~\ref{ch:hermite} and~\ref{ch:fn-expansion}, we start by defining a sequence of numbers $(d_n)_{n=0}^\infty$ that will play the role of the coefficients associated with the new expansion. Define
\begin{align}
\label{eq:def-gn-coeffs}
d_n & = \frac{(3/2)_n}{2^{n-3/2} n!}\int_0^\infty \frac{\omega(x)}{(x+1)^{3/2}} \left(\frac{x-1}{x+1}\right)^n 
\gausshyper\left(\frac{n}{2}+\frac34,
\frac{n}{2}+\frac54; n+2; \left(\frac{x-1}{x+1}\right)^2 \right)
 \,dx.
\end{align}

As a first step towards demistifying this somewhat obscure definition, we expand the $\gausshyper$ term in an infinite series. Momentarily ignoring issues of convergence, we have that
\begin{align}
\label{eq:dn-gausshyper-expansion}
d_n & = \frac{(3/2)_n}{2^{n-3/2} n!}\int_0^\infty \frac{\omega(x)}{(x+1)^{3/2}} \left(\frac{x-1}{x+1}\right)^n 
\left[\sum_{m=0}^\infty 
\frac{\left(\frac{n}{2}+\frac34\right)_m
\left(\frac{n}{2}+\frac54\right)_m 
}{m! (n+2)_m} 
\left( \frac{x-1}{x+1}\right)^{2m} \right]
 \,dx
\\ \nonumber &
= 
\frac{(3/2)_n}{2^{n-3/2} n!}
\sum_{m=0}^\infty 
\frac{\left(\frac{n}{2}+\frac34\right)_m \left(\frac{n}{2}+\frac54\right)_m }{m! (n+2)_m}
\int_0^\infty 
\frac{\omega(x)}{(x+1)^{3/2}} \left(\frac{x-1}{x+1}\right)^{n+2m} 
\,dx
\\ \nonumber &
= \frac{(3/2)_n}{2^n n!}
\sum_{m=0}^\infty 
\frac{\left(\frac{n}{2}+\frac34\right)_m \left(\frac{n}{2}+\frac54\right)_m }{m! (n+2)_m}
c_{n+2m}
=
\frac{n+1}{2^n} \sum_{m=0}^\infty \frac{(3/2)_{n+2m}}{4^m m! (n+m+1)!} c_{n+2m},
\end{align}
where in the last step we use the relation
$\left(\frac{n}{2}+\frac34\right)_m \left(\frac{n}{2}+\frac54\right)_m
=
\frac{(3/2)_{n+2m}}{2^{2m} (3/2)_n}
$. In addition to being an interesting way to express $d_n$ in terms of the coefficients $c_k$, this suggests a relatively simple way to see that the integral \eqref{eq:def-gn-coeffs} converges absolutely (which would also justify the above formal computation); namely, letting
\begin{equation*}
c_n' = \int_0^\infty  \frac{\omega(x)}{(x+1)^{3/2}} \left|\frac{x-1}{x+1}\right|^n \,dx,
\end{equation*}
we have the simple relations $c_r' \ge c_{r+1}'$, $c_{2r}'=c_{2r}$, and therefore (using \eqref{eq:fn-coeff-asym}, or some easy corollary of Lemma~\ref{lem:fn-easybound2}) get that $c_r' \le K e^{-M \sqrt{r}}$ for all $r\ge0$ and some constants $K,M>0$. This then gives that
\begin{align}
\int_0^\infty
&
\left|
\frac{\omega(x)}{(x+1)^{3/2}} \left(\frac{x-1}{x+1}\right)^n 
\gausshyper\left(\frac{n}{2}+\frac34,
\frac{n}{2}+\frac54; n+2; \left(\frac{x-1}{x+1}\right)^2 \right)
\right| \,dx 
\label{eq:finiteness-bound}
\\ & \leq
\sum_{m=0}^\infty 
\frac{\left(\frac{n}{2}+\frac34\right)_m \left(\frac{n}{2}+\frac54\right)_m }{m! (n+2)_m}
\int_0^\infty 
\left|
\frac{\omega(x)}{(x+1)^{3/2}} \left(\frac{x-1}{x+1}\right)^{n+2m} 
\right|
\,dx
\nonumber \\ &=
\sum_{m=0}^\infty 
\frac{\left(\frac{n}{2}+\frac34\right)_m \left(\frac{n}{2}+\frac54\right)_m }{m! (n+2)_m} c_{n+2m}'
=
\frac{(n+1)!}{(3/2)_n} \sum_{m=0}^\infty 
\frac{(3/2)_{n+2m}}{4^m m!(n+m+1)!}
 c_{n+2m}'
\nonumber \\ &\leq
K \frac{(n+1)!}{(3/2)_n} \sum_{m=0}^\infty 
\frac{(3/2)_{n+2m}}{4^m m!(n+m+1)!}
e^{-M \sqrt{n+2m}}
\nonumber \\ &\leq
2K \frac{(n+1)!}{(3/2)_n} \sum_{m=0}^\infty 
\frac{(2)_{n+2m}}{2^{2m+1} m!(n+m+1)!}
e^{-M \sqrt{n+2m}}
\nonumber \\ &=
2K \frac{(n+1)!}{(3/2)_n} \sum_{m=0}^\infty 
\frac{1}{2^{2m+1}} \binom{n+2m+1}{m}
e^{-M \sqrt{n+2m}}
\nonumber \\ &\leq
2^{n+1} K \frac{(n+1)!}{(3/2)_n} \sum_{m=0}^\infty 
e^{-M \sqrt{n+2m}} < \infty,
\nonumber
\end{align}
establishing the absolute convergence.

We summarize the above observations as a proposition.

\begin{prop}
(i) The integral defining $d_n$ converges absolutely for all \mbox{$n\ge 0$.}

(ii) We have $d_{2n+1} = 0$ for all $n\ge 0$.

(iii) We have $d_{2n} > 0$ for all $n\ge 0$.

(iv) $d_n$ can be expressed alternatively in terms of the coefficients $c_k$ as
\begin{equation}
\label{eq:dn-cn-expansion}
d_n = \frac{n+1}{2^n} \sum_{m=0}^\infty \frac{(3/2)_{n+2m}}{4^m m!(n+m+1)!} c_{n+2m}.
\end{equation}
\end{prop}

We are ready to formulate the main results concerning the expansion of $\Xi(t)$ in the polynomials $g_n$, which are precise analogues of 
Theorem~\ref{THM:HERMITE-EXPANSION} and~\ref{thm:hermite-coeff-asym} in Chapter~\ref{ch:hermite} and 
Theorem~\ref{THM:FN-EXPANSION}~and~\ref{THM:FN-COEFF-ASYM} in Chapter~\ref{ch:fn-expansion}.

\begin{thm}[Infinite series expansion for $\Xi(t)$ in the polynomials $g_n$]
\label{THM:GN-EXPANSION}
The Riemann xi function has the infinite series representation
\begin{equation}
\label{eq:gn-expansion}
\Xi(t) = \sum_{n=0}^\infty (-1)^n d_{2n} g_{2n}\left(\frac{t}{2}\right)
\end{equation}
which converges uniformly on compacts for all $t\in\C$. More precisely, for any compact set $K\subset \C$ there exist constants $C_1, C_2>0$ depending on $K$ such that
\begin{equation}
\label{eq:gn-expansion-errorbound}
\left|\Xi(t) - \sum_{n=0}^N (-1)^n d_{2n} g_{2n}\left(\frac{t}{2}\right) \right| \leq C_1 e^{-C_2 N^{2/3}}
\end{equation}
holds for all $N\ge0$ and $t\in K$.
\end{thm}

\begin{thm}[Asymptotic formula for the coefficients $d_{2n}$]
\label{thm:gn-coeff-asym}
The asymptotic behavior of $d_{2n}$ for large $n$ is given by
\begin{equation}
\label{eq:gn-coeff-asym}
d_{2n} = 
\left(1+O\left(n^{-1/10}\right)\right) D n^{4/3} \exp\left(-E n^{2/3}\right)
\end{equation}
as $n\to\infty$, where $D,E$ are constants given by
\begin{equation}
\label{eq:gncoeffs-asym-consts}
D = \frac{128\times 2^{1/3} \pi^{2/3}e^{-2\pi/3}}{\sqrt{3}}, \qquad
E = 3(4\pi)^{1/3}.
\end{equation}

\end{thm}

As in Chapters~\ref{ch:hermite} and~\ref{ch:fn-expansion}, we note the fact that the coefficients $d_n$ can be computed as inner products in the $L^2$-space $L^2(\R,\left|\Gamma\left(\frac34+\frac{it}{2}\right)\right|^4)$.

\begin{corollary}
The coefficients $d_n$ can be alternatively expressed as
\begin{equation}
\label{eq:gn-coeff-innerproduct}
d_n = \frac{8}{\pi^3}(-i)^n \int_{-\infty}^\infty \Xi(t) g_n\left(\frac{t}{2}\right) \left|\Gamma\left(\frac34+\frac{it}{2}\right)\right|^4\,dt.
\end{equation}
\end{corollary}

\begin{proof}
Repeat the arguments in the proofs of 
Corollaries~\ref{cor:hermite-coeff-innerproduct}
and~\ref{cor:fn-coeff-innerproduct}.
\end{proof}

\section{Proof of Theorem~\ref{THM:GN-EXPANSION}}

\label{sec:gnexp-proofexpansion}

The next two lemmas are analogues of Lemmas~\ref{lem:hermite-easybound} and~\ref{lem:hermite-easybound2} in Chapter~\ref{ch:hermite} and Lemmas~\ref{lem:fn-easybound} and~\ref{lem:fn-easybound2} in Chapter~\ref{ch:fn-expansion}.

\begin{lem} 
\label{lem:gn-easybound}
The polynomials $g_n(x)$ satisfy the bound
\begin{equation}
\label{eq:gn-easybound}
|g_n(x)|\leq C_1 e^{C_2 n^{1/3}}
\end{equation}
for all $n\ge 0$, uniformly as $x$ ranges over any compact set $K\subset \C$, with $C_1,C_2>0$ being constants that depend on $K$ but not on $n$.
\end{lem}

\begin{proof}
This is identical to the proof of Lemma~\ref{lem:fn-easybound}, except that the use of the recurrence relation \eqref{eq:fn-recurrence} is replaced by the analogous relation \eqref{eq:gn-recurrence} for the sequence $g_n(x)$, with the result that some small modifications need to be made to the constants in the proof. We leave the details as an exercise.
\end{proof}

\begin{lem}
There exist constants $J_1, J_2>0$ such that for all $n\ge0$, the bound
\begin{align}
\label{eq:gn-easybound2}
\frac{1}{2^n}\int_0^\infty \frac{\omega(x)}{(x+1)^{3/2}}
& \left|
\left( \frac{x-1}{x+1}\right)^n
\gausshyper\left(\frac{n}{2}+\frac34,
\frac{n}{2}+\frac54; n+2; \left(\frac{x-1}{x+1}\right)^2 \right)
\right|
\,dx 
\leq J_1 e^{-J_2 n^{2/3}}
\end{align}
holds.
\end{lem}

\begin{proof}
Note that this is a stronger version of the finiteness bound \eqref{eq:finiteness-bound} that makes explicit the dependence of the bound on $n$. To prove it, we refer back to the penultimate line of \eqref{eq:finiteness-bound} and proceed from there a bit more economically than before. Multiplying by $\frac{1}{2^n}$ and using the trivial fact that $\frac{(n+1)!}{(3/2)_n}\leq 2n$, we get that the integral in \eqref{eq:gn-easybound2} (together with the leading factor of $\frac{1}{2^n}$) is bounded from above by
\begin{align*}
& 4 K n\sum_{m=0}^\infty 
\frac{1}{2^{n+2m+1}} \binom{n+2m+1}{m}
e^{-M \sqrt{n+2m}}
\\ &=
4 K  n \Bigg[\sum_{m\le \frac{1}{8M^{2/3}} n^{4/3}}
\frac{1}{2^{n+2m+1}} \binom{n+2m+1}{m}
e^{-M \sqrt{n+2m}}
+ \sum_{m> \frac{1}{8M^{2/3}} n^{4/3}}
e^{-M \sqrt{n+2m}} \Bigg]
\end{align*}
(with the same constants $K,M$ appearing in \eqref{eq:finiteness-bound}).
We will show that each of the two sums in this last expression satisfies a bound of the sort we need.
For the second sum, observe that it is bounded by the integral
\begin{equation*}
\int_{\frac{1}{8M^{2/3}} n^{4/3}-1}^\infty
e^{-M\sqrt{n+2x}}\,dx,
\end{equation*}
and this integral is $O\left(\exp\left(-\frac{M^{2/3}}{2} n^{2/3}\right)\right)$, by the relation
\begin{equation*}
\int_{A}^\infty e^{-M \sqrt{n+2x}}
= \frac{1}{M^2} (M\sqrt{n+2A}+1) e^{-M\sqrt{n+2A}}.
\end{equation*}
To estimate the first sum, we claim that the terms in that sum are increasing as a function of $m$ for all large enough (but fixed) $n$; this would imply that the sum is bounded for large $n$ by $\frac{1}{8M^{2/3}} n^{4/3}$ times the last term, which in turn is at most $O\left(\exp\left(-\frac{M^{2/3}}{2} n^{2/3}\right)\right)$, and hence, when combined with the estimates above, would imply the claim of the lemma.

To prove the claim, observe that the ratio of successive terms in the sum is
\begin{align}
\label{eq:ratio-successive}
\frac{
\frac{1}{2^{n+2m+3}} \binom{n+2m+3}{m+1}
e^{-M\sqrt{n+2m+2}}
}{
\frac{1}{2^{n+2m+1}} \binom{n+2m+1}{m}
e^{-M\sqrt{n+2m}}
}
& =
\frac{(n+2m+2)(n+2m+3)}{4(m+1)(n+m+2)}
e^{M(\sqrt{n+2m}-\sqrt{n+2m+2})}
\\ \nonumber & \geq
\frac{\left(m+\frac{n}{2}\right)^2}{(m+1)(m+n+2)}
\left(1-\frac{M}{\sqrt{m+\frac{n}{2}}}\right)
\\ \nonumber & =
\frac{\left(m+\frac{n}{2}\right)^2}{\left(m+\frac{n}{2}\right)^2 -\left(\frac{n^2}{4}-3m-n-2\right)}
\left(1-\frac{M}{\sqrt{m+\frac{n}{2}}}\right).
\end{align}
Our claim is equivalent to the statement that, under the assumption $m\le \frac{1}{8M^{2/3}} n^{4/3}$, the last expression in \eqref{eq:ratio-successive} is $\ge 1$. Equivalently, we need to show that the inequality
\begin{equation}
\label{eq:ineq2verify1}
1- \frac{\frac{n^2}{4}-3m-n-2}{\left(m+\frac{n}{2}\right)^2} \leq 1-\frac{M}{\sqrt{m+\frac{n}{2}}}
\end{equation}
holds for those values of $m$. This reduces after some further simple algebra to verifying the inequality
\begin{equation}
\label{eq:ineq2verify2}
m+\frac{n}{2} \leq \frac{1}{M^{2/3}}\left(\frac{n^2}{4}-3m-n-2\right)^{2/3}.
\end{equation}
To check this, assume that $n$ is large enough so that the inequalities
\begin{equation}
\label{eq:largen-assumption}
3\left(\frac{1}{8M^{2/3}}\right) n^{4/3}+n + 2 \le \frac{n^2}{8},
\qquad
\frac{n}{2} \le \frac{1}{8 M^{2/3}}n^{4/3}
\end{equation}
are satisfied. Then, together with our assumption on $m$, that also implies that
\begin{equation*}
\frac{n^2}{8} \leq \frac{n^2}{4}-3m-n-2,
\end{equation*}
and therefore also that
\begin{align*}
m + \frac{n}{2} 
&\leq
\frac{1}{8M^{2/3}} n^{4/3} + \frac{1}{8M^{2/3}} n^{4/3}
=
\frac{1}{4M^{2/3}} n^{4/3}
=
\frac{1}{M^{2/3}} \left( \frac{n^2}{8} \right)^{2/3}
\leq
\frac{1}{M^{2/3}} \left( \frac{n^2}{4}-3m-n-2 \right)^{2/3}.
\end{align*}
This verifies~\eqref{eq:ineq2verify2}, hence also \eqref{eq:ineq2verify1}, for all $n$ satisfying \eqref{eq:largen-assumption} (which clearly includes all values of $n$ larger than some fixed $N_0$), and therefore finishes the proof of the claim and also of the lemma.
\end{proof}

We need one final bit of preparation before proving Theorem~\ref{THM:GN-EXPANSION}. Recall that in the proof of Theorem~\ref{THM:FN-EXPANSION}, a key idea was the observation that the integration kernel $x^{s/2-1}$ can be related to the generating function of the polynomials $f_n(t/2)$. The next lemma shows a way of representing the same generating function as an infinite series involving the polynomials $g_n(t/2)$.

\begin{lem}
\label{lem:fn-genfun-gnexpansion}
For $w\in \C$ and $|z|<1$, we have the identity
\begin{equation*}
\sum_{n=0}^\infty f_n(w) z^n 
= \sum_{n=0}^\infty \frac{(3/2)_n}{2^n n!} g_n(w) 
\,\gausshyper\left(\frac{n}{2}+\frac34,\frac{n}{2}+\frac54;
n+2;-z^2\right) z^n .
\end{equation*}
\end{lem}

\begin{proof}
Using the relation \eqref{eq:fn-intermsof-gn} expressing the polynomial $f_n$ in terms of the $g_k$'s, we can write
\begin{equation}
\label{eq:fn-genfun-intermsof-gn}
\sum_{n=0}^\infty f_n(w) z^n 
=
\sum_{n=0}^\infty
\left( \frac{(3/2)_n}{2^n (n+1)!} 
\sum_{m=0}^{\lfloor n/2 \rfloor} (-1)^m (n-2m+1)
\binom{n+1}{m} g_{n-2m}(w)
\right) z^n.
\end{equation}
The claim will follow by suitably rearranging the terms in this double summation. First, let us check that this is permitted by showing that the sum is in fact absolutely convergent. Indeed, using Lemma~\ref{lem:gn-easybound} to bound $|g_{n-2m}(w)|$ (with $w$ being fixed and the resulting constants $C_1,C_2$ depending on $w$ but not on $n,m$), we see that
\begin{align*}
& \sum_{n=0}^\infty 
\sum_{m=0}^{\lfloor n/2 \rfloor}
 \frac{(3/2)_n}{2^n (n+1)!} 
\left| (-1)^m (n-2m+1)
\binom{n+1}{m} g_{n-2m}(w) \right|
\cdot |z|^n
\\& \ \leq
\sum_{n=0}^\infty
\frac{(3/2)_n}{2^n (n+1)!} (n+1) 2^{n+1} \times C_1 e^{C_2 n^{1/3}}
|z|^n
= 
2C_1 \sum_{n=0}^\infty (n+1) e^{C_2 n^{1/3}} |z|^n
< \infty.
\end{align*}
With absolute convergence established, we can rewrite the double sum in \eqref{eq:fn-genfun-intermsof-gn}, introducing a new summation index $k=n-2m$ in place of the index $n$, as
\begin{align*}
\sum_{k=0}^\infty \sum_{m=0}^\infty &
\frac{(3/2)_{k+2m}}{2^{k+2m}(k+2m+1)!} 
(-1)^m
(k+1)
\binom{k+2m+1}{m} g_k(w)
z^{k+2m}
\\ & \qquad=
\sum_{k=0}^\infty \frac{(3/2)_k}{2^k k!} g_k(w) z^k
\left(
\sum_{m=0}^\infty \frac{(-z^2)^m}{m!} 
\frac{(3/2)_{k+2m}}{(3/2)_k 2^{2m}} \frac{(k+1)!}{(k+m+1)!}
\right)
\\ & \qquad=
\sum_{k=0}^\infty \frac{(3/2)_k}{2^k k!} g_k(w) z^k
\left(
\sum_{m=0}^\infty 
\frac{\left(\frac{k}{2}+\frac34\right)_m
\left(\frac{k}{2}+\frac54\right)_m}{m!(k+2)_m}
\left(-z^2\right)^m
\right)
\\ & \qquad=
\sum_{k=0}^\infty \frac{(3/2)_k}{2^k k!} g_k(w) z^k
\gausshyper\left( \frac{k}{2}+\frac34, 
\frac{k}{2}+\frac54 ; k+2 ; -z^2 \right),
\end{align*}
as was the claim to prove.
\end{proof}

We are ready to prove \eqref{eq:gn-expansion-errorbound}. The calculation parallels that in the proofs of Theorems~\ref{THM:HERMITE-EXPANSION} and~\ref{THM:FN-EXPANSION}. Namely, start by estimating in a fairly simple-minded way that
\begin{align}
\label{eq:gnexp-ineq-chain}
\Bigg|\Xi(t) - & \sum_{n=0}^N (-1)^{2n}d_{2n} g_{2n}\left(\frac{t}{2}\right)\Bigg|
=
\left|\Xi(t) - \sum_{n=0}^{2N} i^n d_n g_n\left(\frac{t}{2}\right)\right|
\\ \nonumber &=
\Bigg|\int_0^\infty
\omega(x)
\Bigg(x^{-\frac34+\frac{it}{2}} - 
\sum_{n=0}^{2N} i^n
\frac{(3/2)_n}{2^{n-3/2} n!} \frac{1}{(x+1)^{3/2}} \left(\frac{x-1}{x+1}\right)^n 
\\ \nonumber & \qquad\qquad\qquad\qquad\qquad\qquad \times 
\gausshyper\left(\frac{n}{2}+\frac34,
\frac{n}{2}+\frac54; n+2; \left(\frac{x-1}{x+1}\right)^2 \right)
g_n\left(\frac{t}{2}\right)\Bigg)
\,dx\Bigg|
\\ \nonumber &\leq
\int_0^\infty
\omega(x)
\Bigg|
x^{-\frac34+\frac{it}{2}} - 
\sum_{n=0}^{2N} i^n
\frac{(3/2)_n}{2^{n-3/2} n!} \frac{1}{(x+1)^{3/2}} \left(\frac{x-1}{x+1}\right)^n 
\\ \nonumber & \qquad\qquad\qquad\qquad\qquad\qquad \times 
\gausshyper\left(\frac{n}{2}+\frac34,
\frac{n}{2}+\frac54; n+2; \left(\frac{x-1}{x+1}\right)^2 \right)
g_n\left(\frac{t}{2}\right)\Bigg|
\,dx.
\end{align}
By \eqref{eq:int-kernel-expansion} and Lemma~\ref{lem:fn-genfun-gnexpansion} (with $z=i(x-1)/(x+1)$), the kernel $x^{-\frac34+\frac{it}{2}}$ can be expanded as
\begin{align*}
x^{-\frac34+\frac{it}{2}}
& =
\sum_{n=0}^\infty i^n
\frac{(3/2)_n}{2^{n-3/2} n!} \frac{1}{(x+1)^{3/2}} \left(\frac{x-1}{x+1}\right)^n 
\gausshyper\left(\frac{n}{2}+\frac34,
\frac{n}{2}+\frac54; n+2; \left(\frac{x-1}{x+1}\right)^2 
\right) g_n\left(\frac{t}{2}\right). 
\end{align*}
Continuing the chain of inequalities \eqref{eq:gnexp-ineq-chain}, we therefore get that
\begin{align*}
& \left|\Xi(t) - \sum_{n=0}^N (-1)^{2n}d_{2n} g_{2n}
\right|
\\ &\leq
\int_0^\infty
\frac{\omega(x)}{(x+1)^{3/2}}
\Bigg|
\sum_{n=2N+1}^\infty
i^n\frac{(3/2)_n}{2^{n-3/2} n!} \left(\frac{x-1}{x+1}\right)^n 
\gausshyper\left(\frac{n}{2}+\frac34,
\frac{n}{2}+\frac54; n+2; \left(\frac{x-1}{x+1}\right)^2 \right)
g_n\left(\frac{t}{2}\right) \Bigg|
\,dx
\\ &\leq
\int_0^\infty
\frac{\omega(x)}{(x+1)^{3/2}}
\sum_{n=2N+1}^\infty
\frac{(3/2)_n}{2^{n-3/2} n!} \left|\frac{x-1}{x+1}\right|^n 
\left|\gausshyper\left(\frac{n}{2}+\frac34,
\frac{n}{2}+\frac54; n+2; \left(\frac{x-1}{x+1}\right)^2 \right)
\right|\cdot
\left|g_n\left(\frac{t}{2}\right)\right|
\,dx
\\ &=
\sum_{n=2N+1}^\infty
\frac{(3/2)_n}{2^{n-3/2} n!} 
\left|g_n\left(\frac{t}{2}\right)\right|
\int_0^\infty
\frac{\omega(x)}{(x+1)^{3/2}}
\left|\frac{x-1}{x+1}\right|^n 
\left|\gausshyper\left(\frac{n}{2}+\frac34,
\frac{n}{2}+\frac54; n+2; \left(\frac{x-1}{x+1}\right)^2 \right)
\right|
\,dx.
\end{align*}
Appealing to \eqref{eq:gn-easybound} (with a fixed compact set $K$ on which we are allowing $t$ to range) and finally to \eqref{eq:gn-easybound2}, we see that this last expression is bounded by
\begin{align*}
\sum_{n=2N+1}^\infty
&
\frac{(3/2)_n}{2^{n-3/2} n!} 
C_1 e^{C_2 n^{1/3}}
\int_0^\infty
\frac{\omega(x)}{(x+1)^{3/2}}
\left|\frac{x-1}{x+1}\right|^n 
\left|\gausshyper\left(\frac{n}{2}+\frac34,
\frac{n}{2}+\frac54; n+2; \left(\frac{x-1}{x+1}\right)^2 \right)
\right|
\,dx
\\ &\leq
\sum_{n=2N+1}^\infty
\frac{(3/2)_n}{2^{n-3/2} n!} 
C_1 e^{C_2 n^{1/3}}
\times J_1 e^{-J_2 n^{2/3}}
= O(e^{-\frac{J_2}{2} n^{2/3}})
\end{align*}
as $n\to\infty$;
this gives \eqref{eq:gn-expansion-errorbound} and finishes the proof.
\qed

\section{Asymptotic analysis of the coefficients $d_{2n}$}

\label{sec:gnexp-proofasym}

In this section we prove Theorem~\ref{thm:gn-coeff-asym}. We will give two independent proofs of this result, one relying on the representation \eqref{eq:dn-cn-expansion} of the coefficients $d_{2n}$ in terms of the coefficients $c_{2k}$---whose asymptotic behavior we already analyzed---and another relying on a separate representation of $d_{2n}$ as a double integral, which seems of independent interest.

\begin{proof}[First proof of Theorem~\ref{thm:gn-coeff-asym}]
Our starting point is the formula \eqref{eq:dn-cn-expansion}. We start by rewriting this relation in a form that's slightly more convenient for asymptotics, namely as
\begin{align*}
d_{2n} 
&=
\frac{2n+1}{2^{2n}} \sum_{m=0}^\infty \frac{(3/2)_{2n+2m}}{2^{2m}m!(2n+m+1)!} c_{2n+2m}
\\ & =
\frac{2n+1}{2^{2n}} \sum_{m=0}^\infty \frac{(4n+4m+2)!}{2^{4n+6m+1}m!(2n+m+1)!(2n+2m+1)!} c_{2n+2m}
\\ & =
\frac12(2n+1) \sum_{k=n}^\infty \frac{1}{2^{6k}} \frac{(4k+2)!}{(k-n)!(k+n+1)!(2k+1)!} c_{2k},
\end{align*}
substituting $k=n+m$ in the last step.
Making use of \eqref{eq:fn-coeff-asym}, we get that
\begin{align*}
d_{2n} 
&= 
\left(1+O\left(n^{-1/10}\right)\right) 128 \sqrt{2} \pi^{3/2} n \sum_{k=n}^\infty 
\frac{k^{3/2}}{k+n} \cdot
\frac{(4k)!}{2^{6k} (k-n)!(k+n)!(2k)!}
e^{-4\sqrt{\pi k}}
\\ &= 
\left(1+O\left(n^{-1/10}\right)\right) 128 \sqrt{2} \pi^{3/2} n \sum_{k=n}^\infty 
\frac{k^{3/2}}{k+n} \cdot
\frac{1}{2^{4k}} \binom{4k}{2k}\times \frac{1}{2^{2k}} \binom{2k}{k-n}
e^{-4\sqrt{\pi k}},
\end{align*}
where for convenience the terms have been simplified slightly by making use of trivial approximations such as $4k+2 = (1+O(n^{-1}))4k$, etc.; the errors in these approximations are absorbed into the leading $\left(1+O\left(n^{-1/10}\right)\right)$ factor. By Stirling's approximation, the binomial coefficients in the summand have asymptotic behavior
\begin{align*}
\frac{1}{2^{4k}}\binom{4k}{2k} & = \left(1+O\left(n^{-1}\right)\right) \frac{1}{\sqrt{2\pi k}}
\ \qquad\qquad \qquad \qquad \qquad \qquad \ (k\ge n, n\to\infty),
\\
\frac{1}{2^{2k}}\binom{2k}{k-n} & = 
\left(1+O\left(\frac{1}{k-n}\right)\right)
\frac{\sqrt{k}}{\sqrt{\pi(k-n)(k+n)}}
\\ & \qquad \qquad \times
\left(
\left( \frac{k-n}{k}\right)^{-(k-n)}
\left( \frac{k+n}{k}\right)^{-(k+n)}
\right)
\qquad (k\ge n, n\to\infty),
\end{align*}
Here, the error term $\left(1+O(k-n)^{-1}\right)$ is slightly bothersome as it makes it necessary to separately bound the summands for values of $k$ near $n$, but this is easy enough to do: observe that if $n\le k\le 2n$ then $k-n \le k/2$, and in this case we have for some constant $C>0$ independent of $n$ that
\begin{equation*}
\binom{2k}{k-n} \leq \binom{2k}{\lceil k/2 \rceil} \leq C(1.8)^{2k},
\end{equation*}
using Stirling's formula or a well-known bound such as \cite[p.~113, Eq.~(4.7.1)]{ash}. Thus, combining the latest estimates we obtain the expression
\begin{align}
\label{eq:d2n-largedev-sum}
d_{2n} &=
\left(1+O\left(n^{-1/10}\right)\right) 128 \sqrt{\pi} n 
\\ \nonumber & \quad\times
\Bigg[
\sum_{k=2n}^\infty 
\frac{k^{3/2}}{(k+n)^{3/2}(k-n)^{1/2}} 
\left(
\left( \frac{k-n}{k}\right)^{-(k-n)}
\left( \frac{k+n}{k}\right)^{-(k+n)}
\right)
e^{-4\sqrt{\pi k}}
+ O\left((0.9)^{2n}\right) \Bigg]
\\ \nonumber & =
\left(1+O\left(n^{-1/10}\right)\right) 128 \sqrt{\pi} n 
\sum_{k=2n}^\infty 
\frac{k^{3/2}}{(k+n)^{3/2}(k-n)^{1/2}} \exp\left(n^{2/3}\phi_n\left(\frac{k}{n^{4/3}}\right)\right),
\end{align}
where we denote
\begin{equation*}
\phi_n(t) = 
\frac{- (n^{4/3}t-n)\log\left(1-\frac{n^{-1/3}}{t}\right) - (n^{4/3}t+n)\log\left(1+\frac{n^{-1/3}}{t}\right)}{n^{2/3}}
- 4\sqrt{\pi t}.
\end{equation*}
We are now in a position to apply what is essentially a variant of Laplace's method in the setting of a discrete sum.
The following claims about the functions $\phi_n(t)$ are clearly relevant.
\begin{lem}
(i) The inequality
\begin{equation}
\label{eq:phin-upperbound}
\phi_n(t) \leq F(t) := -\frac{1}{t}-4\sqrt{\pi t}
\end{equation}
holds for all $n\ge1$ and $t\ge 2 n^{-1/3}$.

(ii) We have the asymptotic relation
\begin{equation}
\label{eq:phin-approximation}
\phi_n(t) = 
F(t) - \frac{1}{6t^3} n^{-2/3}
+ O\left( \frac{1}{n^{4/3}}
\right)
\qquad \textrm{as }n\to\infty \ \ \left(\textrm{with }\frac{1}{10}\le t \le 10\right),
\end{equation}
where the constant implicit in the big-$O$ is independent of $n$ and $t$, subject to the specified constraint.
\end{lem}

\begin{proof}
Consider the function of a real variable $0<x<1$ given by
\begin{equation*}
p(x) = -\left(\frac{1}{x}-1\right) \log(1-x) - \left(\frac{1}{x}+1\right)\log(1+x).
\end{equation*}
It is easy to verify that $p(x)$ has the Taylor expansion
\begin{equation*}
p(x) = 
-x - \frac{x^3}{6} - \frac{x^5}{3\times 5}
- \frac{x^7}{4\times 7}
- \frac{x^9}{5\times 9} - \ldots
= -\sum_{m=1}^\infty \frac{x^{2m-1}}{m(2m-1)}.
\end{equation*}
In particular, $p(x)\leq -x$ for all $0\le x<1$.
Substituting $x=1/(n^{1/3} t)$ gives the first claim of the lemma, and the second claim is obtained from the same substitution applied to the fact that $p(x) = -x-\frac{1}{6}x^3 + O(x^5)$ for $0\le x\le 1/2$.
\end{proof}
Some additional easy facts to note are that the function $F(t)$ has a unique global maximum at $t = \alpha_0:=(4\pi)^{-1/3}$; that $F(t)$ is increasing on $(0,\alpha_0)$ and decreasing on $(\alpha_0,\infty)$; that $F(\alpha_0) = -E$ (where $E$ is defined in \eqref{eq:gncoeffs-asym-consts}), and that $F''(\alpha_0) = -6\pi$. In particular, we have the Taylor expansion
\begin{equation}
\label{eq:foft-taylor}
F(t) = -E - 3\pi (t-\alpha_0)^2 + O\left(|t-\alpha_0|^3\right)
\qquad \left(\frac{1}{10} \le t\le 10\right).
\end{equation}

Now, split up the sum in \eqref{eq:d2n-largedev-sum} (without the leading numerical constant) into four parts, representing it as $S_n^{(1)}+S_n^{(2)}+S_n^{(3)}+S_n^{(4)}$, where
\begin{align*}
S_n^{(1)} & = 
\sum_{k\,:\,2n\le k < \alpha_0 n^{4/3}- n^{19/18}}
\frac{k^{3/2}}{(k+n)^{3/2}(k-n)^{1/2}} 
\exp\left(n^{2/3}\phi_n\left(\frac{k}{n^{4/3}}\right)\right),
\\
S_n^{(2)} & = 
\sum_{k\,:\,|k- \alpha_0 n^{4/3}|\le n^{19/18}}
\frac{k^{3/2}}{(k+n)^{3/2}(k-n)^{1/2}} 
\exp\left(n^{2/3}\phi_n\left(\frac{k}{n^{4/3}}\right)\right),
\\
S_n^{(3)} & = 
\sum_{k\,:\,\alpha_0 n^{4/3}+ n^{19/18} < k \le 2n^{4/3}}
\frac{k^{3/2}}{(k+n)^{3/2}(k-n)^{1/2}} 
\exp\left(n^{2/3}\phi_n\left(\frac{k}{n^{4/3}}\right)\right),
\\
S_n^{(4)} & = 
\sum_{k\,:\,k > 2n^{4/3}}
\frac{k^{3/2}}{(k+n)^{3/2}(k-n)^{1/2}} 
\exp\left(n^{2/3}\phi_n\left(\frac{k}{n^{4/3}}\right)\right).
\end{align*}
Of these four sums, it is $S_n^{(2)}$ that makes the asymptotically most significant contribution. Making use of \eqref{eq:phin-approximation} and \eqref{eq:foft-taylor}, we can estimate it for large $n$ as
\begin{align*}
S_n^{(2)} & =
\sum_{|k- \alpha_0 n^{4/3}|\le n^{19/18}}
\frac{k^{3/2}}{(k+n)^{3/2}(k-n)^{1/2}} 
\exp\left[ n^{2/3}\left( F\left(\frac{k}{n^{4/3}}\right)
- \frac{n^{10/3}}{6 k^3} + O\left(n^{-4/3}\right)
\right)\right]
\\ & =
\left(1+O\left(n^{-2/3}\right)\right)
\sum_{|k- \alpha_0 n^{4/3}|\le n^{19/18}}
\frac{k^{3/2}}{(k+n)^{3/2}(k-n)^{1/2}} 
\\ & \qquad\qquad\qquad\qquad\qquad\qquad
\times
\exp\left[ n^{2/3} F\left(\frac{k}{n^{4/3}}\right)
- \frac16 \alpha_0^{-3}\left(1+O\left(n^{-5/18}\right)\right)
\right]
\\ & =
\left(1+O\left(n^{-5/18}\right)\right)e^{-2\pi/3}
\sum_{|k- \alpha_0 n^{4/3}|\le n^{19/18}}
\frac{k^{3/2}}{(k+n)^{3/2}(k-n)^{1/2}} 
\exp\left( n^{2/3} F\left(\frac{k}{n^{4/3}}\right)
\right)
\\ & =
\left(1+O\left(n^{-5/18}\right)\right)e^{-2\pi/3}
\left(\alpha_0 n^{4/3}\right)^{-1/2}
\\ &
\qquad\qquad\qquad\qquad\quad\times
\sum_{|k- \alpha_0 n^{4/3}|\le n^{19/18}}
\exp\left( -E n^{2/3} 
-3\pi n^{2/3} \left(\frac{k}{n^{4/3}}-\alpha_0\right)^2
+O\left( n^{-1/6} \right)
\right)
\\ & =
\left(1+O\left(n^{-1/6}\right)\right)e^{-E n^{2/3}-2\pi/3}
\alpha_0^{-1/2}
n^{-2/3}
\sum_{|k- \alpha_0 n^{4/3}|\le n^{19/18}}
\exp\left(
-3\pi \left(\frac{k-\alpha_0 n^{4/3}}{n}\right)^2
\right).
\end{align*}
The sum in this last expression can be regarded in the usual way as a Riemann sum for a Gaussian integral; specifically, it is asymptotically equal to
\begin{align*}
\left(1+O\left(n^{-1}\right)\right) n \int_{-n^{1/18}}^{n^{1/18}} e^{-3\pi u^2}\,du
& =
\left(1+O\left(n^{-1}\right)\right) n \left(
\frac{1}{\sqrt{3}}-O\left(e^{-n^{1/9}}\right)
\right)
=
\left(1+O\left(n^{-1}\right)\right) \frac{n}{\sqrt{3}}
\end{align*}
as $n\to\infty$ (again making use of \eqref{eq:gaussian-tail-bound} to justify the first transition).
Thus, we have obtained the relation
\begin{equation*}
S_n^{(2)} =
\left(1+O\left(n^{-1/6}\right)\right)
\frac{\alpha_0^{-1/2}}{\sqrt{3}}
n^{1/3} e^{-E n^{2/3}-2\pi/3}
\qquad (n\to\infty).
\end{equation*}
Next, we bound the sum $S_n^{(1)}$ to show that its contribution is negligible compared to that of $S_n^{(2)}$. The polynomial-order factor appearing in front of the exponential term in the sum is bounded from above by $1$. Thus, by \eqref{eq:phin-upperbound} and the fact that $F(t)$ is increasing on $(0,\alpha_0)$ we have that, as $n\to\infty$,
\begin{align*}
0\leq S_n^{(1)} & \leq
\sum_{2n\le k < \alpha_0 n^{4/3}- n^{19/18}}
\exp\left(n^{2/3}\phi_n\left(\frac{k}{n^{4/3}}\right)\right)
\\ & \leq
\sum_{2n \le k< \alpha_0 n^{4/3}- n^{19/18}}
\exp\left(n^{2/3}F\left(\frac{k}{n^{4/3}}\right)\right)
\\ & \leq
\alpha_0 n^{4/3} \exp\left(
n^{2/3}F\left(\alpha_0-n^{-5/18}\right)
\right)
\\ &\leq
\alpha_0 n^{4/3}  \exp\left(-E n^{2/3}
-3\pi n^{1/9}
+ O\left(n^{-1/6}
\right)
\right)
=
O\left(e^{-E n^{2/3}-n^{1/9}}\right).
\end{align*}
The third sum $S_n^{(3)}$ can be bounded in a completely analogous fashion, resulting (the reader can easily check) in the same bound
\begin{equation*}
S_n^{(3)} =
O\left(e^{-E n^{2/3}-n^{1/9}}\right).
\end{equation*}
Finally, to bound $S_n^{(4)}$, we use the fact that $F(t)\leq -4\sqrt{\pi t}$ to write
\begin{align*}
0\leq S_n^{(4)} &\leq
\sum_{k> 2 n^{4/3}} \frac{10}{k^{1/2}}
\exp\left(n^{2/3} F\left(\frac{k}{n^{4/3}}\right)\right)
\leq
\sum_{k> 2 n^{4/3}} \frac{10}{k^{1/2}}
\exp\left(-4\sqrt{\pi k}\right)
\\ &\leq
10\int_{2n^{4/3}}^\infty e^{-4\sqrt{\pi x}}\,dx
=
\frac{10}{8\pi} \left( 4 \sqrt{2\pi} n^{2/3}+1 \right)
\exp\left(-4\sqrt{2\pi} n^{2/3}\right)
=
O\left(e^{-E n^{2/3}-n^{1/9}}\right).
\end{align*}
Combining the above estimates for $S_n^{(1)}$, $S_n^{(2)}$, $S_n^{(3)}$ and $S_n^{(4)}$, we have finally from \eqref{eq:d2n-largedev-sum} that
\begin{align*}
d_{2n}
& = 
\left(1+O\left(n^{-1/10}\right)\right)
\left(128 \sqrt{\pi} n \right)
\left(\frac{\alpha_0^{-1/2}}{\sqrt{3}}
n^{1/3} e^{-E n^{2/3}-2\pi/3}\right) \quad \textrm{as }n\to\infty,
\end{align*}
which, after a trivial reshuffling of the terms, is exactly 
\eqref{eq:gn-coeff-asym}.
\end{proof}

\begin{proof}[Second method for proving Theorem~\ref{thm:gn-coeff-asym}]
We give most of the details of a second proof of Theorem~\ref{thm:gn-coeff-asym}, except for the rate of convergence result, which we weaken to a less explicit $1+o(1)$ multiplicative error term. This seems of independent interest as it highlights yet another way of approaching the study of the coefficients $d_{2n}$. This proof requires some calculations that would be tedious to perform by hand, but are easily done using a computer algebra system (we used Mathematica). We omit the details of these calculations and a few other details needed to make the proof watertight, which may be filled in by an enthusiastic reader.

We start by deriving a new representation of $d_{2n}$ suitable for asymptotic analysis. Start with the formula \eqref{eq:def-gn-coeffs} for $d_{2n}$ in a slightly modified form  
\begin{align*}
d_{2n} & = \frac{(3/2)_{2n}}{2^{2n-5/2} (2n)!}\int_1^\infty \frac{\omega(x)}{(x+1)^{3/2}} \left(\frac{x-1}{x+1}\right)^{2n} 
\gausshyper\left(n+\frac34,
n+\frac54; 2n+2; \left(\frac{x-1}{x+1}\right)^2 \right)
 \,dx
\end{align*}
in which the integration is performed on $(1,\infty)$ (this follows from \eqref{eq:def-gn-coeffs} by the same symmetry under the change of variables $u=1/x$ as in \eqref{eq:omega-int-powersn-symmetry}, a consequence of the functional equation \eqref{eq:theta-functional-equation}). Now use Euler's integral representation
\begin{equation*}
\gausshyper(a,b;c;z) = \frac{\Gamma(c)}{\Gamma(b)\Gamma(c-b)}
\int_0^1 t^{b-1} (1-t)^{c-b-1} \frac{1}{(1-zt)^a}\,dt
\end{equation*}
for the Gauss hypergeometric function (see \cite[p.~65]{andrews-askey-roy}) to represent the $\gausshyper$ term inside the integral. This gives
\begin{align*}
d_{2n} &= 
\frac{(3/2)_{2n}}{2^{2n-5/2} (2n)!}
\frac{\Gamma(2n+2)}{\Gamma(n+3/4)\Gamma(n+5/4)}
\\ & \quad \times
\int_1^\infty \int_0^1 \frac{\omega(x)}{(x+1)^{3/2}} \left(\frac{x-1}{x+1}\right)^{2n} 
 t^{n+1/4}(1-t)^{n-1/4} 
\left(1-
\left(\frac{x-1}{x+1}\right)^2 t\right)^{-(n+3/4)} \,dt
\,dx.
\end{align*}
As the reader can check, the constant in front of the integral simplifies to
\begin{equation*}
\frac{(2n+1)(3/2)_{2n}}{2^{2n-5/2}\Gamma(n+3/4)\Gamma(n+5/4)}
= \frac{16}{\pi}(2n+1).
\end{equation*}
Thus, after some further trivial algebraic manipulations we arrive at the representation
\begin{align}
\label{eq:d2n-simplerep-doubleint}
d_{2n} &= 
\frac{16}{\pi} (2n+1)
\int_1^\infty \int_0^1 
\frac{\omega(x)}{((x+1)^2-t(x-1)^2)^{3/4}}
\left(\frac{t}{1-t}\right)^{1/4}
\left(
\frac{t(1-t)(x-1)^2}{(x+1)^2-t(x-1)^2}
\right)^n
\,dt\,dx.
\end{align}
Recalling 
\eqref{eq:omegaxdiff-asym-xinfty},
we see that it makes sense to write
\begin{equation}
\label{eq:d2n-etan-mun}
d_{2n} = \frac{16}{\pi}(2n+1)(R_n + \mu_n),
\end{equation}
where we define the quantities $R_n, \mu_n$ by 
\begin{align}
\label{eq:def-etan-doubleint}
R_n &=
\int_1^\infty
\int_0^1 \frac{\pi x(2\pi x-3)}{((x+1)^2-t(x-1)^2)^{3/4}}
\left(\frac{t}{1-t}\right)^{1/4}
e^{-\pi x}
\left(
\frac{t(1-t)(x-1)^2}{(x+1)^2-t(x-1)^2}
\right)^n
\,dt\,dx,
\\
\mu_n &=
\int_1^\infty
\int_0^1 \frac{\omega(x) - \pi x(2\pi x-3)e^{-\pi x}}{((x+1)^2-t(x-1)^2)^{3/4}}
\left(\frac{t}{1-t}\right)^{1/4}
\left(
\frac{t(1-t)(x-1)^2}{(x+1)^2-t(x-1)^2}
\right)^n
\,dt\,dx.
\end{align}
It will be enough to obtain the asymptotic behavior of $R_n$ as $n\to\infty$, and separately to show that $\mu_n$ is asymptotically negligible compared to $R_n$.

\paragraph{\textbf{Part 1: deriving asymptotics for $R_n$.}}
Define functions
\begin{align*}
g(t,x) &= \frac{\pi x(2\pi x-3)}{((x+1)^2-t(x-1)^2)^{3/4}}
\left(\frac{t}{1-t}\right)^{1/4}, \\
h_n(t,x) &= \frac{n}{\pi} \log\left( \frac{t(1-t)(x-1)^2}{(x+1)^2-t(x-1)^2}\right)-x
= M \log\left( \frac{t(1-t)(x-1)^2}{(x+1)^2-t(x-1)^2}\right) - x,
\end{align*}
where for convenience throughout the proof we denote $M= \frac{n}{\pi}$.
Then $R_n$ can be rewritten in the form
\begin{equation}
\label{eq:etan-doubleint-largedev}
R_n = \int_1^\infty \int_0^1 g(t,x)\exp\left(\pi h_n(t,x)\right) \,dt\,dx.
\end{equation}
This form is suitable for applying a two-dimensional version of Laplace's method. The method consists of identifying the global minimum point of $h_n(\cdot,\cdot)$ and analyzing the second-order Taylor expansion of $h_n$ around the minimum point. We will need the partial derivatives of $h_n(\cdot,\cdot)$ up to second order, which after some calculation are found to be
\begin{align}
\label{eq:hn-partial-x}
\frac{\partial h_n}{\partial x} &= 
\frac{4M(x+1)-(x-1)((x+1)^2-t(x-1)^2)}{(x-1)((x+1)^2-t(x-1)^2)},
\\
\label{eq:hn-partial-t}
\frac{\partial h_n}{\partial t} &=
-M\cdot \frac{(2t-1)(x+1)^2 - t^2(x-1)^2}{t(1-t)((x+1)^2-t(x-1)^2},
\\
\label{eq:hn-partial-tt}
\frac{\partial^2 h_n}{\partial t^2} &=
{\scriptstyle -M \frac{t^4(x-1)^4 + (x+1)^4 - 4t^3(x^2-1)^2-4t(x+1)^2(x^2+1)+2t^2(x+1)^2(3x^2-2x+3)}{t^2(1-t)^2((x+1)^2-t(x-1)^2)^2},}
\\
\label{eq:hn-partial-xx}
\frac{\partial^2 h_n}{\partial x^2} &=
-8M \frac{x(x+1)^2-t(x^3-3x+2)}{(x-1)^2((x+1)^2-t(x-1)^2)},
\\
\label{eq:hn-partial-tx}
\frac{\partial^2 h_n}{\partial t \partial x} &=
4 M \frac{x^2-1}{((x+1)^2-t(x-1)^2)^2}.
\end{align}
To find the minimum point, we solve the equations $\frac{\partial h_n}{\partial t}=\frac{\partial h_n}{\partial t}=0$. By \eqref{eq:hn-partial-x}--\eqref{eq:hn-partial-t}, this gives the system of two equations
\begin{align}
\label{eq:hn-saddlepoint1}
4M(x+1)-(x-1)((x+1)^2-t(x-1)^2) &= 0, \\
\label{eq:hn-saddlepoint2}
(2t-1)(x+1)^2 - t^2 (x-1)^2 &= 0.
\end{align}
Solving \eqref{eq:hn-saddlepoint1} (a linear equation in $t$) for $t$ gives the relation
\begin{equation}
\label{eq:hn-saddlepoint1-sol}
t = \frac{(x+1)(x^2-4M-1)}{(x-1)^3} .
\end{equation}
Substituting this value back into \eqref{eq:hn-saddlepoint2} gives the equation
\begin{equation*}
\frac{4(x+1)^2}{(x-1)^4}(x(x-1)^2-4M^2) = 0.
\end{equation*}
That is, $x$ has to satisfy the cubic equation
\begin{equation*}
x(x-1)^2-4M^2 = 0.
\end{equation*}
For $M\ge 1$, one can check that the cubic has a single real solution, given by
\begin{equation}
\label{eq:hn-saddlepoint-solx}
x = \frac{\left(\left(54M^2-1+6M\sqrt{3(27M^2-1)}\right)^{1/3}+1\right)^2}{3\left(54M^2-1+6M\sqrt{3(27M^2-1)}\right)^{1/3}}.
\end{equation}
The corresponding $t$ value is given by \eqref{eq:hn-saddlepoint1-sol}, which, for $x$ given by \eqref{eq:hn-saddlepoint-solx}, can be brought to the slightly simpler form
\begin{equation*}
t = \frac{1}{2M^3}\left[ (2M-1)x^2 + (-2M^2+1)x+2M^2(M-2) \right].
\end{equation*}
Summarizing the above remarks, define quantities
\begin{align}
\alpha_n &= 54M^2-1+6M\sqrt{3(27M^2-1)}, \\
\label{eq:hn-saddlept-xin}
\xi_n &= \frac{(\alpha_n^{1/3}+1)^2}{3\alpha_n^{1/3}},
\\
\tau_n &= \frac{1}{2M^3}\left((2M-1)\xi_n^2+(-2M^2+1)\xi_n + 2M^2(M-2) \right).
\end{align}
Then $(\tau_n,\xi_n)$ is the unique solution of the equations
\begin{equation*}
\frac{\partial h_n}{\partial x}(\tau_n,\xi_n) = 0, \qquad
\frac{\partial h_n}{\partial t}(\tau_n,\xi_n) = 0.
\end{equation*}
Using these formulas one can now also find the asymptotic behavior of $\xi_n$ and $\tau_n$ as $M\to\infty$, which is given by
\begin{align*}
\tau_n &= 1- 2^{2/3} M^{-1/3} + 2^{4/3} M^{-2/3} - \frac83 M^{-1} + O(M^{-4/3}), \\
\xi_n &= 2^{2/3} M^{2/3} + \frac23 + \frac{1}{9\times 2^{2/3}} M^{-2/3} + O(M^{-4/3}).
\end{align*}
In particular, note that for large $M$ (that is, for large $n$) we have $\xi_n > 1$, $0 < \tau_n < 1$. That is, the point $(\xi_n,\tau_n)$ lies in the (interior of) the region of integration in the expression \eqref{eq:etan-doubleint-largedev} for $R_n$.

Next, having found the values $(\tau_n,\xi_n)$, we want to understand the values $h_n(\tau_n,\xi_n)$, $\frac{\partial^2 h_n}{\partial t^2}(\tau_n,\xi_n)$, $\frac{\partial^2 h_n}{\partial x^2}(\tau_n,\xi_n)$, $\frac{\partial^2 h_n}{\partial t \partial x}(\tau_n,\xi_n)$. These are somewhat complicated numbers, but can be brought to simpler forms by taking the relevant rational functions in $\tau_n, \xi_n$, expressing them as rational functions of $\xi_n$ only using \eqref{eq:hn-saddlepoint1-sol}, and then performing polynomial reduction modulo the polynomial $\xi_n(\xi_n-1)^2-4M^2$ (the cubic polynomial of which $\xi_n$ is a root). Using Mathematica to perform the reduction, we arrived at the following simplified formulas:
\begin{align*}
& \hspace{-50pt} \frac{\tau_n(1-\tau_n)(\xi_n-1)^2}{(\xi_n+1)^2-\tau_n(\xi_n-1)^2} \\ & 
= \frac{1}{M^3}\left[(2M-1)\xi_n^2+(-2M^2+1)\xi_n+M^2(M-4)\right] = 2\tau_n-1,
\\
h_n(\tau_n,\xi_n) &=
M \log\left(
\frac{\tau_n(1-\tau_n)(\xi_n-1)^2}{(\xi_n+1)^2-\tau_n(\xi_n-1)^2}
\right)-\xi_n 
= \xi_n-M\log(2\tau_n-1) - \xi_n,
\\
\frac{\partial^2 h_n}{\partial t^2}(\tau_n,\xi_n)
&=
-\frac{1}{2(M^2+1)}\left[ (M^2+3)\xi_n^2 + 2(2M^2-1)\xi_n
+ (8M^3-9M^2+8M-1) \right],
\\
\frac{\partial^2 h_n}{\partial x^2}(\tau_n,\xi_n)
&=
-\frac{1}{4M^2(M^2+1)}\left(
(2M^2+3)\xi_n^2-3\xi_n-4M(M^2+M+1)
\right),
\\
\frac{\partial^2 h_n}{\partial t \partial x}(\tau_n,\xi_n)
&=
\frac{1}{4M(M^2+1)}\left(
(M^2-1)\xi_n^2-2(2M^2-1)\xi_n+(7M^2-1)
\right).
\end{align*}
Finally, the Hessian
\begin{equation*}
\Delta_n := 
\frac{\partial^2 h_n}{\partial t^2}(\tau_n,\xi_n)
\frac{\partial^2 h_n}{\partial x^2}(\tau_n,\xi_n)
-
\left(\frac{\partial^2 h_n}{\partial t \partial x}(\tau_n,\xi_n)\right)^2
\end{equation*}
can be found to be expressible by the (still ungainly) formula
\begin{align*}
\Delta_n &= {\textstyle \frac{
(24M^3-17M^2+24M-1)\xi_n^2
+2(6M^4-16M^3+31M^2-16M+1)
+(56M^4+8M^3-9M^2+8M-1)
}{16M^2(M^2+1)}.}
\end{align*}
From these expressions and \eqref{eq:hn-saddlept-xin}, we derive some additional useful asymptotic expansions:
\begin{align}
\label{eq:hn-partialxx-asym}
\frac{\partial^2 h_n}{\partial x^2}(\tau_n,\xi_n)
&= -2^{1/3} M^{-2/3} + O(M^{-1}), \\
\label{eq:hn-deltan-asym}
\Delta_n &= \frac{3}{2^{4/3}} M^{2/3} + O(M^{-1/3}), \\
\label{eq:hn-deltan-minushalf}
\frac{1}{\sqrt{\Delta_n}} &= 
\frac{2^{2/3}}{\sqrt{3}} M^{-1/3} - \frac{2^{4/3}}{\sqrt{3}} M^{-2/3} + \frac{10}{3\sqrt{3}} M^{-1} + O(M^{-4/3}),
\\
\label{eq:hn-taun-xin}
h_n(\tau_n,\xi_n) &=
-3\times 2^{2/3} M^{2/3} - \frac23 - \frac{1}{15\times 2^{2/3}} M^{-2/3} + O(M^{-4/3}).
\end{align}
One additional quantity we need to understand is
\begin{equation}
\label{eq:g-taun-xin}
g(\tau_n,\xi_n) = \frac{\pi \xi_n(2\pi \xi_n-3)}{((\xi_n+1)^2-\tau_n(\xi_n-1)^2)^{3/4}} \left(\frac{\tau_n}{1-\tau_n}\right)^{1/4}.
\end{equation}
This can be written as
\begin{equation*}
g(\tau_n,\xi_n) = \pi \xi_n(2\pi\xi_n-3)X_n^{1/4} Y_n^{3/4}
\end{equation*}
where we define
\begin{equation*}
X_n = \frac{\tau_n}{1-\tau_n}, \qquad 
Y_n = \frac{1}{(\xi_n+1)^2-\tau_n(\xi_n-1)^2}.
\end{equation*}
Some more algebraic simplification then shows that
\begin{equation*}
X_n = \frac{1}{4M}(\xi_n^2-1),
\qquad
Y_n = \frac{1}{8M(M^2+1)}(-\xi_n^2+3\xi_n + 2(M^2-1)).
\end{equation*}
Using these relations, we then get the asymptotic expansion
\begin{equation*}
g(\tau_n,\xi_n) = 2^{2/3} \pi^2 M^{2/3} - \frac16 \pi (9-\pi) + O(M^{-2/3}).
\end{equation*}
Now note that \eqref{eq:hn-partialxx-asym} and \eqref{eq:hn-deltan-asym} imply that (for large $n$) the Hessian matrix of $h_n$ at $(\tau_n,\xi_n)$ is negative-definite. Thus, $(\tau_n,\xi_n)$ is indeed a local maximum point of $h_n$. We leave to the reader to check that it is in fact a \emph{global} maximum.

Now recall that the two-dimensional version of Laplace's method gives the asymptotic formula
\begin{equation*}
(1+o(1)) \frac{2}{\sqrt{\Delta_n}} g(\tau_n,\xi_n)
\exp\left(\pi h_n(\tau_n,\xi_n) \right),
\end{equation*}
for the integral on the right-hand side of \eqref{eq:etan-doubleint-largedev}. This arises by making a suitable change of variables in the integral to center it around the point $(\xi_n,\tau_n)$ and introduce scaling that turns the integral to an approximate Gaussian integral---see \cite[Ch.~VIII]{wong} for details; we omit the derivation of bounds needed to rigorously justify the approximation.
Substituting the asymptotic values found in 
\eqref{eq:hn-deltan-minushalf}--\eqref{eq:g-taun-xin} 
therefore gives that
\begin{align}
\label{eq:etan-asym}
R_n &= (1+o(1))2\times \left(\frac{2^{2/3}}{\sqrt{3}} M^{-1/3}\right) 2^{2/3} \pi^2 M^{2/3}
\exp\left(
- 3 \times 2^{2/3} \pi M^{2/3}-\frac{2\pi}{3}
\right)
\\ \nonumber &=
(1+o(1))\left(
2\times \frac{2^{2/3}}{\sqrt{3}}\pi^{1/3}\times
2^{2/3}\pi^2 \frac{1}{\pi^{2/3}} e^{-2\pi/3}
\right)
n^{1/3} 
\exp\left(-3\times 2^{2/3} \pi^{1/3} n^{2/3} \right)
\\ \nonumber &=
(1+o(1))\left( \frac{4\times 2^{1/3}}{\sqrt{3}} \pi^{5/3}
e^{-2\pi/3}
\right) n^{1/3}
\exp\left(-3 (4\pi)^{1/3} n^{2/3} \right).
\end{align}

\paragraph{\textbf{Part 2: bounding $\mu_n$.}}
The next step is to prove that the contribution of $\mu_n$ is asymptotically negligible relative to $R_n$. 
This relies as usual on \eqref{eq:omegaxdiff-asym-xinfty}.
We sketch the argument but leave the details to the interested reader to develop. 
Observe that by \eqref{eq:omegaxdiff-asym-xinfty}, $\mu_n$ satisfies a bound of the form
\begin{align*}
|\mu_n| &\leq 
C
\int_1^\infty
\int_0^1 g(t,x)\exp\left(\pi h_n(t,x)-2\pi x\right)
\,dt\,dx
=
C
\int_1^\infty
\int_0^1 g(t,x)\exp\left(\pi k_n(t,x)\right)
\,dt\,dx,
\end{align*}
for some constant $C>0$,
where we denote $k_n(t,x) = h_n(t,x)-2x$. But now $k_n(t,x)$ can be analyzed in a similar fashion to our analysis of $h_n(t,x)$ above. In particular, it can be shown that for $n$ large enough, $k_n(t,x)$ has a unique global maximum point $(t_n,x_n)\in (0,1)\times (1,\infty)$, and that the maximum value 
\begin{equation*}
K_n^* := k_n(t_n,x_n)
\end{equation*}
behaves asymptotically as
\begin{equation*}
K_n^* =  c_0 M^{2/3} + o\left(M^{2/3}\right)
\end{equation*}
for some constant $c_0$, where, significantly, $c_0 < -3\times 2^{2/3}$ (the leading constant in the analogous asymptotic expression \eqref{eq:hn-taun-xin} for the maximum value of $h_n(t,x)$). By deriving some auxiliary technical bounds for the decay of $h_n(t,x)$ away from its maximum point and near the boundaries of the integration region, one can then show that for any $\epsilon>0$, $\mu_n$ satisfies a bound of the form
\begin{equation*}
|\mu_n| = O\left(\exp\left(\pi^{1/3} (c_0+\epsilon) n^{2/3}\right)\right).
\end{equation*}
Taking $\epsilon < 3\times 2^{2/3}-c_0$ then gives a rate of growth that is smaller than the exponential rate of growth of $R_n$, establishing that $\mu_n \ll R_n$.

\paragraph{\textbf{Putting everything together.}} Combining the above discussion regarding $\mu_n$ with \eqref{eq:d2n-etan-mun} and~\eqref{eq:etan-asym}, we find that
\begin{align*}
d_{2n} & =
(1+o(1))
\left(\frac{16}{\pi} (2n+1)\right)\times 
\left( \frac{4\times 2^{1/3}}{\sqrt{3}} \pi^{5/3}
e^{-2\pi/3}
\right) n^{1/3}
\exp\left(-3 (4\pi)^{1/3} n^{2/3} \right),
\\ & =
(1+o(1))
\left( \frac{128\times 2^{1/3}}{\sqrt{3}} \pi^{2/3}
e^{-2\pi/3}
\right) n^{4/3}
\exp\left(-3 (4\pi)^{1/3} n^{2/3} \right),
\end{align*}
which is the same (except for the weaker rate of convergence estimate) as \eqref{eq:gn-coeff-asym}.

\end{proof}

\section[Connection to the function $\tilde{\nu}(t)$ and the Chebyshev polynomials]{Connection to the function $\tilde{\nu}(t)$ and the Chebyshev polynomials of the second kind}

\label{sec:gnexp-chebyshev}

We now prove yet another formula for $d_n$, tying it in a surprising way to the function $\tilde{\nu}(t)$ (discussed in
\secref{sec:rad-properties-nu}) and its expansion in yet another family of orthogonal polynomials, the Chebyshev polynomials of the second kind. The properties of these very classical polynomials, denoted $U_n(t)$, are summarized in \secref{sec:orth-chebyshev}.

\begin{prop}
The coefficients $d_n$ can be alternatively expressed as
\begin{equation}
\label{eq:dn-chebyshev-int}
d_n = (-1)^n \frac{4\sqrt{2}}{\pi} \int_{-1}^1 \tilde{\nu}(t) U_n(t)\sqrt{1-t^2}\,dt.
\end{equation}
\end{prop}

\begin{proof}
By the identity \eqref{eq:nutilde-taylor} expressing $\tilde{\nu}(t)$ as a power series with coefficients related to $c_n$, we have that
\begin{align}
\label{eq:nutilde-un-innerprod}
\int_{-1}^1 \tilde{\nu}(t) U_n(t)\sqrt{1-t^2}\,dt
& =
\frac{1}{2\sqrt{2}} 
\int_{-1}^1 
\left(\sum_{m=0}^\infty 
\frac{(-1)^m (3/2)_m}{m!} c_m 
t^m \right) U_n(t) \sqrt{1-t^2}\,dt
\\ \nonumber & =
\frac{1}{2\sqrt{2}} \sum_{m=0}^\infty 
\frac{(-1)^m (3/2)_m}{m!} c_m 
\int_{-1}^1 t^m U_n(t) \sqrt{1-t^2}\,dt.
\end{align}
The integrals in this last expression can be interpreted as inner products in the space $L^2((-1,1),\sqrt{1-t^2}\,dt)$ of the monomial $t^m$ with the Chebyshev polynomial $U_n(t)$, so they can be evaluated by using the relation \eqref{eq:chebyshev-monomial-expansion} to expand the monomial $t^m$ in the polynomials $U_j(t)$ and then using of the orthogonality relation \eqref{eq:chebyshev-orthogonality}. Together these relations imply that
\begin{equation*}
\int_{-1}^1 t^m U_n(t) \sqrt{1-t^2}\,dt
=
\begin{cases}
\frac{\pi}{2} \frac{1}{(m+1)2^m} (n+1)\binom{m+1}{k}
& \textrm{if $n=m-2k$ for some $k\ge0$,}
\\
0 & \textrm{otherwise}.
\end{cases}
\end{equation*}
Thus, we can rewrite the series in \eqref{eq:nutilde-un-innerprod} as
\begin{align*}
\frac{1}{2\sqrt{2}} \sum_{k=0}^\infty &
\frac{(-1)^{n+2k} (3/2)_{n+2k}}{(n+2k)!} c_{n+2k}
\times \frac{\pi}{2}
\frac{1}{(n+2k+1)2^{n+2k}}(n+1)\binom{n+2k+1}{k}
\\ & =
\frac{(-1)^n \pi}{4\sqrt{2}}\cdot \frac{n+1}{2^n} \sum_{k=0}^\infty 
\frac{(3/2)_{n+2k}}{2^{2k} k! (n+k+1)!} c_{n+2k}.
\end{align*}
By \eqref{eq:dn-cn-expansion} this gives precisely $\frac{(-1)^n \pi}{4\sqrt{2}} d_n$, so we are done.
\end{proof}

The last proposition leads naturally to another central result of this chapter, which, in a manner analogous to Theorem~\ref{thm:aofr-selftrans-expansion}, gives a thought-provoking alternative point of view regarding the significance of the coefficients $d_{2n}$.

\begin{thm}[Expansion of $\tilde{\nu}(t)$ in the Chebyshev polynomials of the second kind]
The function $\tilde{\nu}(t)$ has the series expansion
\begin{equation}
\label{eq:nutilde-chebyshev-exp}
\tilde{\nu}(t) = \frac{1}{2\sqrt{2}}\sum_{n=0}^\infty d_{2n} U_{2n}(t),
\end{equation}
The series in \eqref{eq:nutilde-chebyshev-exp} converges pointwise for all $t\in (-1,1)$ and \ in\  the sense of the \, function space $L^2((-1,1),\sqrt{1-t^2}\,dt)$.
\end{thm}

\begin{proof} 
From the general theory of orthogonal polynomial expansions, the function $\tilde{\nu}(t)$, being a continuous and bounded function on $(-1,1)$, has an expansion of the form
\begin{equation*}
\tilde{\nu}(t) = \sum_{n=0}^\infty \sigma_n U_n(t)
\end{equation*}
in the polynomials $U_n(t)$. The expansion converges in $L^2((-1,1),\sqrt{1-t^2}\,dt)$ and for all $t\in(-1,1)$, see \cite[Ch.~IX]{szego}. Using the orthogonality relation \eqref{eq:chebyshev-orthogonality}, the coefficients $\sigma_n$ can be extracted as $L^2$ inner products, namely
\begin{equation*}
\sigma_n = \frac{2}{\pi} \int_{-1}^1 \tilde{\nu}(t)U_n(t)\sqrt{1-t^2}\,dt,
\end{equation*}
and this is equal to $\frac{(-1)^n}{2\sqrt{2}} d_n$ by \eqref{eq:dn-chebyshev-int}.
\end{proof}

\section[Alternative interpretation for the $g_n$-expansion]{Mellin transform representation for $g_n(x)$ and an alternative interpretation for the $g_n$-expansion}

\label{sec:gnexp-mellin}

The next result gives a formula representing the polynomials $g_n(x)$ in terms of Mellin transforms involving the Chebyshev polynomials of the second kind evaluated at $\frac{x-1}{x+1}$. This representation, which we have not found described explicitly in the literature but is a special case of a more general result \cite[eq.~(3.4)]{koelink}, stands as an interesting parallel to the integral representation for $f_n(x)$ given in Proposition~\ref{prop:fn-mellintrans3}.

\begin{prop}[Mellin transform representation of $g_n$]
\label{prop:gn-mellintrans}
We have the relation
\begin{align}
\label{eq:gn-mellintrans-scoord}
\int_0^\infty & \frac{1}{(x+1)^{3/2}} U_n\left(\frac{x-1}{x+1}\right) x^{s-1}\,dx
\\ \nonumber & 
=
i^n \frac{2}{\sqrt{\pi}}
\Gamma(s)\Gamma\left(\frac32-s\right)
g_n\left(\frac{1}{i}\left(s-\frac34\right)\right)
\qquad \left(0 < \re(s) < \frac32\right),
\end{align}
or, equivalently,
\begin{equation}
\label{eq:gn-mellintrans-tcoord}
\int_0^\infty \frac{1}{(x+1)^{3/2}} U_n\left(\frac{x-1}{x+1}\right) x^{-\frac14+it}\,dx
=
i^n \frac{2}{\sqrt{\pi}}
\Gamma\left(\frac34+it\right)\Gamma\left(\frac34-it\right)
g_n(t).
\end{equation}
\end{prop}

\begin{proof}
We prove this in the equivalent form \eqref{eq:gn-mellintrans-tcoord}.
Use the expansion \eqref{eq:chebyshev-explicit2} of $U_n(t)$ in monomials and then the Mellin transform representation \eqref{eq:fn-mellintrans-tcoord} for $f_n(t)$, to get that
\begin{align*}
\int_0^\infty & \frac{1}{(x+1)^{3/2}} U_n\left(\frac{x-1}{x+1}\right) x^{-\frac14+it}\,dx
\\ & =
\sum_{k=0}^{\lfloor \frac{n}{2}\rfloor}
(-1)^k \binom{n-k}{k} 2^{n-2k}
\int_0^\infty \frac{1}{(x+1)^{3/2}} 
\left( \frac{x-1}{x+1}\right)^{n-2k}
x^{-\frac14+it}\,dx
\\ & =
\sum_{k=0}^{\lfloor \frac{n}{2}\rfloor}
(-1)^k \binom{n-k}{k} 2^{n-2k}
i^{n-2k} \frac{2(n-2k)!}{\sqrt{\pi} (3/2)_{n-2k}}
\Gamma\left(\frac34+it\right)\Gamma\left(\frac34-it\right)
f_{n-2k}
\\ & =
i^n \Gamma\left(\frac34+it\right)\Gamma\left(\frac34-it\right)
\frac{2}{\sqrt{\pi}}
\sum_{k=0}^{\lfloor \frac{n}{2}\rfloor}
\frac{2^{n-2k} (n-k)!}{k! (3/2)_{n-2k}} f_{n-2k}(t).
\end{align*}
By the relation \eqref{eq:gn-intermsof-fn} expressing $g_n(t)$ in terms of the $f_k$'s, this is equal to the expression on the right-hand side of \eqref{eq:gn-mellintrans-tcoord}.
\end{proof}

Recall that in \secref{sec:rad-alt-fnexp}
we showed how the $f_n$-expansion of the Riemann xi function can be thought of as arising from the expansion of the radial function $A(r)$ in the orthogonal basis $(G_n^{(3)}(r))_{n=0}^\infty$, by taking Mellin transforms. In a completely analogous manner, the above Mellin transform representation of $g_n(t)$ makes a similar reinterpretation possible for the $g_n$-expansion of $\Xi(t)$ as originating in the expansion \eqref{eq:nutilde-chebyshev-exp} of $\tilde{\nu}(x)$ in the Chebyshev polynomials of the second kind. To see this, first recall the Mellin transform representation \eqref{eq:nu-mellin-trans}, in which we make the substitution $s=\frac32+it$ to bring it to the form
\begin{equation}
\label{eq:nu-mellintrans-alternative}
\int_0^\infty \nu(x) x^{-\frac14 + \frac{it}{2}}\,dx
= \frac{2}{\sqrt{\pi}} \Gamma\left(\frac34+\frac{it}{2}\right) \Gamma\left(\frac34-\frac{it}{2}\right) \Xi(t).
\end{equation}
Note however that $\nu(x)$ can be expressed in terms of $\tilde{\nu}(t)$ as
\begin{equation*}
\nu(x) = \frac{2\sqrt{2}}{(x+1)^{3/2}} \tilde{\nu}\left(\frac{1-x}{1+x}\right)
= \frac{2\sqrt{2}}{(x+1)^{3/2}} \tilde{\nu}\left(\frac{x-1}{x+1}\right)
\end{equation*}
by inverting the defining relation \eqref{eq:balanced-centered-def} for centered functions and using the fact that $\tilde{\nu}(t)$ is an even function. This implies, using \eqref{eq:nutilde-chebyshev-exp}, that $\nu(x)$ has the series expansion
\begin{align*}
\nu(x) = \frac{1}{(x+1)^{3/2}}  \sum_{n=0}^\infty d_{2n} U_{2n}\left(\frac{x-1}{x+1}\right).
\end{align*}
We can now use this together with \eqref{eq:gn-mellintrans-tcoord} to evaluate the Mellin transform on the left-hand side of \eqref{eq:nu-mellintrans-alternative} in a different way as
\begin{align*}
\int_0^\infty \nu(x) x^{-\frac14 + \frac{it}{2}}\,dx
& =
\int_0^\infty \frac{1}{(x+1)^{3/2}}
\sum_{n=0}^\infty d_{2n} 
U_{2n}\left(\frac{x-1}{x+1}\right)x^{-\frac14+\frac{it}{2}}\,dx
\\& =
\sum_{n=0}^\infty d_{2n} \int_0^\infty \frac{1}{(x+1)^{3/2}}U_{2n}\left(\frac{x-1}{x+1}\right)x^{-\frac14+\frac{it}{2}}\,dx
\\& = 
\sum_{n=0}^\infty d_{2n} (-1)^n \frac{2}{\sqrt{\pi}} 
\Gamma\left(\frac34+\frac{it}{2}\right) \Gamma\left(\frac34-\frac{it}{2}\right) g_n\left(\frac{t}{2}\right).
\end{align*}
Equating this last expression to the right-hand side of \eqref{eq:nu-mellintrans-alternative} and canceling common terms recovers the $g_n$-expansion \eqref{eq:gn-expansion}, as we predicted.

\chapter{Additional results}

\label{ch:misc}

In the previous chapters we developed the main parts of the theory associated with the expansions of the Riemann xi function in the Hermite, $(f_n)_{n=0}^\infty$ and $(g_n)_{n=0}^\infty$ polynomial families. In this chapter we include a few additional results that continue to shed light on the themes we explored.

\section{An asymptotic formula for the Taylor coefficients of $\Xi(t)$}

\label{sec:xi-taylor-asym}

The method we used in Chapter~\ref{ch:hermite} to analyze the asymptotic behavior of the Hermite expansion coefficients $b_{2n}$ has the added benefit of enabling us to also prove an analogous asymptotic formula for the Taylor coefficients $a_{2n}$ in the Taylor expansion \eqref{eq:riemannxi-taylor} of the Riemann xi function. The reason for this is a pleasing similarity between the formulas for $a_{2n}$ and $b_{2n}$. It was noted by the authors of \cite{coffey} and \cite{csordas-norfolk-varga} (and probably others before them) that the formula for $a_{2n}$ can be written in the form
\begin{equation}
\label{eq:a2n-qn-rnprime}
a_{2n} = 
\frac{2}{(2n)!} \int_0^\infty \Phi(x)x^{2n}\,dx
= 
\frac{1}{(2n)!} \int_{-\infty}^\infty \Phi(x)x^{2n}\,dx,
\end{equation}
as can be seen by performing the usual change of variables $x= e^{2u}$ in \eqref{eq:riemannxi-taylorcoeff-int} (or by differentiating $2n$ times under the integral sign in \eqref{eq:riemannxi-fouriertrans} and setting $t=0$). The striking resemblence of this formula to \eqref{eq:turan-coeff-formula} seems however to have gone unremarked in the literature.

\begin{thm}[Asymptotic formula for the coefficients $a_{2n}$]
\label{thm:riemannxi-taylorcoeff-asym}
The coefficients $a_{2n}$ satisfy the asymptotic formula
\begin{align}
\label{eq:taylorxi-coeff-asym}
a_{2n} & =
\left(1+O\left(\frac{\loglog n}{\log n}\right)\right)
\frac{\pi^{1/4}}{2^{2n-\frac52} (2n)!} \left(\frac{2n}{\log (2n)}\right)^{7/4}
\exp\left[
2n\left(\log \left(\frac{2n}{\pi}\right) - W\left(\frac{2n}{\pi}\right) - \frac{1}{W\left(\frac{2n}{\pi}\right)} \right)
\right]
\end{align}
as $n\to\infty$, where $W(\cdot)$ denotes as in Chapter~\ref{ch:hermite} the Lambert $W$ function.
\end{thm}

\begin{proof}
The idea is to repeat the analysis in the proof of Theorem~\ref{thm:hermite-coeff-asym}, but with the numbers $Q_{2n}$ and $r_{2n}$ in \eqref{eq:qn-int-def}--\eqref{eq:rn-int-def} being replaced by
\begin{align}
\label{eq:qn-int-def-mod}
Q_n' & =
\int_0^\infty x^{2n} e^{\frac{5x}{2}}\left(e^{2x}-\frac{3}{2\pi}\right) \exp\left(-\pi e^{2x}\right)\,dx, \\
\label{eq:rn-int-def-mod}
r_n' & =
\int_0^\infty x^{2n} 
e^{\frac{5x}{2}} \sum_{m=2}^\infty \left(m^4 e^{2x}-\frac{3m^2}{2\pi}\right) \exp\left(-\pi m^2 e^{2x}\right)\,dx,
\end{align}
for which, by~\eqref{eq:phix-def} and~\eqref{eq:a2n-qn-rnprime}, we then have that \begin{equation}
\label{eq:a2n-q2n-r2n-relation}
a_{2n} = \frac{8\pi^2}{(2n)!}(Q_{2n}' + r_{2n}').
\end{equation}
Note that the only difference from the original definitions of $Q_{2n}$ and $r_{2n}$ is the absence of the factor $e^{-x^2/4}$. Thus, the analysis carries over essentially verbatim to our current case, except that we replace the function $f(x)$ in the reformulated equation \eqref{eq:qnint-rewritten} for $Q_n$ with
\begin{equation*}
\varphi(x) = e^{\frac{5x}{2}}\left(e^{2x}-\frac{3}{2\pi}\right),
\end{equation*}
to get the analogous representation
\begin{equation}
\label{eq:qnint-rewritten-analogue}
Q_n' = \frac{1}{2^{2n}}\int_0^\infty \varphi(x) \exp\left( \psi_{2n}(2x)\right)\,dx
\end{equation}
for $Q_n'$.
The effect of this change on the subsequent formulas is that the factor $\gamma_n$ in \eqref{eq:gamman-hermite-def} then also gets replaced by the simpler factor
$\gamma_n' = \varphi(x_{2n}/2)$ 
in the asymptotic formula
\begin{equation}
\label{eq:qn-asym-preformula-analogue}
Q_n' =
\left(1+O\left(\frac{1}{n^{1/5}}\right)\right)
\frac{\sqrt{\pi}}{2^{2n}\sqrt{2\beta_n}} \gamma_n' e^{\alpha_n}.
\end{equation}
that is the analogue of \eqref{eq:qn-asym-preformula}---the factors $\beta_n$ and $\alpha_n$ (and, importantly, the maximum point value $x_{2n}$ from which they are derived) remain the same.

Now, $\gamma_n'$ has the asymptotic behavior (the counterpart to \eqref{eq:gamman-asym})
\begin{equation}
\label{eq:gamman-prime-asym}
\gamma_n' = 
\left(1+ O\left(
e^{-x_{2n}}\right)\right)\exp\left(\frac94 x_{2n} \right) 
=
\left(1+ O\left( \frac{\log n}{n}\right)\right)
\left( \frac{2n}{\pi x_{2n}} \right)^{9/4}
\end{equation}
as $n\to\infty$. With these facts in mind, 
it is now a simple matter to go through the calculations and various bounds in the proof of Theorem~\ref{thm:hermite-coeff-asym} and verify that they remain valid in the current setting (including the bound \eqref{eq:rn-asym-bound} with $Q_n'$ and $r_n'$ replacing $Q_n$ and $r_n$, respectively), with the final result being that the relation \eqref{eq:qn-asym-formula} is now replaced by
\begin{align*}
Q_n' &= 
\left(1+O\left(\frac{\loglog n}{\log n}\right)\right)
\frac{1}{2^{2n+\frac12}} \left(\frac{2n}{\pi x_{2n}}\right)^{7/4}
\exp\left[
2n\left(\log (2n) - \log \pi - x_{2n} - \frac{1}{x_{2n}} \right)
\right].
\end{align*}
Inserting this into \eqref{eq:a2n-q2n-r2n-relation} gives \eqref{eq:taylorxi-coeff-asym}.
\end{proof}

It is interesting to compare our formula \eqref{eq:taylorxi-coeff-asym} to other asymptotic formulas for the coefficients $a_{2n}$ which have appeared in the literature. At the time we completed the first version of this paper, the strongest result of this type we were aware of was the one due to Coffey \cite[Prop.~1]{coffey}. Coffey's formula is more explicit, since it contains only elementary functions, but is less accurate, since (if expressed in our notation as a formula for $a_{2n}$ rather than in Coffey's logarithmic notation) it has a multiplicative error term of $\exp(O(1)) = \Theta(1)$, compared to our $1+O\left(\frac{\log n}{n}\right)$.

After we finished the initial version of this paper, we learned of another recent asymptotic formula for the coefficients $a_{2n}$ that was proved by Griffin, Ono, Rolen and Zagier in a 2018 paper \cite[Th.~7]{griffin-etal} (see also equations (1) and (13) in their paper). Griffin et al's result is more accurate than our Theorem~\ref{eq:taylorxi-coeff-asym}, as it gives a full asymptotic expansion for $a_{2n}$ whereby the relative error term can be made smaller than $o(n^{-K})$ for any fixed $K$ by truncating the expansion after sufficiently many terms. Their formula is expressed in terms of an implicitly-defined quantity $L(n)$ that solves the equation
$$ n = L\left(\pi e^L + \frac34 \right). $$
This equation (a slightly more exotic variant of our equation for $x_{2n}$ involving Lambert's $W$-function) arises out of an an application of Laplace's method in a manner quite similar to
our own analysis. It is interesting to ask whether our approach can be similarly extended to obtain a full asymptotic expansion for $a_{2n}$ that is expressed in terms of the (arguably simpler) quantities $x_{2n} = W\left(\frac{2n}{\pi}\right)$.

\section{The function $\tilde{\omega}(t)$}

\label{sec:omega-tilde-properties}

In \secref{sec:rad-centered-recbal} we defined the centered version of a balanced function, and applied that concept to the study of the function $\nu(x)$ and its centered version $\tilde{\nu}(t)$, which has turned out to be quite significant in the developments of Chapters~\ref{ch:radial} and~\ref{ch:gn-expansion}. We now consider the function $\tilde{\omega}(t)$, the centered version of $\omega(x)$, which is not only a more fundamental object than $\tilde{\nu}(t)$ (in the sense that the latter is computed from the former), but turns out to also be significant and interesting in several distinct (and seemingly unrelated) ways.

As an initial and rather trivial observation, we already noted in Proposition~\ref{prop:cn-as-moments} that the coefficients $c_n$ can be interpreted as moments of $\tilde{\omega}(t)$, except for a trivial scaling factor.

The next observation, which is also trivial as it is essentially a restatement of the above result, is that the Fourier transform of $\tilde{\omega}(t)$ can be interpreted as a generating function for the coefficient sequence $c_n$ (that is different from $\tilde{\nu}(t)$, which in Theorem~\ref{eq:nutilde-taylor} we also interpreted as a generating function for the $c_n$'s). Namely, we have the relation
\begin{equation}
\label{eq:capital-omega-def}
\int_{-1}^1 \tilde{\omega}(u)e^{itu}\,du = \frac12 \sum_{n=0}^\infty \frac{i^n c_n t^n}{n!}.
\end{equation}

Next, we arrive at a somewhat more surprising fact, which is that $\tilde{\omega}(u)$ also arises in a different way as a scaling limit of the Fourier spectrum of the Poisson flow associated with the $f_n$-expansion. To make this precise, recall that in Theorem~\ref{thm:fn-poisson-mellin} we derived a Mellin transform representation for the Poisson flow $X_r^{\mathcal{F}}(t)$. We will consider a limit of this representation as $r\to 0$, but scale the $t$ variable by a factor of $r$ since, as the formula \eqref{eq:fn-omega-compressed} shows, the Mellin spectrum without scaling gets compressed into the interval $[(1-r)/(1+r),(1+r)/(1-r)]$, which shrinks to a point as $r\to 0$. We also rewrite the Mellin transform as an ordinary ``additive'' Fourier transform, in other words expressing the rescaled Poisson flow as
\begin{equation*}
X_r^{\mathcal{F}}\left(\frac{t}{r}\right)
= \int_{-\infty}^\infty \Psi_r(v) e^{ivt}\,dv.
\end{equation*}
This representation is obtained from \eqref{eq:fn-poisson-mellin} by a standard exponential change of variables ($x=e^{2r v}$ in the particular scaling we use), and it is straightforward to check that $\Psi_r(v)$ is given by
\begin{align*}
\Psi_r(v) &=
2 r e^{rv/2} \omega_r\left(e^{2rv}\right)
=
\begin{cases}
2 r \frac{1+\eta}{\sqrt{1-\eta}} 
\frac{1}{\sqrt{1-\eta e^{2rv}}} e^{rv/2} 
\omega\left( \frac{e^{2rv}-\eta}{1-\eta e^{2rv}}
\right)
& \textrm{if }|v| < \frac{1}{2r}\log\left(\frac{1}{\eta}\right),
\\
0 & \textrm{otherwise}
\end{cases}
\end{align*}
(refer to \eqref{eq:fn-omega-compressed} for the second equality, and recall the notation \eqref{eq:compression-eta-notation}).

\begin{prop}[The centered function $\tilde{\omega}(u)$ as a scaling limit of the Poisson flow frequency spectrum]
We have the pointwise limits
\begin{equation*}
\lim_{r\to 0+} \Psi_r(v) =
\begin{cases}
2\tilde{\omega}(v) & \textrm{if }|v|<1,
\\
0 & \textrm{otherwise}.
\end{cases}
\end{equation*}
\end{prop}

\begin{proof}
This is a somewhat mundane verification involving Taylor series approximations. Specifically, one finds that, as $r\to 0$, 
\begin{align*}
\frac{1}{2r}\log \left(\frac{1+r}{1-r} \right) &
= 1 + O(r^2), \\
2 r \frac{1+\eta}{\sqrt{1-\eta}} 
\frac{1}{\sqrt{1-\eta e^{2rv}}} e^{rv/2} 
& = \frac{2}{\sqrt{1-v}} + O(r^2),
\\
\frac{e^{2rv}-\eta}{1-\eta e^{2rv}}
& = \frac{1+v}{1-v} + O(r^2).
\end{align*}
The first of these three limits substantiates the claim that the Fourier spectrum $\Psi_r(v)$ is supported in the limit on the interval $(-1,1)$; the second and third limits show that for $v\in (-1,1)$ we have
\begin{equation*}
\lim_{r\to 0+} \Psi_r(v) =
\frac{2}{\sqrt{1-v}} \omega\left(\frac{1+r}{1-r}\right)
= 2 \tilde{\omega}(-v)
= 2 \tilde{\omega}(v)
\end{equation*}
(since $\tilde{\omega}(v)$ is an even function), as claimed.
\end{proof}

Our final result on $\tilde{\omega}(u)$ will show that not just its Fourier transform, but also $\tilde{\omega}(u)$ itself, is a generating function for an interesting sequence, which can be given explicitly in terms of a recently studied sequence of integers. For this, we first recall our recent results \cite{romik} on the Taylor expansion of the Jacobi theta series $\theta(x)$ (defined in \eqref{eq:jactheta-def}) and its centered version, which as usual is related to $\theta(x)$ by
\begin{equation}
\label{eq:theta-thetatilde-relation}
\tilde{\theta}(u) = \frac{1}{\sqrt{1+u}}
\theta\left( \frac{1-u}{1+u} \right) \qquad (|u|<1).
\end{equation}
In \cite{romik} (where $\theta(x)$ was denoted $\theta_3(x)$ and $\tilde{\theta}(u)$ was denoted $\sigma_3(u)$) we proved that $\tilde{\theta}(u)$ has the Taylor expansion
\begin{equation}
\label{eq:tilde-theta-taylorexp}
\tilde{\theta}(u) = W \sum_{n=0}^\infty \frac{\delta(n)}{(2n)!} \Phi^n u^{2n} \qquad (|u|<1),
\end{equation}
where $W$ and $\Phi$ are two special constants, given by
\begin{equation*}
W = \theta(1) = 
\frac{\Gamma\left(\frac14\right)}{\sqrt{2}\pi^{3/4}},
\qquad
\Omega = 
\frac{\Gamma\left(\frac14\right)^8}{128 \pi^4},
\end{equation*}
respectively,
and where the sequence
$(\delta(n))_{n=0}^\infty = 1,1,-1,51,849,-26199, 
1341999, \ldots$
(denoted as $d(n)$ in \cite{romik}---for the current discussion we changed the notation in order to avoid a potential confusion with the coefficient sequence $d_n$ in the expansion \eqref{eq:gn-expansion}) is a sequence of integers first introduced and studied in \cite{romik} (see also \cite{jac3-oeis}).

With this preparation, we can formulate a result identifying the coefficients in the Taylor expansion of $\tilde{\omega}(u)$.

\begin{thm}[Taylor expansion of $\tilde{\omega}(u)$]
\label{thm:omega-tilde-taylor}
The Taylor expansion of $\tilde{\omega}(u)$ is given by
\begin{equation*}
\tilde{\omega}(u)
= W \sum_{n=0}^\infty \frac{\rho(n)}{(2n)!} \Omega^n u^{2n} \qquad (|u|<1),
\end{equation*}
where $\rho(n)$ are numbers defined in terms of the sequence $(\delta(n))_{n=0}^\infty$ as
\begin{align*}
\rho(n) &=
\frac{1}{16} \Big(
4 \Omega \delta(n+1) - (32n^2 + 8n+3) \delta(n)
+ (2n-1)(2n)(4n-1)(4n-3) \Omega^{-1} \delta(n-1)
\Big).
\end{align*}
\end{thm}

\begin{proof}
The idea is to first of all find a way to express $\tilde{\omega}(u)$ in terms of $\tilde{\theta}(u)$, and then use~\eqref{eq:tilde-theta-taylorexp}.
Start with the relation 
\begin{equation*}
\theta(x) = 
\frac{\sqrt{2}}{\sqrt{1+x}} \tilde{\theta}\left(\frac{1-x}{1+x}\right)
\end{equation*}
that is inverse to \eqref{eq:theta-thetatilde-relation}. Differentiating twice, we get
\begin{align*}
\theta'(x) &= 
-\frac12 \sqrt{2} \frac{1}{(1+x)^{3/2}} 
\tilde{\theta}\left(\frac{1-x}{1+x}\right)
+
\frac{-2\sqrt{2}}{(1+x)^{5/2}}
\tilde{\theta}'\left(\frac{1-x}{1+x}\right),
\\
\theta''(x) &= 
\sqrt{2}
\Bigg(
\frac{3}{4(1+x)^{5/2}}
\tilde{\theta}\left(\frac{1-x}{1+x}\right)
+ \frac{6}{(1+x)^{7/2}}
\tilde{\theta}'\left(\frac{1-x}{1+x}\right)
+\frac{4}{(1+x)^{9/2}}
\tilde{\theta}''\left(\frac{1-x}{1+x}\right)
\Bigg).
\end{align*}
It then follows that
\begin{align*}
\tilde{\omega}(u)
& =
\frac{1}{\sqrt{1+u}} \omega\left(\frac{1-u}{1+u} \right)
=
\frac{1}{2\sqrt{1+u}}
\Bigg(
2\left(\frac{1-u}{1+u}\right)^2
\theta''\left(\frac{1-u}{1+u}\right)
+3\left(\frac{1-u}{1+u}\right)
\theta'\left(\frac{1-u}{1+u}\right)
\Bigg)
\\ &=
\frac{1}{\sqrt{1+u}}
\Bigg[
  \sqrt{2} \left(\frac{1-u}{1+u}\right)^2
\Bigg(
\frac34 \frac{(1+u)^{5/2}}{2^{5/2}}
\tilde{\theta}(u)
+ \frac{6(1+u)^{7/2}}{2^{7/2}}
\tilde{\theta}'(u)
+\frac{4(1+u)^{9/2}}{2^{9/2}}
\tilde{\theta}''(u)
\Bigg)
\\ & \qquad\qquad\qquad\qquad\qquad\qquad\qquad +
\frac{3\sqrt{2}}{2} \left(\frac{1-u}{1+u}\right)
\Bigg(
-\frac12 \frac{(1+u)^{3/2}}{2^{3/2}} 
\tilde{\theta}(u)
-2 \frac{(1+u)^{5/2}}{2^{5/2}} \tilde{\theta}'(u)
\Bigg)
\Bigg].
\end{align*}
From here, a trivial algebraic simplification, which we omit, leads to the identity
\begin{equation}
\label{eq:omegatilde-thetatilde-exp}
\tilde{\omega}(u) =
\frac{1}{16}
\left[
3(u^2-1) \tilde{\theta}(u)
+ 12 u (u^2-1) \tilde{\theta}'(u)
+ 4 (u^2-1)^2 \tilde{\theta}''(u)
\right].
\end{equation}
But now observe that from \eqref{eq:tilde-theta-taylorexp} we have
\begin{align*}
\tilde{\theta}'(u) &= W \sum_{n=1}^\infty \frac{\delta(n)}{(2n-1)!} \Omega^n u^{2n-1},
\qquad\quad 
\tilde{\theta}''(u) 
= W \sum_{n=2}^\infty \frac{\delta(n)}{(2n-2)!} \Omega^n u^{2n-2}.
\end{align*}
Inserting \eqref{eq:tilde-theta-taylorexp} and these last two expansions into \eqref{eq:omegatilde-thetatilde-exp} and simplifying gives the claim.
\end{proof}

\chapter{Final remarks}

\label{ch:summary}

This work has seen the introduction of a curious menagerie of previously unnoticed (or, at the very least, under-appreciated) special functions that are tied in an interesting way to the theory of the Riemann zeta function. This collection includes the orthogonal polynomial families $U_n, H_n, L_n^{1/2}, f_n$, and $g_n$; the elementary (though esoteric) radial functions $A(r)$, $B(r)$; the well-known functions $\theta(x)$ and $\omega(x)$, originating in the world of modular forms and theta series; and the function $\nu(x)$ and its centered version $\tilde{\nu}(t)$, which do not seem to have been previously studied. 

These functions and their many subtle interconnections add a new set of tools to the arsenal of methods available to attack central open problems in the theory of the zeta function, the Riemann hypothesis foremost among them. Most significantly, one is left with the impression that the theory of orthogonal polynomials may have a more central role to play in the study of the zeta function, and perhaps a greater potential to lead to new insights, than had been previously suspected.

We conclude with a few open problems and suggestions for future research.

\begin{enumerate}

\item There has been much discussion in the literature of sufficient conditions guaranteeing that a polynomial $p(z)$ has only real zeros based on knowledge of its coefficients in the expansion $p(z)=\sum_{k=0}^n \alpha_k \phi_k(z)$, where $(\phi_k)_{k=0}^\infty$ is some given family of orthogonal polynomials. We note Tur\'an's many results in \cite{turan1952, turan1954, turan1959}, particularly his observation (Lem.~II in \cite{turan1959}, a result he discovered independently but attributes to an earlier paper by P\'olya \cite{polya1915}) that if the zeros of $\sum_{k=0}^n a_k z^k$ are all real then that is also the case for the corresponding Hermite expansion $\sum_{k=0}^n a_k H_k(z)$; and the many analogous theorems of Iserles and Saff \cite{iserles-saff}, among them the result (a special case of Prop.~6 in their paper)
that if the zeros of the polynomial $\sum_{k=0}^n a_k z^k$
are all real then that is also the case for the polynomial $\sum_{k=0}^n \frac{k!}{(3/2)_k} a_k f_k(z)$.
See also \cite{bates, bleecker-csordas, iserles-norsett1987, iserles-norsett1988, iserles-norsett1990, piotrowski} and the survey \cite{schmeisser} for further developments along these lines.

One question that now arises naturally is: to what extent do these developments inform the attempts to prove the reality of the zeros of the Riemann xi function, in view of our new results?

\item One rather striking fact is that the four different series expansions we have considered for the Riemann xi function, namely
\begin{equation*}
\begin{array}{rclcrcl}
\Xi(t) & \!\!\!=\!\!\!& \displaystyle \sum_{n=0}^\infty (-1)^n a_{2n} t^{2n},
& &
\Xi(t) & \!\!\!=\!\!\!& \displaystyle \sum_{n=0}^\infty (-1)^n c_{2n} f_{2n}(t),
\\[14pt]
\Xi(t) & \!\!\!=\!\!\!& \displaystyle \sum_{n=0}^\infty (-1)^n b_{2n} H_{2n}(t),
& &
\Xi(t) & \!\!\!=\!\!\!& \displaystyle \sum_{n=0}^\infty (-1)^n d_{2n} g_{2n}(t),
\end{array}
\end{equation*}
exhibit remarkably similar structural similarities: namely, in all four expansions the coefficients appear with alternating signs (and their asymptotics can be understood to a good level of accuracy, as our analysis shows).

It is intriguing to wonder about the significance of this structural property of $\Xi(t)$. Can this information be exploited somehow to derive information about the location of the zeros of $\Xi(t)$?

By way of comparison, one can consider ``toy'' expansions of the above forms involving more elementary coefficient sequences. For example, we have the trivial expansions
(the latter two of which being easy consequences of \eqref{eq:hermite-genfun} and \eqref{eq:fn-genfun}, respectively)
\begin{align*}
\qquad\qquad 
\sum_{n=0}^\infty (-1)^n \frac{\alpha^{2n}}{(2n)!} t^{2n}
\negphantom{t^{2n}}\phantom{H_{2n}(t)}
&= \cos(\alpha t)
& (\alpha>0),
\\
\sum_{n=0}^\infty (-1)^n \frac{\alpha^{2n}}{(2n)!} H_{2n}(t) &= e^{\alpha^2} \cos(2\alpha t)
& (\alpha>0),
\\
\sum_{n=0}^\infty (-1)^n \alpha^{2n} 
f_{2n}(t)\negphantom{f_{2n}(t)}\phantom{H_{2n}(t)}
&= 
\frac{2}{(1-\alpha^2)^{3/4}}
\cos\left(t \log\left(\frac{1+\alpha}{1-\alpha} \right) \right)
& (0<\alpha<1),
\end{align*}
which are entire functions of $t$ that---needless to say---all have only real zeros. On the other hand, we do not know for which values of $\alpha\in (0,1)$ the expansion (whose explicit form is evaluated using \eqref{eq:gn-genfun})
\begin{align*}
\qquad \ \ \ \ \sum_{n=0}^\infty (-1)^n \alpha^{2n}g_{2n}(t) 
&= 
\frac{1}{(1-\alpha)^2} \,\gausshyper\left(1,\frac34-it;\frac32;\frac{-4\alpha}{(1-\alpha)^2}\right)
\\ & \qquad\ \  +
\frac{1}{(1+\alpha)^2} \,\gausshyper\left(1,\frac34-it;\frac32;\frac{4\alpha}{(1+\alpha)^2}\right)
\end{align*}
has only real zeros.

\item
The notion of Poisson flows we introduced seems worth exploring further. The Poisson flow associated with the polynomial family $(f_n)_{n=0}^\infty$ has interesting properties, and while it does not preserve hyperbolicity in the sense of ``continuous time'' as we discussed in \secref{sec:zeros-evolution},
it seems not inconceivable that a weaker form of preservation of reality of the zeros for discrete time parameter values might still hold. For example, does there exist a constant $0<r_0<1$ such that if the polynomial $\sum_{k=0}^n a_k r_0^k f_k(t)$ has only real zeros then the same is guaranteed to be true for the polynomial 
$\sum_{k=0}^n a_k f_k(t)$? It appears like it may be possible to approach this question using the biorthogonality techniques developed in the papers by Iserles and coauthors~\cite{iserles-norsett1987, iserles-norsett1988, iserles-norsett1990, iserles-saff}.
And what can be said about the Poisson flow associated with the orthogonal polynomial family $(g_n)_{n=0}^\infty$?

\item Does the function in \eqref{eq:capital-omega-def}, the Fourier transform of $\tilde{\omega}(u)$ (which as we have seen can be thought of as a scaling limit of the Poisson flow), have only real zeros? Is this question related to the Riemann hypothesis?

\end{enumerate}

\appendix

\chapter{Orthogonal polynomials}

\label{appendix:orthogonal}

In this appendix we summarize some background facts we will need on several families of orthogonal polynomials, and prove a few additional auxiliary results. We assume the reader is familiar with the basic theory of orthogonal polynomials, as described, e.g., in Chapter 2--3 of Szeg\H{o}'s classical book \cite{szego}.

\section{Chebyshev polynomials of the second kind}

\label{sec:orth-chebyshev}

The Chebyshev polynomials, denoted $U_n(x)$, are a sequence of orthogonal polynomials with respect to the weight function $\sqrt{1-x^2}$ on $(-1,1)$, and are one of the most classical families of orthogonal polynomials. A few of their main properties are given below; see \cite[pp.~225--229]{koekoek-etal} for more details.

\begin{enumerate}
\item Definition:
\begin{align}
\label{eq:chebyshev-explicit1}
U_n(x) &= \frac{\sin((n+1)\arccos(x))}{\sin(\arccos(x))}
\\ 
\label{eq:chebyshev-explicit2}
& =
\sum_{k=0}^{\lfloor n/2\rfloor} (-1)^k \binom{n-k}{k} (2x)^{n-2k}
\end{align}

\item Inverse relationship with monomial basis:
\begin{equation}
\label{eq:chebyshev-monomial-expansion}
x^n = \frac{1}{(n+1)2^n} \sum_{k=0}^{\lfloor n/2 \rfloor}
(n-2k+1)\binom{n+1}{k} U_{n-2k}(x).
\end{equation}

\item Orthogonality relation:
\begin{equation}
\label{eq:chebyshev-orthogonality}
\int_0^1 U_m(x) U_n(x) \sqrt{1-x^2}\,dx
=
\frac{\pi}{2} \delta_{m,n}
\end{equation}

\item Recurrence relation:
\begin{equation}
U_{n+1}(x)-2x U_n(x) + U_{n-1}(x) 
= 0
\end{equation}

\item Differential equation:
\begin{equation}
(1-x^2)U_n''(x) - 3x U_n'(x) + n(n+2) U_n(x) = 0
\end{equation}

\item Generating function:
\begin{equation}
\sum_{n=0}^\infty U_n(x) z^n =
\frac{1}{1-2xz+z^2}
\end{equation}

\item Poisson kernel:
\begin{equation}
\label{eq:chebyshev-poissonker}
\frac{2}{\pi} \sum_{n=0}^\infty U_n(x)U_n(y) z^n 
= \frac{2}{\pi}\cdot \frac{1-z^2}{1-4xyz(z^2+1)+2(2x^2+2y^2-1)z^2+z^4}
\end{equation}

\item 
Symmetry: $U_n(-x) = (-1)^n U_n(x)$.

\end{enumerate}

\textbf{Notes.} To derive \eqref{eq:chebyshev-monomial-expansion}, take derivatives of both sides of the relation (found in \cite[p.~22]{mason-handscomb}) $\cos^{n+1} \theta = \sum_{k=0}^{\lfloor n/2\rfloor } \binom{n}{k}\cos(n-2k)\theta - \frac12 \chi_{\{n \textrm{ even}\}} \binom{n}{n/2}$, and use \eqref{eq:chebyshev-explicit1}. Formula~\eqref{eq:chebyshev-poissonker} is derived in \cite{millar}, where it appears as equation (15), except that the formula there contains a typo (the term $-4axy$ in the denominator needs to be changed to $-2axy$), which we corrected.

\section{Hermite polynomials}

\label{sec:orth-hermite}

The Hermite polynomials are the well-known sequence $H_n(x)$ of polynomials that are orthogonal with respect to the Gaussian weight function $e^{-x^2}$ on $\R$. A few of their main properties are given below; see \cite[Sec.~6.1]{andrews-askey-roy} for proofs.

\begin{enumerate}
\item Definition:
\begin{equation}
H_n(x) = (-1)^n e^{x^2} \frac{d^n}{dx^n}\left( e^{-x^2}\right)
\end{equation}

\item Orthogonality relation:
\begin{equation}
\label{eq:hermite-orthogonality}
\int_{-\infty}^\infty H_m(x) H_n(x) e^{-x^2}\,dx
=
2^n n! \sqrt{\pi} \delta_{m,n}
\end{equation}

\item Recurrence relation:
\begin{equation}
\label{eq:hermite-recurrence}
H_{n+1}(x)-2x H_n(x) + 2n H_{n-1}(x) 
= 0
\end{equation}

\item Differential equation:
\begin{equation}
\label{eq:hermite-ode}
H_n''(x) - 2x H_n'(x) + 2n H_n(x) = 0
\end{equation}

\item Generating function:
\begin{equation}
\label{eq:hermite-genfun}
\sum_{n=0}^\infty \frac{H_n(x)}{n!} z^n =
\exp(2xz-z^2)
\end{equation}

\item Poisson kernel:
\begin{equation}
\frac{1}{\sqrt{\pi}} \sum_{n=0}^\infty \frac{H_n(x)H_n(y)}{2^n n!} z^n 
= \frac{1}{\sqrt{\pi}\sqrt{1-z^2}} \exp\left( \frac{2xyz-(x^2+y^2)z^2}{1-z^2}\right)
\end{equation}

\item 
Symmetry: $H_n(-x) = (-1)^n H_n(x)$.

\end{enumerate}

\section{Laguerre polynomials}

\label{sec:orth-laguerre}

The (generalized) Laguerre polynomials form a sequence $L_n^\alpha(x)$ of polynomials, dependent on a parameter $\alpha> -1$, that are orthogonal with respect to the weight function $e^{-x}x^\alpha$ on $(0,\infty)$. Of particular interest to us will be the case $\alpha=1/2$; in this case the polynomials are essentially rescalings of the odd-indexed Hermite polynomials. For convenience, we summarize below a few of the main formulas associated with the polymomials $L_n^{\alpha}(x)$; proofs can be found in \cite[Sec~6.2]{andrews-askey-roy}.

\begin{enumerate}
\item Definition:
\begin{equation}
\label{eq:laguerre-explicit}
L_n^\alpha(x) = \sum_{k=0}^n \frac{(-1)^k}{k!} \binom{n+\alpha}{n-k} x^k.
\end{equation}

\item Orthogonality relation:
\begin{equation}
\label{eq:laguerre-orthogonality}
\int_{-\infty}^\infty L_m(x) L_n(x) e^{-x}x^\alpha\,dx = \frac{\Gamma(n+\alpha+1)}{n!} \delta_{m,n}
\end{equation}

\item Recurrence relation:
\begin{equation}
\label{eq:laguerre-recurrence}
(n+1)L_{n+1}^\alpha(x)+(x-2n-\alpha-1)L_n^\alpha(x)+(n+\alpha)L_{n-1}^\alpha(x) = 0
\end{equation}

\item Differential equation:
\begin{equation}
\label{eq:laguerre-diffeq}
x (L_n^\alpha)''(x)+(\alpha+1-x)(L_n^\alpha)'(x)+n L_n^\alpha(x) = 0
\end{equation}

\item Generating function:
\begin{equation}
\label{eq:laguerre-genfun}
\sum_{n=0}^\infty L_n^\alpha(x) z^n = \frac{1}{(1-z)^{\alpha+1}} \exp\left(-\frac{xz}{1-z}\right)
\end{equation}

\item Poisson kernel:
\begin{equation}
\sum_{n=0}^\infty \frac{n!}{\Gamma(n+\alpha+1)} L_n^\alpha(x) L_n^\alpha(y) z^n =
\frac{1}{1-z} \exp\left(-\frac{z(x+y)}{1-z}\right) (xyz)^{-\alpha/2} I_\alpha\left(\frac{2\sqrt{xyz}}{1-z}\right),
\end{equation}
where $I_\alpha(\cdot)$ denotes the modified Bessel function of the first kind.

\end{enumerate}

\section[The symmetric Meixner-Pollaczek polynomials]{The symmetric Meixner-Pollaczek polynomials $f_n(x)=P_n^{(3/4)}(x;\pi/2)$}

\label{sec:orth-fn}

The \textbf{Meixner-Pollaczek polynomials} are a two-parameter family of orthogonal polynomial sequences $P_n^{(\lambda)}(x;\phi)$. The parameters satisfy $\lambda>0, 0\le \phi<\pi$. In the special case $\phi=\pi/2$, the polynomials are sometimes referred to as the \textbf{symmetric Meixner-Pollaczek polynomials} (see \cite{araaya}). In this paper we make use of the special case $\lambda=3/4$ of the symmetric case, namely the polynomials, which we denote $f_n(x)$ for simplicity, given by
\begin{equation}
f_n(x) = P_n^{(3/4)}(x;\pi/2).
\end{equation}

The key property of the polynomials $f_n(x)$ is that they are an orthonormal family for the weight function $\left|\Gamma\left(\frac34+ix\right)\right|^2$, $x\in\R$. Additional properties we will need are given in the list below. Bibliographic notes and a few more details regarding proofs are given at the end of this section. See also \secref{sec:orth-fngn-relation} where we prove additional results relating the polynomial family $f_n$ to another family $g_n$ of orthogonal polynomials, discussed in \secref{sec:orth-gn}.

\begin{enumerate}
\item Definition and explicit formulas:
\begin{align}
f_n(x) & = 
\frac{(3/2)_n}{n!} 
i^n \gausshyper\left(-n, \frac34+ix; \frac32; 2 \right)
\label{eq:fn-explicit1}
\\ & =
(-i)^n \sum_{k=0}^n 2^k \binom{n+\frac12}{n-k} \binom{-\frac34+i x}{k}
\label{eq:fn-explicit2}
\\ & =
i^n \sum_{k=0}^n (-1)^k 2^k \binom{n+\frac12}{n-k} \binom{-\frac14+i x+k}{k}
\label{eq:fn-explicit3}
\\ & = i^n \sum_{k=0}^n (-1)^k \binom{-\frac34+ix}{k}
\binom{-\frac34-ix}{n-k}.
\label{eq:fn-explicit4}
\end{align}

\item Orthogonality relation:
\begin{equation}
\label{eq:fn-orthogonality}
\int_{-\infty}^\infty f_m(x) f_n(x) \left|\Gamma\left(\frac34+ix\right)\right|^2\,dx = \frac{\pi^{3/2}(3/2)_n}{2\sqrt{2} n!}\delta_{m,n}.
\end{equation}

\item Recurrence relation:
\begin{equation}
\label{eq:fn-recurrence}
(n+1)f_{n+1}(x)-2x f_n(x) + \left(n+\frac12\right) f_{n-1}(x) = 0.
\end{equation}

\item Difference equation:
\begin{equation}
\label{eq:fn-diffeq}
2\left(n+\frac34\right)f_n(x)-\left(\frac34-ix\right) f_n(x+i)
-\left(\frac34+ix\right) f_n(x-i) = 0.
\end{equation}

\item Generating function:
\begin{equation}
\label{eq:fn-genfun}
\sum_{n=0}^\infty f_n(x) z^n = (1-iz)^{-\frac34+ix}(1+iz)^{-\frac34-ix}
= \frac{1}{(1+z^2)^{\frac34}}\left(\frac{1-iz}{1+iz}\right)^{ix}.
\end{equation}

\item Poisson kernel:
\begin{align}
\label{eq:fn-poisson}
\frac{2\sqrt{2}}{\pi^{3/2}}\sum_{n=0}^\infty \frac{n!}{(3/2)_n} f_n(x) f_n(y) z^n
& = 
\frac{2\sqrt{2}}{\pi^{3/2}}
\frac{1}{(1-z)^{3/2}} 
\left(\frac{1+z}{1-z}\right)^{i(x+y)}
\gausshyper\left(\frac34+ix, \frac34+iy; \frac32; \frac{-4z}{(1-z)^2}
\right)
\end{align}

\item Mellin transform representations:
\begin{align}
\label{prop:fn-mellintrans1-appendix}
f_n(x) &= (-i)^n \frac{\sqrt{\pi} (3/2)_n}{2 n!}
\left( \Gamma\left(\frac34-ix\right)\Gamma\left(\frac34+ix\right) \right)^{-1} 
\int_0^\infty
\frac{1}{(u+1)^{3/2}} \left(\frac{u-1}{u+1}\right)^n
u^{-\frac14+ix}\,du
\\
\label{prop:fn-mellintrans2-appendix}
&=
2 i^n \pi^{\frac34+ix} \Gamma\left(\frac34+ix\right)^{-1}
\int_0^\infty e^{-\pi r^2} L_n^{1/2}(2\pi r^2) r^{\frac12+2ix} \,dr 
\end{align}

\item 
Symmetry: $f_n(-x) = (-1)^n f_n(x)$.

\end{enumerate}

\textbf{Notes.} The above list is based on the general list of properties of the Meixner-Pollaczek polynomials $P_n^{(\lambda)}(x;\phi)$ provided in \cite[pp.~213--216]{koekoek-etal}, except for \eqref{eq:fn-poisson}, which is a special case of \cite[Eq.~(2.25)]{ismail-stanton}, and the Mellin transform representations \eqref{prop:fn-mellintrans1-appendix}--\eqref{prop:fn-mellintrans2-appendix}, which are proved in our Propositions~\ref{prop:fn-mellintrans1} and~\ref{prop:fn-mellintrans3}.

In the formulas \eqref{eq:fn-explicit1}--\eqref{eq:fn-explicit4}, the first formula is the definition as given in \cite{koekoek-etal}; formula \eqref{eq:fn-explicit3} is an explicit rewriting of \eqref{eq:fn-explicit1} as a sum, and formula \eqref{eq:fn-explicit2} follows from \eqref{eq:fn-explicit3} by applying the symmetry property $f_n(-x) = (-1)^n f_n(x)$ (which in turn is an easy consequence of either the recurrence relation \eqref{eq:fn-recurrence} or the generating function \eqref{eq:fn-genfun}). Formula \eqref{eq:fn-explicit4} appears to be new, and follows by evaluating the sequence of coefficients of $z^n$ in the generating function \eqref{eq:fn-genfun} as a convolution of the coefficient sequences for the functions $(1-iz)^{-3/4+ix}$ and $(1+iz)^{-3/4-ix}$. Note that \eqref{eq:fn-explicit4} has the benefit of making the odd/even symmetry of $f_n(x)$ readily apparent, which the other explicit formulas do not.

\begin{table}
\caption{The first few polynomials $f_n(x)$}
\begin{equation*}
\begin{array}{r|l}
n & f_n(x) \\
\hline
 0 & 1 \\[3pt]
 1 & 2 x \\[3pt]
 2 & 2 x^2-\frac{3}{4} \\[3pt]
 3 & \frac{4 }{3}x^3-\frac{13 }{6} x \\[3pt]
 4 & \frac{2 }{3}x^4-\frac{17 }{6}x^2+\frac{21}{32} \\[3pt]
 5 & \frac{4 }{15}x^5-\frac{7 }{3}x^3+\frac{177 }{80}x \\[3pt]
 6 & \frac{4 }{45}x^6-\frac{25 }{18}x^4+\frac{2401 }{720}x^2-\frac{77}{128} \\[3pt]
 7 & \frac{8 }{315}x^7-\frac{29 }{45}x^5+\frac{1123 }{360}x^3-\frac{4987 }{2240}x \\[3pt]
 8 & \frac{2 }{315}x^8-\frac{11 }{45}x^6+\frac{1499 }{720}x^4-\frac{24749 }{6720}x^2+\frac{1155}{2048}
\end{array}
\end{equation*}
\end{table}

\section[The continuous Hahn polynomials]{The continuous Hahn polynomials $g_n(x) = p_n\left(x; \frac34,\frac34,\frac34,\frac34\right)$}

\label{sec:orth-gn}

The \textbf{continuous Hahn polynomials} are a four-parameter family $p_n(x;a,b,c,d)$ of orthogonal polynomial sequences. They were introduced in increasing degrees of generality by Askey and Wilson \cite{askey-wilson} and later Atakishiyev and Suslov \cite{atakishiyev-suslov} as continuous-weight analogues of the Hahn polynomials; earlier special cases appeared in the work of Bateman \cite{bateman} and later Pasternack \cite{pasternack} (see also \cite{koelink} for a chronology of these discoveries and related discussion).

For our purposes, a special role will be played by the special case $a=b=c=d=3/4$ of the continuous Hahn polynomials, that is, the polynomial sequence
\begin{equation}
g_n(x) = p_n\left(x; \frac34,\frac34,\frac34,\frac34\right).
\end{equation}
A few of the main properties of these polynomials we will need are listed below. The notes at the end of the section provide references and additional details.

\begin{enumerate}
\item Definition and explicit formulas:
\begin{align}
\label{eq:gn-explicit1}
g_n(x) & =
i^n (n+1) \ {}_3F_2\left(-n,n+2,\frac34+ix;\frac32,\frac32; 1\right)
\\ 
\label{eq:gn-explicit2}
&=
(-i)^n \sum_{k=0}^n \frac{(n+k+1)!}{(n-k)!(3/2)_k^2}
\binom{-\frac34+i x}{k}
\\ 
\label{eq:gn-explicit3}
& = i^n 
\sum_{k=0}^n (-1)^k 
\frac{(n+1)!}{(3/2)_k(3/2)_{n-k}}
\binom{-\frac34+ix}{k}
\binom{-\frac34-ix}{n-k}.
\end{align}

\item Orthogonality relation:
\begin{equation}
\label{eq:gn-orthogonality}
\int_{-\infty}^\infty g_m(x) g_n(x) \left|\Gamma\left(\frac34+ix)\right)\right|^4\,dx = \frac{\pi^3}{16} \delta_{m,n}.
\end{equation}

\item Recurrence relation:
\begin{equation}
\label{eq:gn-recurrence}
(2n+3) g_{n+1}(x) - 8x g_n(x) + (2n+1)g_{n-1}(x)
= 0.
\end{equation}

\item Difference equation:
\begin{equation}
\label{eq:gn-diffeq}
\left((n+1)^2-2x^2+\frac18\right)g_n(x) - \left(\frac34-ix\right)^2g_n(x+i) - \left(\frac34+ix\right)^2 g_n(x-i) = 0.
\end{equation}

\item Generating functions:
\begin{align}
\label{eq:gn-genfun}
\sum_{n=0}^\infty g_n(x) z^n &= \frac{1}{(1+iz)^2} \,\gausshyper\left(1,\frac34-ix;\frac32;\frac{4iz}{(1+iz)^2}\right)
\\
\label{eq:gn-genfun2}
\sum_{n=0}^\infty \frac{g_n(x)}{(n+1)!}z^n
& =
\onefone\left(\frac34+ix;\frac32;-iz\right)
\onefone\left(\frac34-ix;\frac32;iz\right)
\end{align}

\item 
Symmetry: $g_n(-x) = (-1)^n g_n(x)$.

\item Mellin transform representation:
\begin{align}
g_n(x) &=
(-i)^n \frac{\sqrt{\pi}}{2}
\left(
  \Gamma\left(\frac34-ix\right)
  \Gamma\left(\frac34+ix\right)
\right)^{-1}
\int_0^\infty \frac{1}{(u+1)^{3/2}} U_n\left(\frac{u-1}{u+1}\right)\,du
\end{align}

\end{enumerate}

\textbf{Notes.} This list is based on the list of properties of the continuous Hahn polynomials $p_n(x;a,b,c,d)$ given in \cite[pp.~200--204]{koekoek-etal}, except for the Mellin transform representation, which is proved in our Proposition~\ref{prop:gn-mellintrans}.

In the formulas \eqref{eq:gn-explicit1}--\eqref{eq:gn-explicit3}, the first formula is the definition as given in \cite{koekoek-etal}, and formula \eqref{eq:gn-explicit2} is the explicit rewriting of \eqref{eq:gn-explicit1} as a sum. Formula \eqref{eq:gn-explicit3}, which (like \eqref{eq:fn-explicit4} discussed in the previous section) has the benefit of highlighting the odd/even symmetry of $g_n(x)$, seems new, and is proved by evaluating the coefficient of $z^n$ in \eqref{eq:gn-genfun2}
as 
\begin{equation*}
\frac{g_n(x)}{(n+1)!} =
\sum_{k=0}^n [z^k]
\Big( \onefone\left(\frac34+ix;\frac32;-iz\right) \Big)
\times
[z^{n-k}]
\Big( \onefone\left(\frac34-ix;\frac32;iz\right) \Big)
\end{equation*}
and simplifying.

\begin{table}
\caption{The first few polynomials $g_n(x)$}
\begin{equation*}
\begin{array}{r|l}
n & g_n(x) \\[3pt]
\hline
 0 & 1 \\[3pt]
 1 & \frac{8 }{3}x \\[3pt]
 2 & \frac{64 }{15}x^2-\frac{3}{5} \\[3pt]
 3 & \frac{512 }{105}x^3-\frac{272 }{105}x \\[3pt]
 4 & \frac{4096 }{945}x^4-\frac{5312 }{945}x^2+\frac{7}{15} \\[3pt]
 5 & \frac{32768 }{10395}x^5-\frac{83968 }{10395}x^3+\frac{568 }{231}x \\[3pt]
 6 & \frac{262144 }{135135}x^6-\frac{77824 }{9009}x^4+\frac{847232 }{135135}x^2-\frac{77}{195} \\[3pt]
 7 & \frac{2097152 }{2027025}x^7-\frac{14876672 }{2027025}x^5+\frac{20968448 }{2027025}x^3-\frac{527392 }{225225}x \\[3pt]
 8 & \frac{16777216 }{34459425}x^8-\frac{25427968 }{4922775}x^6+\frac{33107968 }{2650725}x^4-\frac{25399936 }{3828825}x^2+\frac{77}{221} 
\end{array}
\end{equation*}
\end{table}

\section{The relationship between the polynomial sequences $f_n$ and $g_n$}

\label{sec:orth-fngn-relation}

The goal of this section is to prove the following pair of identities, which seem new, relating the two orthogonal polynomial families $(f_n)_{n=0}^\infty$ and $(g_n)_{n=0}^\infty$

\begin{prop} The polynomial families $f_n(x)$ and $g_n(x)$ are related by the equations
\begin{align}
\label{eq:gn-intermsof-fn}
g_n(x) & =
\sum_{k=0}^{\lfloor n/2 \rfloor} \frac{2^{n-2k} (n-k)!}{(3/2)_{n-2k} k!} f_{n-2k}(x), \\
\label{eq:fn-intermsof-gn}
f_n(x) & =
\frac{(3/2)_n}{2^n(n+1)!}
\sum_{k=0}^{\lfloor n/2 \rfloor} (-1)^k (n-2k+1)\binom{n+1}{k}
g_{n-2k}(x).
\end{align}
\end{prop}

The proofs relies on two binomial summation identities, given in the next two lemmas.

\begin{lem}
For integers $p,q\ge 0$ we have the summation identity
\begin{equation}
\label{eq:fn-gn-binomialiden1}
\sum_{k=0}^{\lfloor p/2 \rfloor}
\frac{(-1)^k (p+q-k)!}{2^{2k} k! (p-2k)!}
= \frac{1}{2^p} q! \binom{p+2q+1}{p}.
\end{equation}
\end{lem}

\begin{proof}
Consider, for fixed $q\ge 0$, the generating function in an indeterminate $x$ of the sequence of numbers (indexed by the parameter $p\ge 0$) on the left-hand side of \eqref{eq:fn-gn-binomialiden1}. This generating function can be evaluated as
\begin{align*}
\sum_{p\ge 0} &
\left( \sum_{k=0}^{\lfloor p/2 \rfloor}
\frac{(-1)^k (p+q-k)!}{2^{2k} k! (p-2k)!} \right) x^p
\\ & =
\sum_{p\ge 0}
\left( \sum_k \left(-\frac14\right)^k \binom{p-k}{k} (p-k+1)\cdots (p-k+q)
\right) x^p
\\ & 
=
\sum_{m\ge 0}
\sum_k \left(-\frac14\right)^k \binom{m}{k} (m+1)\cdots (m+q)
x^{m+k}
=
\sum_{m\ge 0} (m+1)\cdots (m+q) x^m \left(\sum_k \binom{m}{k} \left(-\frac{x}{4}\right)^k\right)
\\ & =
\sum_{m\ge 0} (m+1)\cdots (m+q) \bigg(x \left(1-\frac{x}{4}\right)\bigg)^m
= 
\frac{d^q}{d y^q}_{\raisebox{2pt}{$\big|$} y=x(1-x/4)} \left( \sum_{m=0}^\infty y^m \right)
\\ & = \frac{d^q}{d y^q}_{\raisebox{2pt}{$\big|$} y=x(1-x/4)} \left( \frac{1}{1-y} \right)
 = 
\frac{q!}{(1-y)^{q+1}}_{\raisebox{2pt}{$\big|$} y=x(1-x/4)}
\\ & 
= 
\frac{q!}{\left(1-x\left(1-\frac{x}{4}\right)\right)^{q+1}}
= 
\frac{q!}{\left(1-\frac{x}{2}\right)^{2q+2}}
=
\sum_{p=0}^\infty \frac{q!}{2^p} \binom{p+2q+1}{p} x^p,
\end{align*}
which is the generating function for the sequence on the right-hand side of \eqref{eq:fn-gn-binomialiden1}.
\end{proof}

\begin{lem}
The summation identity
\begin{equation}
\label{eq:fn-gn-binomialiden2}
\sum_{k=0}^N (N-2k)\binom{N}{k} \binom{N+m-2k}{2m+1} 
=
N \binom{N-1}{m} 2^{N-m}
\end{equation}
holds for integers $N,m\ge0$
\end{lem}

\begin{proof}
Denote 
\begin{equation*}
F_m(N,k) = \frac{(N-2k)\binom{N}{k} \binom{N+m-2k}{2m+1}}{N \binom{N-1}{m} 2^{N-m}},
\end{equation*}
so that the identity to prove becomes the statement that $\sum_{k=0}^N F_m(N,k)=1$.
This claim in turn follows by applying the method of Wilf-Zeilberger pairs \cite[Ch.~7]{aequalsb}, \cite{wilf-zeilberger} to the rational certificate function (in which $m$ is regarded as a parameter)
\begin{equation*}
R_m(N,k) = 
\frac{k (m+N+1-2k) (m+N+2-2k)}{2(N-2k)(N+1-k)(N-m-2k)}.
\end{equation*}
The certificate was found using the Mathematica package \texttt{fastZeil} \cite{paule-schorn1, paule-schorn2}, a software implementation of Zeilberger's algorithm.
\end{proof}

\begin{proof}[Proof of~\eqref{eq:gn-intermsof-fn}]
An immediate consequence of \eqref{eq:fn-gn-binomialiden1} is the identity
\begin{equation}
\label{eq:fngn-binomialiden2-variant}
\sum_{k=0}^{\lfloor (n-m)/2 \rfloor}
\frac{(-1)^k (n-k)!}{2^{2k} k! (3/2)_{n-2k}} \binom{n-2k+\frac12}{n-2k-m}
=
\frac{(n+m+1)!}{2^{n+m} (n-m)! (3/2)_m^2},
\end{equation}
which holds for integers $n\ge m\ge 0$---indeed, this relation reduces to \eqref{eq:fn-gn-binomialiden1} after a short simplification on taking $p=n-m$, $q=m$ and using the facts that $(3/2)_m = \frac{(2m+2)!}{2^{2m+1} (m+1)!}$ and $\binom{n-2k+1/2}{n-2k-m} = \frac{(3/2)_{n-2k}}{(n-2k-m)! (3/2)_m}$. Now \eqref{eq:fngn-binomialiden2-variant} is the key to proving \eqref{eq:gn-intermsof-fn}: making use of the explicit formulas \eqref{eq:fn-explicit2} and \eqref{eq:gn-explicit2} for $f_n(x)$ and $g_n(x)$, respectively, we write
\begin{align*}
\sum_{k=0}^{\lfloor n/2 \rfloor} \frac{2^{n-2k} (n-k)!}{(3/2)_{n-2k} k!} f_{n-2k}(x)
& =
\sum_{k=0}^{\lfloor n/2 \rfloor} \frac{2^{n-2k} (n-k)!}{(3/2)_{n-2k} k!} \left(
(-i)^{n-2k} \sum_{m=0}^{n-2k} 2^m \binom{n-2k+\frac12}{n-2k-m}
\binom{-\frac34+ix}{m} 
\right)
\\ & = 
(-i)^n \sum_{m=0}^n
\left(
2^m \sum_{k=0}^{\lfloor (n-m)/2 \rfloor}
\frac{(-1)^k 2^{n-2k} (n-k)!}{(3/2)_{n-2k}k!} \binom{n-2k+\frac12}{n-2k-m}
\right) \binom{-\frac34+ix}{m}
\\ & =
(-i)^n \sum_{m=0}^n \frac{(n+m+1)!}{(n-m)! (3/2)_m^2} \binom{-\frac34+ix}{m} = g_n(x),
\end{align*}
giving the result.
\end{proof}

\begin{proof}[Proof of~\eqref{eq:fn-intermsof-gn}]
Using \eqref{eq:fn-gn-binomialiden2}, we can deduce the slightly more messy identity
\begin{equation}
\label{eq:fngn-binomialiden2-newvariant}
\frac{(3/2)_n}{2^n (3/2)_m^2} \sum_{k=0}^{\lfloor (n-m)/2 \rfloor}
\frac{(n-2k+1)(n-2k+m+1)!}{k!(n-k+1)!(n-2k-m)!}
= 2^m \binom{n+1/2}{n-m} \qquad (n\ge m\ge 0).
\end{equation}
The way to see this is to first massage the left-hand side of \eqref{eq:fn-gn-binomialiden2} a bit by rewriting it as
\begin{align*}
\sum_{k=0}^N (N-2k)\binom{N}{k} \binom{N+m-2k}{2m+1} 
& =
2 \sum_{k=0}^{\lfloor N/2\rfloor} (N-2k)\binom{N}{k} \binom{N+m-2k}{2m+1} 
\\
& =
2 \sum_{k=0}^{\lfloor \frac{N-m}{2} \rfloor} (N-2k)\binom{N}{k} \binom{N+m-2k}{2m+1}, 
\end{align*}
where the first equality follows from the symmetry of the summand under the relabeling $k\mapsto N-k$, and the second equality follows on noticing that the summands actually vanish for values of $k$ for which $\frac{N-m}{2} < k < \frac{N+m}{2}$. Thus, we obtain another variant of \eqref{eq:fn-gn-binomialiden2}, namely
\begin{equation}
\label{eq:fngn-binomialiden2-anothervariant2}
\sum_{k=0}^{\lfloor \frac{N-m}{2}\rfloor} (N-2k)\binom{N}{k} \binom{N+m-2k}{2m+1} 
=
N \binom{N-1}{m} 2^{N-m+1}.
\end{equation}
We leave to the reader to verify (using similar simple substitutions as in the proof of \eqref{eq:gn-intermsof-fn} above) that setting $N=n+1$ in this new identity gives a relation that is equivalent to \eqref{eq:fngn-binomialiden2-newvariant}.

Finally, from \eqref{eq:fngn-binomialiden2-newvariant} we can prove the relation \eqref{eq:fn-intermsof-gn} in a manner analogous to the proof of \eqref{eq:gn-intermsof-fn}, again making use of the expansions \eqref{eq:fn-explicit2}, \eqref{eq:gn-explicit2} but working in the opposite direction. We have
\begin{align*}
& \frac{(3/2)_n}{2^n (n+1)!} \sum_{k=0}^{\lfloor n/2 \rfloor} (-1)^k (n-2k+1) \binom{n+1}{k} g_{n-2k}(x)
\\ & =
\frac{(3/2)_n}{2^n (n+1)!} \sum_{k=0}^{\lfloor n/2 \rfloor} 
(-1)^k (n-2k+1) \binom{n+1}{k}
(-i)^{n-2k} \left(\sum_{m=0}^{n-2k} 
\frac{(n-2k+m+1)!}{(n-2k-m)!(3/2)_m^2} \binom{-\frac34+ix}{m}
\right)
\\ & =
(-i)^n
\sum_{m=0}^n
\left(
\frac{(3/2)_n}{2^n} \sum_{k=0}^{\lfloor (n-m)/2 \rfloor}
\frac{(n-2k+1)(n-2k+m+1)!}{k!(n-k+1)!(n-2k-m)!(3/2)_m^2}
\right)
\binom{-\frac34+ix}{m}
= f_n(x),
\end{align*}
as claimed.
\end{proof}

\chapter{Summary of main formulas}

\noindent
\textbf{Series expansions for the Riemann xi function}
\begin{align*}
\Xi(t) & = \sum_{n=0}^\infty (-1)^n a_{2n} t^{2n}
& \textrm{(p.~\pageref{eq:riemannxi-taylor})}
\\
\Xi(t) & = \sum_{n=0}^\infty (-1)^n b_{2n} H_{2n}(t)
& \textrm{(p.~\pageref{eq:hermite-expansion})}
\\
\Xi(t) & = \sum_{n=0}^\infty (-1)^n c_{2n} f_{2n}\left(\frac{t}{2}\right)
& \textrm{(p.~\pageref{eq:fn-expansion})}
\\
\Xi(t) & = \sum_{n=0}^\infty (-1)^n d_{2n} g_{2n}\left(\frac{t}{2}\right)
& \textrm{(p.~\pageref{eq:gn-expansion})}
\end{align*}

\noindent
\textbf{Series expansions for related functions}
\begin{align*}
A(r) &= \sum_{n=0}^\infty c_{2n} G_{2n}^{(3)}(r)
& \textrm{(p.~\pageref{eq:aofr-selftrans-expansion})}
\\[10pt]
\tilde{\nu}(t) &=
\frac{1}{2\sqrt{2}} \sum_{n=0}^\infty 
\frac{(3/2)_{2n}}{(2n)!} c_{2n} t^{2n}
& \textrm{(p.~\pageref{eq:nutilde-taylor})}
\\
\tilde{\nu}(t) &= \frac{1}{2\sqrt{2}}\sum_{n=0}^\infty d_{2n} U_{2n}(t)
& \textrm{(p.~\pageref{eq:nutilde-chebyshev-exp})}
\end{align*}

\noindent
\textbf{Formulas for the coefficients}
\begin{align*}
a_{2n} & = \frac{1}{2^{2n}(2n)!} \int_0^\infty \omega(x)x^{-3/4}(\log x)^{2n}\,dx
& \textrm{(p.~\pageref{eq:riemannxi-taylorcoeff-int})}
\\
& = \frac{1}{(2n)!} \int_{-\infty}^\infty x^{2n} \Phi(x)\,dx
& \textrm{(p.~\pageref{eq:a2n-qn-rnprime})}
\\[13pt]
b_{2n} & = \frac{1}{2^{2n} (2n)!} \int_{-\infty}^\infty x^{2n} e^{-\frac{x^2}{4}} \Phi(x)\,dx
& \textrm{(p.~\pageref{eq:hermite-coeffs-def})}
\\ & = 
\frac{(-1)^n}{\sqrt{\pi}2^{2n}(2n)!}
\int_{-\infty}^\infty \Xi(t) e^{-t^2} H_{2n}(t)\,dt
& \textrm{(p.~\pageref{eq:hermite-coeff-innerproduct})}
\end{align*}
\begin{align*}
c_{2n} & =
2\sqrt{2}\int_0^\infty \frac{\omega(x)}{(x+1)^{3/2}} \left(\frac{x-1}{x+1}\right)^{2n}\,dx
& \textrm{(p.~\pageref{eq:def-fn-coeffs})}
\\ & =
(-1)^n \frac{\sqrt{2} \, (2n)!}{\pi^{3/2} (3/2)_{2n}}
\int_{-\infty}^\infty \Xi(t) f_{2n}\left(\frac{t}{2}\right)\left|\Gamma\left(\frac34+\frac{it}{2}\right)\right|^2\,dt
& \textrm{(p.~\pageref{eq:fn-coeff-innerproduct})}
\\ & =
\frac{8\sqrt{2} \pi (2n)!}{(3/2)_{2n}}
\int_0^\infty A(r) r^2 G_{2n}^{(3)}(r)\,dr
& \textrm{(p.~\pageref{eq:fn-coeffs-radialform})}
\\ & =
2\int_{-1}^1 t^{2n} \tilde{\omega}(t) \,dt
& \textrm{(p.~\pageref{eq:fn-coeffs-asmoments})}
\\[4pt]
d_{2n} & = 
\frac{(3/2)_{2n}}{2^{2n-3/2} (2n)!}\int_0^\infty \frac{\omega(x)}{(x+1)^{3/2}} \left(\frac{x-1}{x+1}\right)^{2n} 
\gausshyper\left(n+\frac34,n+\frac54; 2n+2; \left(\frac{x-1}{x+1}\right)^2 \right)
 \,dx
& \textrm{(p.~\pageref{eq:def-gn-coeffs})}
\\ & =
\frac{8}{\pi^3}(-1)^n \int_{-\infty}^\infty \Xi(t) g_{2n}\left(\frac{t}{2}\right) \left|\Gamma\left(\frac34+\frac{it}{2}\right)\right|^4\,dt
& \textrm{(p.~\pageref{eq:gn-coeff-innerproduct})}
\\ & =
\frac{n+1}{2^n} \sum_{m=0}^\infty \frac{(3/2)_{n+2m}}{4^m m!(n+m+1)!} c_{n+2m}
& \textrm{(p.~\pageref{eq:dn-cn-expansion})}
\\ & =
\frac{4\sqrt{2}}{\pi} \int_{-1}^1 \tilde{\nu}(t) U_{2n}(t)\sqrt{1-t^2}\,dt
& \textrm{(p.~\pageref{eq:dn-chebyshev-int})}
\\ & =
\frac{16}{\pi} (2n+1)
\int_1^\infty \int_0^1 
\frac{\omega(x)}{((x+1)^2-t(x-1)^2)^{3/4}}
\left(\frac{t}{1-t}\right)^{1/4}
\left(
\frac{t(1-t)(x-1)^2}{(x+1)^2-t(x-1)^2}
\right)^n
\,dt\,dx
& \textrm{(p.~\pageref{eq:d2n-simplerep-doubleint})}
\end{align*}

\medskip
\noindent
\textbf{Asymptotic formulas for the coefficients}

\begin{align*}
a_{2n} & =
\left(1+O\left(\frac{\loglog n}{\log n}\right)\right)
\frac{\pi^{1/4}}{2^{2n-\frac52} (2n)!} \left(\frac{2n}{\log (2n)}\right)^{7/4}
\\ & \qquad\qquad \times 
\exp\left[
2n\left(\log \left(\frac{2n}{\pi}\right) - W\left(\frac{2n}{\pi}\right) - \frac{1}{W\left(\frac{2n}{\pi}\right)} \right)
\right]
& \textrm{(p.~\pageref{eq:taylorxi-coeff-asym})}
\\[4pt]
b_{2n} & =
\left(1+O\left(\frac{\loglog n}{\log n}\right)\right)
\frac{\pi^{1/4}}{2^{4n-\frac52} (2n)!} \left(\frac{2n}{\log (2n)}\right)^{7/4}
\\ \nonumber &\qquad \qquad \times \exp\left[
2n\left(\log \left(\frac{2n}{\pi}\right) - W\left(\frac{2n}{\pi}\right) - \frac{1}{W\left(\frac{2n}{\pi}\right)} \right)
-\frac{1}{16} W\left(\frac{2n}{\pi}\right)^2\right]
& \textrm{(p.~\pageref{eq:hermite-coeff-asym})}
\\[4pt]
c_{2n} & = \left(1+O\left(n^{-1/10}\right)\right) 16 \sqrt{2} \pi^{3/2}\, \sqrt{n} \,\exp\left(-4\sqrt{\pi n}\right)
& \textrm{(p.~\pageref{eq:fn-coeff-asym})}
\\[4pt]
d_{2n} & = \left(1+O\left(n^{-1/10}\right)\right) 
\left(\frac{128\times 2^{1/3} \pi^{2/3}e^{-2\pi/3}}{\sqrt{3}}\right)
n^{4/3} 
\exp\left(-3(4\pi)^{1/3} n^{2/3}\right)
& \textrm{(p.~\pageref{eq:gn-coeff-asym})}
\end{align*}

\backmatter
\bibliographystyle{plain}
\bibliography{riemannxi}

%

\end{document}